\documentclass[10pt]{amsart}
\usepackage{mathrsfs}
\usepackage{amssymb}
\usepackage{amsmath, bm}

\usepackage{titletoc}
\usepackage{amscd}
\usepackage{amsmath, amssymb} 
\usepackage{amsthm}
\usepackage{amsfonts}
\usepackage[colorlinks,linkcolor=blue,citecolor=blue, pdfstartview=FitH]{hyperref}

\usepackage[all,cmtip]{xy}

\usepackage{stmaryrd}

\usepackage{upgreek}

\usepackage{tikz-cd}

\usepackage{extarrows} 

\numberwithin{equation}{section}

\theoremstyle{plain}
\newtheorem{thm}{Theorem}[section]

\newtheorem{lem}[thm]{Lemma}
\newtheorem{corollary}[thm]{Corollary}
\newtheorem{prop}[thm]{Proposition}

\newtheorem{thmx}{Theorem}

\theoremstyle{definition}
\newtheorem{claim}[thm]{Claim}

\theoremstyle{remark}
\newtheorem{rmk}[thm]{Remark}

\usepackage{fancyhdr}

\newcommand{\bA}{{\mathbf C}}

\newcommand{\C}{{\mathbb C}}
\newcommand{\sC}{{\mathcal C}}

\newcommand{\F}{{\mathbb F}}
\newcommand{\bF}{{\textbf F}}

\newcommand{\sH}{{\mathcal H}}
\newcommand{\I}{{\mathrm I}}

\newcommand{\K}{{\mathbb K}}

\newcommand{\N}{{\mathbb N}}

\newcommand{\sO}{{\mathcal O}}

\newcommand{\sP}{{\mathcal P}}
 
\newcommand{\rQ}{{\mathrm Q}}

\newcommand{\sT}{{\mathcal T}}

\newcommand{\fU}{{\mathfrak U}}

\newcommand{\sX}{{\mathfrak X}}

\newcommand{\Z}{{\mathbb Z}}

\newcommand{\id}{\mathrm{id}} 
 
\newcommand{\im}{\mathrm{im}} 
\newcommand{\pr}{\mathrm{pr}}

\newcommand{\xs}{{\ \xrightarrow{\sim} \ }} 

\newcommand{\Hom}{\mathrm{Hom}}
\newcommand{\End}{\mathrm{End}}
\newcommand{\Aut}{\mathrm{Aut}}
\newcommand{\Ext}{\mathrm{Ext}}

\newcommand{\op}{{\mathrm{op}}}

\newcommand{\Frac}{{\mathrm{Frac}}}

\newcommand{\Mod}{\text{-}\mathrm{Mod}}
\renewcommand{\mod}{\text{-}\mathrm{mod}}

\newcommand{\Spec}{\mathrm{Spec}}

\newcommand{\rk}{\text{rk}} 
 
\newcommand{\pt}{\mathrm{pt}}

\newcommand{\Gm}{\mathbb{G}_{\mathrm{m}}}

\newcommand{\Lie}{\mathrm{Lie}} 

\newcommand{\g}{{\mathfrak{g}}}
\renewcommand{\b}{{\mathfrak{b}}}
\renewcommand{\t}{{\mathfrak{t}}}
\newcommand{\n}{{\mathfrak{n}}}

\newcommand{\Gr}{\mathcal{G}\mathfrak{r}} 
\newcommand{\Fl}{\mathcal{F}{l}}
\newcommand{\hb}{\mathrm{hb}}

\newcommand{\Fun}{\mathrm{Fun}} 
\newcommand{\GKM}{\mathrm{GKM}} 
\newcommand{\sGKM}{\mathrm{sGKM}} 
\newcommand{\af}{\mathrm{af}} 
\newcommand{\re}{\mathrm{re}} 

\renewcommand{\sc}{\mathrm{sc}} 
 
\newcommand{\fMod}{\text{-}\mathrm{pMod}} 
\newcommand{\Modf}{\mathrm{pMod}\text{-}} 
\newcommand{\rMod}{\mathrm{Mod}\text{-}} 
\newcommand{\ex}{\mathrm{ex}} 
 
\newcommand{\val}{\mathrm{val}} 

\newcommand{\p}{\mathfrak{p}} 
\newcommand{\bb}{\mathbf{b}}

\newcommand{\rI}{\mathrm{I}} 

\renewcommand{\k}{\C} 

\newcommand{\Par}{\mathrm{Par}} 
\newcommand{\Fr}{\mathrm{Fr}}

\newcommand{\Stab}{\mathrm{Stab}} 
 
\newcommand{\adj}{\mathrm{adj}} 
\newcommand{\hgt}{\mathrm{ht}} 
\newcommand{\diag}{\mathrm{diag}} 
 
\newcommand{\ch}{\mathrm{ch}} 
\newcommand{\ev}{\mathrm{ev}} 

\newcommand{\bL}{\mathbf{\Lambda}}

\newcommand{\qbno}[2]{\begin{bmatrix} #1 \\ #2 \end{bmatrix}}

\begin{document}
\title{On the category $\mathcal{O}$ of a hybrid quantum group} 

\makeatletter
\let\MakeUppercase\relax
\makeatother

\author{Quan Situ} 
\address{Yau Mathematical Sciences Center, Department of Mathematical Sciences\\
Tsinghua University\\ 
Beijing 100084, P.~R.~China}
%\address{Universit\'{e} Clermont Auvergne, CNRS, LMBP\\
%F-63000 Clermont-Ferrand, France} 
\email{stq19@tsinghua.org.cn, Quan.SITU@uca.fr}
%\date{}
\begin{abstract}
We study the representation theory of a hybrid quantum group at root of unity $\zeta$ introduced by Gaitsgory. 
After discussing some basic properties of its category $\mathcal{O}$, we study deformations of the category $\mathcal{O}$. 
For subgeneric deformations, we construct the endomorphism algebra of big projective object and compute it explicitly. 
Our main result is an algebra isomorphism between the center of deformed category $\mathcal{O}$ and the equivariant cohomology of $\zeta$-fixed locus on the affine Grassmannian attached to the Langlands dual group. 
\end{abstract}

\maketitle
\setcounter{tocdepth}{1} \tableofcontents 

\section{Introduction} 
\subsection{Hybrid quantum group and the category $\sO_q$} 
Let $G$ be a complex connected and simply-connected semisimple algebraic group, and let $B\supset T$ be a Borel subgroup and a Cartan subgroup. 
Bernstein--Gelfand--Gelfand \cite{BGG76} defined the category $\sO$ for $\g=\Lie(G)$, which plays an important role in the study of infinite dimensional representations of $\g$. 
Its analogue for Lusztig's quantum group was studied by Anderson--Mazorchuk \cite{AM15}. 

Recently, a new version of quantum group called the \textit{hybrid quantum group} has been introduced by Gaitsgory \cite{Gai18} (it was called the mixed quantum group in the \textit{loc. cit.}), in view of extending the Kazhdan--Lusztig equivalence. 
It also appeared in \cite{BBASV} for the study of the center of the small quantum group. 
Let $\mathscr{U}_q$ be the quantum group for $G$ over $\C(q)$. 
There are two remarkable $\Z[q^{\pm 1}]$-forms of $\mathscr{U}_q$: 
\textit{De Concini--Kac quantum group} $\fU_q$ generated by Chevalley generators and \textit{Lusztig's quantum group} $U_q$ generated by their divided powers. 
The hybrid quantum group $U^\hb_q$ is a $\Z[q^{\pm 1}]$-form of $\mathscr{U}_q$ with triangular decomposition 
\footnote{In the present paper we will use the definition $U^\hb_q=\fU^-_q\otimes \fU^0_q\otimes U^+_q$, for the convenience of deforming category $\sO$. 
In fact these two definitions lead to the same category $\sO$ (see \cite[Lem 3.1]{BBASV}).} 
$$U^\hb_q=\fU^-_q\otimes U^0_q \otimes U^+_q, $$ 
where only non-negative half contains divided powers. 

The category $\sO$ for $U^\hb_q$ (denoted by $\sO_q$) is the category of finitely generated $U^\hb_q$-modules that are graded by $X^*(T)$ for the action of $U^0_q$ and locally unipotent for the action of $U^+_q$. 
It can be viewed as a quantum analogue of the BGG category $\sO$ for $\g$ (with integral weights). 
Indeed, note that the latter can be interpreted as the category 
$$\g\mod^B$$ 
of $B$-equivariant finitely generated $\g$-modules on which the action from $\b=\Lie(B)$ by restriction of $\g$-action coincides with the one by derivation of $B$-action. 
Similarly, we can interprete $\sO_q$ as the category 
$$\fU_q\mod^{U^{\geq}_q}$$ 
of finitely generated $\fU_q$-modules equipped with an equivariant and integrable $U^{\geq}_q$-module structure that is compatible with the natural map $\fU^{\geq}_q\rightarrow U^{\geq}_q$. 
Hence $\sO_q$ is a reasonable quantum analogue of $\g\mod^B$. 

If the parameter $q$ is specialized to a transcendental element in $\C^\times$, the category $\sO_q$ is equivalent to the classical BGG category $\sO$ (see e.g. \cite{AM15}). 
In this paper we study the category $\sO_\zeta$, where $\zeta\in \C$ is a primitive $l$-th root of unity ($l$ is an odd integer satisfying mild assumptions specified in \textsection\ref{subsect 2.1.2'}). 

\subsection{Basic properties of $\sO_\zeta$} 
We firstly establish some basic properties for $\sO_\zeta$, including its linkage principle, the construction of projective objects and the BGG reciprocity. 
Here are more details. 
Let $M(\lambda)$ be the Verma module of $U^\hb_\zeta$ of highest weight $\lambda\in X^*(T)$, and let $E(\lambda)$ be its unique simple quotient. 
We show that $E(\lambda)$ coincides with the $X^*(T)$-graded simple module of the small quantum group. 
By studying the Jantzen filtration for Verma modules, we obtain the following linkage principle and the block decomposition. 

\begin{prop}[=Proposition \ref{prop LP}] 
The multiplicity 
$$[M(\lambda):E(\mu)]\neq 0 \quad \text{if and only if}\quad \mu \leq s_{\beta_1}\bullet\mu \leq \cdots \leq s_{\beta_k}\cdots s_{\beta_1}\bullet\mu=\lambda$$
for some reflections $s_{\beta_i}$ in the $l$-dilated affine Weyl group $W_{l,\af}$. 
Consequently, we have a block decomposition 
\begin{equation}\label{equ 0.2} 
\sO_{\zeta}= \bigoplus_{\omega\in X^*(T)/(W_{l,\af},\bullet)} \sO^{\omega}_{\zeta}, 
\end{equation}
such that the Verma module $M(\lambda)$ is contained in $\sO^{\omega}_{\zeta}$ if and only if $W_{l,\af}\bullet \lambda=\omega$. 
\end{prop} 

The category $\sO_\zeta$ does not contain enough projective modules. 
In fact, projective objects only exist in the truncated subcategory $\sO^{\leq \nu}_\zeta$ for each $\nu\in X^*(T)$, namely the category of the modules in $\sO_\zeta$ whose $\mu$-weight subspace is nonzero only if $\mu\leq \nu$. 
Let $Q(\lambda)^{\leq \nu}$ be the projective cover of $E(\lambda)$ in $\sO^{\leq \nu}_\zeta$ (for $\lambda\leq \nu$). 
We establish a BGG reciprocity for $\sO^{\leq \nu}_\zeta$. 
\begin{prop}[=Proposition \ref{prop BGG}] 
The module $Q(\lambda)^{\leq \nu}$ admits a finite filtration with sub-quotient factors given by Verma modules, and the multiplicity of $M(\mu)$ in $Q(\lambda)^{\leq \nu}$ is 
\begin{equation}\label{equ 0.3} 
\big(Q(\lambda)^{\leq \nu}:M(\mu)\big)=[M(\mu):E(\lambda)]. 
\end{equation} 
\end{prop}

\subsection{Deformations of $\sO_\zeta$ and endomorphism of big projective} 
Our next goal is to study the deformation of $\sO_\zeta$. 
For the classical category $\sO$ of $\g$, deformation techniques have been used extensively to study basic structures, e.g. Soergel \cite{Soe90}, Fiebig \cite{Fie03} and Stroppel \cite{Strop09}. 
Let $S=H_{\check{T}}^\bullet(\pt)_{\widehat{0}}$ be the completion of $H_{\check{T}}^\bullet(\pt)$ at the augmentation ideal (where $\check{T}$ is the dual torus of $T$). 
The deformation category for $\sO_\zeta$ is an $S$-linear category $\sO_{\zeta,S}$ whose specialization to the residue field of $S$ is $\sO_\zeta$. 
The category $\sO_{\zeta,S}$ is generically semisimple. 
Subgenerically, namely when localizing at the prime ideal of $S$ generated by a coroot $\check{\alpha}$ (with localized ring $S_\alpha$), the category $\sO_{\zeta,S_\alpha}$ is equivalent to a direct sum of the corresponding category for the rank $1$ subgroup of $G$ associated with $\alpha$. 

We study $\sO_{\zeta,S_\alpha}$ in detail. 
It is a classical result of Soergel \cite{Soe90} that for each block of the BGG categeory $\sO$, the endomorphism algebra of the big projective module (=projective cover of the unique simple module of anti-dominant highest weight) is isomorphic to the center of the block. 
We provide an analogous result for $\sO_{\zeta,S_\alpha}$, namely we construct the ``endomorphism algebra of big projective" for each block of $\sO_{\zeta,S_\alpha}$ and show the isomorphism to the center, by explicit computations. 
Here are some subtleties in our case. 
Firstly, $\sO_{\zeta,S_\alpha}$ does not admit projective object. 
As explained above, this problem can be overcome by considering truncated subcategories. 
Another subtlety is that truncated block does not admit a unique (or even the most) anti-dominiant weight. 
To remedy this, we consider a projective limit involving the endomorphism algebras of infinitely many projective modules. 

To give an example, we assume $G=SL_2$ and let $\alpha$ be the unique positive root (so $S=S_\alpha$). 
We outline our construction for the Steinberg block 
%$\sO^{-\frac{1}{2}\alpha}_{\zeta,S}$ 
of $\sO_{\zeta,S}$ for $SL_2$, i.e. the block containing Verma modules $M(\bm{n}_\alpha)_S$ of highest weight $\bm{n}_\alpha=(-\frac{1}{2}+ln)\alpha$, $n\in \Z$. 
%Firstly, $\sO_{\zeta,S}$ does not admit projective object. 
%As explained above, this problem can be overcome by considering truncated subcategories. 
Let $Q(\bm{n}_\alpha)^{\leq \bm{m}_\alpha}_{S}$ be the projective cover of $M(\bm{n}_\alpha)_S$ in the truncation $\sO^{\leq \bm{m}_\alpha}_{\zeta,S}$. 
One can show that its Verma multiplicity is given by 
$$\big( Q(\bm{n}_\alpha)^{\leq \bm{m}_\alpha}_{S}: M(\bm{j}_\alpha)_S \big) = 
\begin{cases}
	1, & \text{$n\leq j\leq m$} \\ 
	0, & \text{otherwise}. 
\end{cases}$$ 
We construct natural projections 
$$Q(\bm{n}_\alpha)^{\leq \bm{m+1}_\alpha}_{S}\twoheadrightarrow Q(\bm{n}_\alpha)^{\leq \bm{m}_\alpha}_{S}, \quad \forall n\leq m,$$ 
%Another difficulty is that the truncated block does not admit the ``most anti-dominiant" weight (we always have $\bm{n-1}_\alpha<\bm{n}_\alpha$). 
%So we need to consider a projective limit involving the endomorphism algebras of infinitely many projective modules. 
and inclusions 
$$Q(\bm{n}_\alpha)^{\leq \bm{m}_\alpha}_{S}\hookrightarrow Q(\bm{n-1}_\alpha)^{\leq \bm{m}_\alpha}_{S}, \quad \forall n\leq m,$$ 
which are compatible with their filtrations with Verma sub-quotient factors. 
Therefore $Q(\bm{n'}_\alpha)^{\leq \bm{m'}_\alpha}_{S}$ appears as a sub-quotient module of $Q(\bm{n}_\alpha)^{\leq \bm{m}_\alpha}_{S}$, provided that $n\leq n' \leq m' \leq m$. 
We show that each endomorphism of $Q(\bm{n}_\alpha)^{\leq \bm{m}_\alpha}_{S}$ uniquely induces an endomorphism of $Q(\bm{n'}_\alpha)^{\leq \bm{m'}_\alpha}_{S}$, so that we can define the endomorphism algebra of the big projective object as 
$$\End\big(Q(\bm{-\infty}_{\alpha} )^{\leq \bm{\infty}_{\alpha}}_{S}\big):= 
\varprojlim_{n \leq m} \End\big(Q(\bm{n}_{\alpha})^{\leq \bm{m}_{\alpha}}_{S}\big).$$

\subsection{Center of $\sO_{\zeta,S}$} 
As an application, we establish an isomorphism between the center $Z(\sO_{\zeta,S})$ of $\sO_{\zeta,S}$ and equivariant cohomology of $\zeta$-fixed locus on the affine Grassmannian. 
The relation between center and cohomology was firstly discovered and was predicted to be an isomorphism by Bezrukavnikov, Boixeda-Alvarez, Shan and Vasserot \cite{BBASV} in their study of geometric realization of center of small quantum groups. 

Let $\check{G}$ be the Langlands dual group of $G$, and $\Gr=\check{G}\big(\C(\!(t)\!)\big)/\check{G}\big(\C[[t]]\big)$ the affine Grassmannian for $\check{G}$. 
There is a $\Gm$-action on $\Gr$ by loop rotations. 
We consider the $\zeta$-fixed locus $\Gr^\zeta$. 
By \cite{RW22}, there is a decomposition 
\begin{equation}\label{equ 0.-1} 
\Gr^\zeta=\bigsqcup\limits_{\omega\in X^*(T)/W_{l,\af}} \Fl^{\omega,\circ}, 
\end{equation}
where $\Fl^{\omega,\circ}$ is partial affine flag varieties of type $\omega$ of the loop group $\check{G}_\sc\big(\C(\!(t^l)\!)\big)$ for the simply-connected cover $\check{G}_\sc$ for $\check{G}$. 
In \cite{BBASV} the authors constructed an $S$-algebra embedding 
$$\bb:\ H_{\check{T}}^\bullet(\Gr^\zeta)\otimes_{H_{\check{T}}^\bullet(\pt)}S \rightarrow Z(\sO_{\zeta,S}),$$ 
such that under the map $\bb$, the decomposition (\ref{equ 0.-1}) is compatible with the following block decomposition as a lifting of (\ref{equ 0.2}) 
$$\sO_{\zeta,S}=\bigoplus\limits_{\omega\in X^*(T)/(W_{l,\af},\bullet)} \sO^{\omega}_{\zeta,S}.$$ 
We prove that $\bb$ is an isomorphism. 
More precisely, 
\begin{thmx}[=Theorem \ref{thm 3.11}]\label{thm A} 
The map $\bb$ induces an isomorphism of $S$-algebras 
$$\bb:\ H_{\check{T}}^\bullet(\Gr^\zeta)^\wedge_{S} \xs Z(\sO_{\zeta,S}),$$ 
where $H_{\check{T}}^\bullet(\Gr^\zeta)^\wedge_{S}$ is a completion of $H_{\check{T}}^\bullet(\Gr^\zeta)\otimes_{H_{\check{T}}^\bullet(\pt)}S$. 
Block-wisely, it gives an isomorphism 
$$\bb:\ H_{\check{T}}^\bullet(\Fl^{\omega,\circ})^\wedge_{S} \xs Z(\sO^{\omega}_{\zeta,S}),$$ 
for each parahoric/singular type $\omega$. 
\end{thmx} 

To prove Theorem \ref{thm A}, we firstly recall that $\sO_{\zeta,\K}$ is semisimple ($\K$ is the fraction field of $S$), and that $X^*(T)$ parametrizes the simple(=Verma) modules in $\sO_{\zeta,\K}$ by the highest weights as well as the $\check{T}$-fixed points $(\Gr^\zeta)^{\check{T}}$. 
It yields natural identifications 
$$Z(\sO_{\zeta,\K})=\prod_{(\Gr^\zeta)^{\check{T}}} \K\ =\quad
\text{a completion of }H^\bullet_{\check{T}}\big((\Gr^\zeta)^{\check{T}}\big)\otimes_{H^\bullet_{\check{T}}(\pt)} \K .$$ 
Then for each positive root $\alpha$, we consider the codimensional $1$ subtorus $\check{T}_\alpha=(\ker\check{\alpha})^\circ\subset \check{T}$ (where $\circ$ represents the neutral component), and consider the $\check{T}_\alpha$-fixed locus on $\Gr^\zeta$. 
Using the construction of the endomorphism algebra of big projective object explained in the previous section, we show an isomorphism 
$$Z(\sO_{\zeta,S_\alpha})\ \simeq\quad \text{a completion of } H^\bullet_{\check{T}}\big((\Gr^\zeta)^{\check{T}_\alpha}\big)\otimes_{H^\bullet_{\check{T}}(\pt)}S_\alpha.$$  
On one hand, the intersection of $Z(\sO_{\zeta,S_\alpha})$ for positive root $\alpha$ inside $Z(\sO_{\zeta,\K})$ recovers $Z(\sO_{\zeta,S})$. 
On the other hand, by applying the GKM(=Goresky--Kottwitz--MacPherson) theory to $\Gr^\zeta$, the intersection of $H^\bullet_{\check{T}}\big((\Gr^\zeta)^{\check{T}_\alpha}\big)$ inside $H^\bullet_{\check{T}}\big((\Gr^\zeta)^{\check{T}}\big)$ gives a description of $H^\bullet_{\check{T}}(\Gr^\zeta)$. 
So we match up $Z(\sO_{\zeta,S})$ and a completion of $H_{\check{T}}^\bullet(\Gr^\zeta)\otimes_{H_{\check{T}}^\bullet(\pt)}S$.

\subsection{Application to study of $\sO_\zeta$} 
In a sequential paper \cite{Situ2}, we prove a non-deformed version of Theorem \ref{thm A}, i.e. an isomorphism between $Z(\sO_{\zeta})$ and a completion of $H^\bullet(\Gr^{\zeta})$. 
It can not be obtained directly from Theorem \ref{thm A} by specialization, and so it requires some new ideas. 
However, the deformed isomorphisms in Theorem \ref{thm A} help to verify several compatibilities used in the proof. 
In another paper \cite{Situ3} we will show that the principal block of $\sO_\zeta$ is derived equivalent to a version of the affine Hecke category, where we make an essential use of the deformation $\sO_{\zeta,S}$ by establishing a deformed version of the equivalence involving $\sO_{\zeta,S}$ and deducing the desired equivalence by specialization.

\subsection{Organization of the paper} 
In Section 2 we recall the definitions and some general facts about quantum groups, the hybrid version and their categories $\sO$. 
In Section 3, we study some basic properties of the category $\sO$ for $U^\hb_q$, e.g. classification of projective modules and simple modules, the BGG reciprocity, the Jantzen filtration and the linkage principle. 
In Section 4, we study the subgeneric deformation category $\sO_{S_\alpha}$ and construct the ``endomorphism algebra of the big projective object". 
In Section 5, we give a reminder of $\Gr^\zeta$, discuss the cohomology of the subtorus-fixed locus, and then prove Theorem \ref{thm A}, including a $\check{T}\times \Gm$-equivariant version of this theorem. 

\subsection{Notations and conventions} 
\subsubsection{Notations} 
For a complex variety $X$ with an action of a complex linear group $G$, we denote by $H^\bullet_G(X)$ the $G$-equivariant cohomology with coefficients in $\k$. 
For a Lie algebra $\g$ over $\k$, we denote by $U\g$ its enveloping algebra. 
\subsubsection{Conventions} 
For a group $K$, a \textit{$K$-set} refers to a set equipped with a $K$-action. 
Categories and functors are additive and $\k$-linear. 
A \textit{block} in a category means an additive full subcategory that is a direct summand (not necessarily indecomposable). 

Let $\sC$ be an $R$-linear category. 
The \textit{center} $Z(\sC)$ of $\sC$ is the ring of $R$-linear endomorphism of the identity functor of $\sC$, i.e. 
$$Z(\sC)=\{\big(z_M\in \End_{\sC}(M)\big)_{M\in \sC}|\ f\circ z_{M_1}=z_{M_2}\circ f,\ \forall M_1,M_2\in \sC, \forall f\in \Hom_{\sC}(M_1,M_2)\}.$$
We may abbreviate $\Hom(M_1,M_2)=\Hom_{\sC}(M_1,M_2)$ if there is no ambiguity. 
For a set $\sX$, we denote by $\Fun(\sX,R)$ the space of $R$-valued functions on $\sX$, which is naturally endowed with an $R$-algebra structure. 

For a smooth $\k$-scheme $X$, we denote by $\sT_{X}$ its tangent sheaf. 

\subsection{Acknowledgments} 
The author sincerely thanks his supervisor Professor Peng Shan for suggesting this problem and her patient guidance. 
Without her help this article could not be finished by the author alone.

\section{General theories for quantum groups} 
\subsection{Root data} 
Let $G$ be a complex connected and simply-connected semisimple algebraic group, with a Borel subgroup $B$ and a maximal torus $T$ contained in $B$. 
Let $B^-$ be the opposite Borel subgroup, and let $N$, $N^-$ be the unipotent radical of $B$, $B^-$. 
Denote their Lie algebras by 
$$\g=\Lie(G),\quad \b=\Lie(B), \quad \b^-=\Lie(B^-), \quad \n=\Lie(N), \quad \n^-=\Lie(N^-),\quad \t=\Lie(T).$$ 
Let $W=N_G(T)/T$ be the Weyl group for $G$. 
Let $h$ be the Coxeter number of $G$. 
Let $\check{G}$ be the Langlands dual group of $G$, with the dual torus $\check{T}$. 

Let $(X^*(T),X_*(T), \Phi, \check{\Phi})$ be the root datum associated with $G$. 
Let $\Phi^+$ (resp. $\check{\Phi}^+$) and $\Sigma=\{\alpha_i\}_{i\in \I}$ (resp. $\check{\Sigma}=\{\check{\alpha}_i\}_{i\in \I}$) be the subsets of $\Phi$ (resp. $\check{\Phi}$) consisting of positive roots (resp. positive coroots) and simple roots (resp. simple coroots). 
We abbreviate $\Lambda:= X^*(T)$ and $\check{\Lambda}:= X_*(T)$, and let $\langle-,-\rangle:\check{\Lambda} \times \Lambda \rightarrow \Z$ be the canonical pairing. 
There is an order $\leq$ on $\Lambda$ defined by $\lambda\leq \mu$ if $\mu-\lambda\in \Z_{\geq}\Sigma$. 
Let $\rQ\subset \Lambda$, $\check{\rQ}\subset \check{\Lambda}$ be the root and coroot lattices. 
Recall that the fundamental group of $\check{G}$ is $\pi_1:=\pi_1(\check{G})=\Lambda/\rQ$.  
Let $a_{ij}:=\langle \check{\alpha}_i, \alpha_j\rangle$ be the $(i,j)$-th entry of the Cartan matrix of $G$. 
Let $(d_i)_{i\in \I}\in \N^\I$ be a tuple of relatively prime positive integers such that $(d_ia_{ij})_{i,j\in \I}$ is symmetry and positive definite. 
It defines a pairing $(-,-):\rQ \times \rQ \rightarrow \Z$ by $(\alpha_i,\alpha_j):=d_ia_{ij}$ and extents to 
$$(-,-):\Lambda \times \Lambda \rightarrow \frac{1}{e} \Z, \quad e:=|\pi_1|.$$ 
For any $\beta=w(\alpha_i)\in \Phi^+$ for some $w\in W$, we set $d_\beta=d_i$. 
Let $\{e_{\alpha}\}_{\alpha\in \Phi^+}$ and $\{f_{\alpha}\}_{\alpha\in \Phi^+}$ be the Chevalley basis for $\n$ and $\n^-$, respectively. 

\subsubsection{Affine root data} 
Let $\t_\af=\t\oplus \k \delta$ be the affine torus, where $\delta$ is the generator of positive imaginary coroots. 
Let $\check{\Phi}_\af$ and $\check{\Phi}_\re=\check{\Phi} \oplus \Z \delta$ be the set of the affine coroots and affine real coroots. 
Let $\check{\Phi}_\af^+$ and $\check{\Sigma}_\af$ be the set of affine positive coroots and affine simple coroots. 
Set $\check{\Phi}_\re^+:=\check{\Phi}_\re\cap \check{\Phi}_\af^+$. 
Let 
$$W_\af:= W\ltimes \rQ, \quad W_\ex:= W\ltimes \Lambda \simeq W_\af \rtimes \pi_1$$ 
be the {affine Weyl group} and the {extended Weyl group}, where $\pi_1$ acts on $W_\af$ by automorphisms of affine Dynkin graph. 
Denote by $l(-):W_\af \rightarrow \Z_{\geq 0}$ the length function. 
Denote by $s_\alpha\in W_\ex$ the reflection associated with $\alpha\in \check{\Phi}_\re$, and denote by $\tau_\mu\in W_\ex$ the translation by $\mu\in \Lambda$. 
For any $\alpha\in \Phi^+$, we set $s_{\alpha,m}=s_{\check{\alpha}} \cdot \tau_{m\alpha}$ and abbreviate $s_{\alpha}=s_{\alpha,0}$. 
For any subset $J\subset \check{\Sigma}_\af$, let $W_J\subset W_\af$ be the parabolic subgroup generated by reflections $s_\alpha$ with $\alpha\in J$. 
We identify 
$$W^J_\af:=W_\af/W_J=\{ x\in W_\af|\ l(x)\leq l(y),\ \forall y\in xW_J\}.$$ 
Let $W^J_\ex=W_\ex/W_J=\pi_1\times W^J_\af$. 

\subsubsection{The $l$-affine Weyl groups}\label{subsect 2.1.2'} 
Let $l\geq h$ be an odd positive integer which is prime to $e$, and to $3$ if $G$ contains a component of type $G_2$. 
We set $\check{\Phi}_{l,\af}:=\check{\Phi}_\af\cap (\check{\rQ}\oplus l\Z\delta)$ and $\check{\Phi}_{l,\re}:=\check{\Phi}_\re\cap (\check{\rQ}\oplus l\Z\delta)$. 
Let $W_{l,\af}:= W\ltimes l\rQ$ and $W_{l,\ex}:= W\ltimes l\Lambda$ be the $l$-affine Weyl group and the $l$-extended Weyl group. 
There is a shifted action of $W_{l,\ex}$ on $\Lambda$, given by $w\bullet \lambda:=w(\lambda+\rho)-\rho$ for any $w\in W_{l,\ex}$ and $\lambda\in \Lambda$, where $\rho=\frac{1}{2}\sum_{\alpha\in \Phi^+}\alpha$. 
Set 
$$\Xi_\sc:= \Lambda/ (W_{l,\af}, \bullet), \quad \Xi:=\Xi_\sc/\pi_1= \Lambda/ (W_{l,\ex}, \bullet) ,$$ 
where $\bullet$ represents the $\bullet$-action above. 
We can identify 
$$\Xi_\sc=\{\omega\in \Lambda|\ 0\leq\langle\omega+\rho, \check{\alpha}\rangle\leq l,\ \forall \alpha\in \Phi^+ \},$$ 
since any coset in $\Xi_\sc$ is uniquely determined by an element in the RHS. 
For $\omega\in \Xi_\sc$, we denote $W_{l,\omega}=\Stab_{(W_{l,\af},\bullet)}(\omega)$, and set $W^\omega_{l,\af}=W_{\l,\af}/W_{l,\omega}$, $W^\omega_{l,\ex}=W_{\l,\ex}/W_{l,\omega}\simeq \pi_1\times W^\omega_{l,\af}$. 

\subsubsection{Rings} 
Let $\zeta_e\in \k$ be a primitive $l$-th root of unity, and let $\zeta=(\zeta_e)^e$. 
Let $q$ be a formal variable and set $q_e=q^{\frac{1}{e}}$. 
We set $\bA=\k[q_e^{\pm1}]$ and $\bF=\k(q_e)$. 
We identify the graded rings $\k[\hbar]=H_{\Gm}^\bullet(\pt)$, where $\hbar$ is of degree $2$. 
We denote by $\bA_{\widehat{\zeta_e}}$ the completion at $q_e=\zeta_e$, and $\k[\hbar]_{\widehat{0}}$ the completion at $\hbar=0$. 
There is an identification $\bA_{\widehat{\zeta_e}}\simeq \k[\hbar]_{\widehat{0}}$ via $\hbar=q_e-\zeta_e$. 

We set $S'=\k[\t]$ and $\hat{S}'=S'[\hbar]$. 
Consider the $W$-invariant isomorphism $\t^*\xs \t$ such that $\beta\mapsto d_\beta \check{\beta}$ for any $\beta\in \Phi^+$. 
It yields an isomorphism $S'\xs \k[\t^*] =H_{\check{T}}^\bullet(\pt)$, and hence we identify $\hat{S}'=H_{\check{T}\times \Gm}^\bullet(\pt)$. 
Let $S=S'_{\widehat{0}}$ be the completion at $0\in \t$, and $\k[T]_{\widehat{1}}$ be the completion at $1\in T$. 
We have $S=\k[T]_{\widehat{1}}$ via the exponential map $\exp:\t \rightarrow T$. 
Let $\hat{S}=\bA[\t]_{\widehat{(\zeta_e,0)}}=\bA[T]_{\widehat{(\zeta_e,1)}}$ be the completion at $q_e=\zeta_e$ and at $0\in \t$ or $1\in T$. 
Then $\hat{S}$ coincides with the completion $\hat{S}'_{\widehat{(0,0)}}$ at $\hbar=0$ and $0\in \t$. 
We also consider an isomorphism 
$$\hat{S}\xs \k[\t_\af^*]_{\widehat{0}}, \quad \text{such that}\quad q^l\mapsto \exp(l\delta),\ \beta\mapsto d_\beta\check{\beta},\ \forall \beta\in \Phi^+.$$ 
We set $\K=\Frac S$ and $\hat{\K}=\Frac \hat{S}$. 

\subsection{Quantum groups} 
\subsubsection{Quantum groups} 
Let $q$ be a formal variable. 
We abbreviate $q_\beta=q^{d_\beta}$, $q_i=q_{\alpha_i}$ for each $\beta\in \Phi^+$, $i\in \rI$. 
 The quantum group $\mathscr{U}_q$ associated with $G$ is the $\bF$-algebra generated by
 $E_i, F_i, K_\lambda \ (i\in \I, \lambda\in \Lambda)$, subject to the quantum Chevalley relations, and 
 \begin{align*}
    &K_0=1, \quad K_i=K_{\alpha_i}, \quad K_{\lambda}K_{\mu}=K_{\lambda+\mu},  \\
    &K_{\lambda}E_jK_{\lambda}^{-1}=q^{(\lambda, \alpha_j)}E_j, \quad K_{\lambda}F_jK_{\lambda}^{-1}=q^{-(\lambda, \alpha_j)}F_j , \\ 
    &E_iF_j-F_jE_i = \delta_{i,j} \frac{K_i-K_i^{-1}}{q_i -q_i^{-1}}, 
 \end{align*} 
for any $i, j \in \I$ and $\lambda,\mu\in \Lambda$. 
Let $\mathscr{U}_q^0$ be the subalgebra generated by $K_\lambda$'s. 
We will identify $\mathscr{U}_q^0=\bF[\Lambda]=\bF[T]$. 

The algebra $\mathscr{U}_q$ admits a Hopf algebra structure with comultiplication $\Delta$, antipode $\mathrm{S}$ and counit $\epsilon$ defined by 
$$\Delta(E_i)=E_i\otimes 1+ K_i\otimes E_i,\quad \Delta(F_i)=F_i\otimes K_i^{-1}+ 1\otimes F_i,\quad \Delta(K_\lambda)=K_\lambda\otimes K_\lambda, $$ 
$$\mathrm{S}(E_i)=-K_i^{-1}E_i,\quad \mathrm{S}(F_i)=-F_iK_i,\quad \mathrm{S}(K_\lambda)=K_\lambda^{-1},$$ 
$$\epsilon(E_i)=\epsilon(F_i)=0, \quad \epsilon(K_\lambda)=1. $$ 

\subsubsection{Integral forms}\label{subsect 2.1.2} 
For any $m\in \Z$ and $n\in \N$, we introduce the following elements in $\mathscr{U}_q$ 
$$E_i^{(n)}:= \frac{E_i^n}{[n]_{q_i}!}, \quad 
F_i^{(n)}:= \frac{F_i^n}{[n]_{q_i}!}, \quad 
\qbno{K_\lambda;m}{n}_{q_i}:= \prod_{r=1}^n \frac{K_\lambda q_i^{m+1-r}-K_\lambda^{-1}q_i^{r-m-1}}{q_i^r-q_i^{-r}} ,$$ 
where $[n]_{q_i}:=\frac{q_i^n-q_i^{-n}}{q_i-q_i^{-1}}$ is the quantum number, and $[n]_{q_i}!:=\prod_{j=1}^n [j]_{q_i}$ is the quantum factorial. 
We abbreviate $[K_\lambda;m]_{q_i}:=\qbno{K_\lambda;m}{1}_{q_i}$. 
There is a chain of $\bA$-subalgebras (in fact Hopf subalgebras) of $\mathscr{U}_q$ 
\begin{equation}\label{equ QGs} 
\fU_q \subset U_q^{\hb} \subset U_q 
\end{equation}
defined as follows: 
\begin{itemize}
	\item $\fU_q$ is generated by $E_i, F_i, K_\lambda $ for ${i\in \I, \lambda\in \Lambda}$; 
	\item $U_q^{\hb}$ is generated by $E_i^{(n)}, F_i, K_\lambda $ for ${i\in \I, \lambda\in \Lambda, n\in \Z_{>0}}$; 
	\item $U_q$ is generated by $E_i^{(n)}, F_i^{(n)}, K_\lambda $ for ${i\in \I, \lambda\in \Lambda, n\in \Z_{>0}}$. 
\end{itemize} 
The integral form $U_q$ was firstly introduced and studied by Lusztig \cite{Lus90}, and $\fU_q$ by De Concini and Kac \cite{DeCK90}. 
Following \cite{BBASV}, we call $U_q^{\hb}$ the \textit{hybrid quantum group}. 
We denote the positive, negative and the zero parts by 
$$\fU_q^+:=\langle E_i \rangle, \quad \fU_q^-:=\langle F_i \rangle, \quad \fU^0_q:=\langle K_\lambda, [K_i;0]_{q_i} \rangle, $$ 
$$U_q^+:=\langle E_i^{(n)} \rangle, \quad U_q^-:=\langle F_i^{(n)} \rangle, \quad U^0_q:=\langle K_\lambda, \qbno{K_i;m}{n}_{q_i} \rangle .$$ 
Then there are triangular decompositions 
$$ \fU_q= \fU_q^-\otimes \fU_q^0 \otimes \fU_q^+, \quad U_q= U_q^-\otimes U_q^0 \otimes U_q^+ , \quad \text{and} \quad U_q^{\hb}=  \fU_q^-\otimes \fU_q^0 \otimes U_q^+ .$$ 

\subsubsection{Specializations} 
For any integral form $A_q$ above, we let the $\k$-algebra $A_\zeta:=A_q\otimes_{\bA} \k$ be the specialization at $q_e= \zeta_e$. 
The specialization yields a chain of maps 
\begin{equation}\label{equ QGspe} 
\fU_\zeta \rightarrow U_\zeta^{\hb} \rightarrow U_\zeta .
\end{equation} 
The \textit{small quantum group} $u_\zeta$ was defined by Lusztig to be the finite dimensional subalgebra in $U_\zeta$ generated by $E_i, F_i, K_\lambda$ for $i\in \I$ and $\lambda\in \Lambda$, which is also the image of $\fU_\zeta\rightarrow U_\zeta$. 
We denote by $u_\zeta^\flat$ the image of $\fU_\zeta^\flat$, for $\flat=-,0,+$. 
Uniformly we define 
\begin{equation}\label{equ midQG} 
u_\zeta:=\im(\fU_\zeta\rightarrow U_\zeta),\quad U_\zeta^{\frac{1}{2}}:=\im (U_\zeta^{\hb}\rightarrow U_\zeta), \quad \fU^b_\zeta:=\im(\fU_\zeta \rightarrow U_\zeta^{\hb}) ,
\end{equation} 
then there are triangular decompositions 
$$u_\zeta=u_\zeta^-\otimes u_\zeta^0 \otimes u_\zeta^+, \quad U_\zeta^{\frac{1}{2}}=u_\zeta^-\otimes u_\zeta^0 \otimes U_\zeta^+ , \quad \fU^b_\zeta= \fU^-_\zeta\otimes \fU^0_\zeta \otimes u_\zeta^+ .$$ 
For $A=A_q$ or $A_\zeta$ above, we abbreviate the subalgebras $A^{\leq}:=A^- A^0$ and $A^{\geq}:=A^+ A^0$. 

\subsubsection{PBW basis}\label{subsect PBW} 
Fix a convex order on $\Phi^+$, Lusztig \cite[\textsection 4]{Lus90} introduced the \textit{root vectors} $E_\beta\in U_q^+$ and $F_\beta\in U_q^-$ for any $\beta\in \Phi^+$. 
By our hypothesis that $l\geq h$, those root vectors are well-defined in the algebras in (\ref{equ QGspe}) and (\ref{equ midQG}). 
We denote the set of \textit{Kostant partitions} by 
$$\Par:= \bigsqcup\limits_{\eta\in \rQ} \Par(\eta), \quad \Par(\eta ):=\{ \wp\in \N^{\Phi^+} |\ \sum_{\beta\in \Phi^+} \wp(\beta)\beta = \eta \}.$$ 
For $\wp\in \N^{\Phi^+}$, we abbreviate 
$$ E^{\wp}:= \prod_{\beta\in \vec{\Phi}^+} E^{\wp(\beta)}_{\beta} ,\quad 
E^{(\wp)}:= \prod_{\beta\in \vec{\Phi}^+} E_\beta^{(\wp(\beta))} , $$ 
and similarly for $F^{\wp}$ and $F^{(\wp)}$, where $\vec{\Phi}^+$ is the set $\Phi^+$ endowed with the convex order and the products go from left to right when the elements in $\vec{\Phi}^+$ increase. 
By \cite[Thm 6.7]{Lus90}, the family $\{E^{(\wp)}\}_{\wp\in \Par}$ (resp. $\{F^{(\wp)}\}_{\wp\in \Par}$) forms a $\bA$-basis for $U^+_q$ (resp. $U^-_q$); by \cite[Prop 1.7]{DeCK90}, the family $\{E^{\wp}\}_{\wp\in \Par}$ (resp. $\{F^{\wp}\}_{\wp\in \Par}$) forms a $\k$-basis for $\fU^+_\zeta$ (resp. $\fU^-_\zeta$). 

Lusztig \cite[\textsection 8]{Lus90} defined the \textit{quantum Frobenius homomorphism} 
\begin{equation}
\Fr: U_\zeta \rightarrow U\g, \quad 
\text{by}\quad E_\beta^{(n)},\ F_\beta^{(n)}\mapsto 
\begin{cases} 
		\frac{e_\beta^{n/l}}{(n/l)!},\ \frac{f_\beta^{n/l}}{(n/l)!} & \text{if}\ l|n , \\ 
		 0 & \text{if else} 
\end{cases}, \quad K_\lambda\mapsto 1. 
\end{equation} 
It restricts to the homomorphisms $\Fr : U_\zeta^+ \rightarrow U\n$ and $\Fr: U_\zeta^-\rightarrow U\n^-$. 

\subsubsection{Frobenius center}\label{subset Frocen} 
The \textit{Frobenius center} of $\fU_\zeta$ is the $\k$-subalgebra 
$$Z_\Fr:=\langle K^l_{\lambda}, F^l_\beta , E^l_\beta \rangle_{\lambda\in \Lambda, \beta\in \Phi^+}. $$ 
We abbreviate $Z_\Fr^{\flat}:= Z_\Fr \cap \fU_\zeta^{\flat}$ for $\flat=-,+,0,\leq$ or $\geq$. 
By \cite[\textsection 0]{DeCKP92}, there are isomorphisms of $\k$-algebras, $Z_\Fr^{-}\xs \k[N^-]$, $Z_\Fr^{+}\xs \k[N]$ and $Z_\Fr^{0}\xs \k[T]$, which give an isomorphism  
\begin{equation}\label{equ 2.7} 
\Spec Z_\Fr \xs G^*, 
\end{equation}
where $G^*=N^-\times T \times N$ is the Poisson dual group of $G$. 
Note that $\fU_\zeta^b$ coincides with the algebra $\fU_\zeta\otimes_{\k[N]} \k$ by evaluating $\k[N]$ at $1\in N$. 
We may view $Z_\Fr^{\leq}$ as a subalgebra in $U_\zeta^\hb$, which is central in $\fU_\zeta^b\subset U_\zeta^\hb$. 

\subsection{Representations}\label{subsect 2.3} 
Following \cite[\textsection 2]{AJS94}, we define some module categories for an algebra with general structures as the quantum groups introduced above. 

There is a group homomorphism $\Lambda\rightarrow \Aut_{\bF\text{-algebra}}(\mathscr{U}_q^0)$, $\mu\mapsto \tau_\mu$, such that 
$${\tau_\mu}(K_\lambda)=q^{(\mu,\lambda)}K_\lambda, \quad \forall \mu,\lambda \in \Lambda.$$ 
The $\Lambda$-action on $\mathscr{U}_q^0$ preserves the integral forms $\fU^0_q$ and $U^0_q$, and specializes to an action on $\fU^0_\zeta$ and $U^0_\zeta$. 

Let $L$ be a commutative ring and $A^0$ be a finitely generated $L$-algebra that are either of the following cases: (1) $L=\k$, $A^0\subset \fU_\zeta^0$ or $A^0\subset U_\zeta^0$; (2) $L=\bA$, $A^0\subset U_q^0$; (3) $L=\bF$, $A^0\subset \mathscr{U}_q^0$, and in each cases, $A^0$ is preserved under the $\Lambda$-action. 

Let $A=\bigoplus_{\lambda\in \rQ}A_\lambda$ be a finitely generated $\rQ$-graded $A^0$-algebra such that $A^0\subset A_0$, and 
$$fm=m\tau_\lambda(f),\quad \forall f\in A^0,\ \forall m\in A_\lambda.$$ 
Suppose that there are subalgebras $A^-$ and $A^+$ of $A$, admitting decompositions of free $L$-modules of finite ranks, 
$$A^-=L\oplus \bigoplus_{\lambda< 0}A^-_{\lambda}, \quad A^+=L\oplus \bigoplus_{\lambda> 0}A^+_{\lambda}, $$ 
such that the multiplication yields an isomorphism of $L$-modules 
$$A^-\otimes A^0\otimes A^+\xs A.$$ 
Here and from now on, $\otimes$ appears without an index $L$. 
We assume $A^-$ is left Noetherian. 
Abbreviate $A^\leq=A^-A^0$ and $A^\geq=A^+A^0$. 

Let $R$ be a commutative and Noetherian $A^0$-algebra, with the structure map $\pi: A^0\rightarrow R$. 
We call $R$ a \textit{deformation ring for $A$}. 
If $A^0\subset \fU_\zeta^0$ (resp. $A^0\subset \fU_q^0$), we will view $S$ (resp. $\hat{S}$) as a deformation ring by the inclusion $\fU_\zeta^0=\k[\Lambda]=\k[T]\subset S$ (resp. $\fU_q^0\subset\bA_{\widehat{\zeta_e}}[\Lambda]=\bA_{\widehat{\zeta_e}}[T]\subset \hat{S}$). 

\subsubsection{Module categories}\label{subsect ModCat} 
We define $A\Mod_R^{\Lambda}$ to be the category consisting of $A\otimes R$-modules $M$ endowed with a decomposition of $R$-modules (called the \textit{weight spaces}) 
$$M=\bigoplus\limits_{\mu\in \Lambda}M_\mu,$$ 
such that $M_\mu$ is killed by the elements in $A\otimes R$ of the form 
$$f\otimes 1-1\otimes \pi({\tau_\mu}(f)), \quad f\in A^0.$$  
Morphisms in $A\Mod_R^{\Lambda}$ are the morphisms of $A\otimes R$-modules that respect the decompositions. 
The category $A\Mod_R^{\Lambda}$ is naturally abelian and $R$-linear. 
Let $A\mod_R^{\Lambda}$ be the full subcategory of $A\Mod_R^{\Lambda}$ consisting of finitely generated $A\otimes R$-modules whose weight spaces are finitely generated $R$-modules. 

\begin{lem}
If $R=\F$ is a field, then the category $A\mod_{\F}^{\Lambda}$ is hom-finite, i.e. the morphism space of any pair of objects is finite dimensional over $\F$. 
In particular, $A\mod_{\F}^{\Lambda}$ is Krull-Schmidt. 
\end{lem}  

Define the \textit{category $\sO$ for $A$} to be the full subcategory $\sO^A_R$ of $A\mod_R^{\Lambda}$ of modules that are locally unipotent for the action of $A^+$. 
It is an abelian subcategory of $A\Mod_R^{\Lambda}$. 

\subsubsection{Verma and simple modules}\label{subsect Verma} 
Let $R_\lambda$ be the $\Lambda$-graded $A^0\otimes R$-module such that 
$$(R_\lambda)_\mu= \delta_{\lambda,\mu} R, \quad (f\otimes 1 -1\otimes \pi({\tau_\lambda}f)).R_\lambda=0, \quad \forall f\in A^0 .$$ 
Define the \textit{Verma module} 
$$M^A(\lambda)_R:= A \otimes_{A^{\geq}} R_\lambda \ \in \sO^A_R$$ 
where $R_\lambda$ is an $A^{\geq}$-module via the projection $A^{\geq} \twoheadrightarrow A^0$. 
Note that $M^A(\lambda)_R = A^{-} \otimes R_\lambda$ as $A^\leq$-modules. 
In particular, $M^A(\lambda)_R$ is a free $R$-module. 

If $R=\F$ is a field, $M^A(\lambda)_\F$ has a unique simple quotient $L^A(\lambda)_\F$, and every simple object in the category $\sO^A_\F$ arises in this way. 

\subsubsection{Projective modules and truncation} 
For any $\nu\in \Lambda$, denote by $A\Mod_R^{\Lambda,\leq \nu}$ (the \textit{truncated category}) the full subcategory of $A\Mod_R^{\Lambda}$ consisting of the module $M$ such that $M_\mu=0$ unless $\mu \leq \nu$. 
Define the \textit{truncation functor} to be (drop the subscripts $A,R$ in $\tau^{\leq \nu}$ for simplicity) 
$$ \tau^{\leq \nu}: A\Mod_R^{\Lambda} \rightarrow A\Mod_R^{\Lambda,\leq \nu} ,\quad M \mapsto M \big/ A. \big( \bigoplus_{\mu\nleq \nu} M_\mu \big) ,$$ 
which is by taking the maximal quotient in $A\Mod_R^{\Lambda,\leq \nu}$. 
Hence $\tau^{\leq \nu}$ is left adjoint to the natural inclusion $A\Mod_R^{\Lambda,\leq \nu} \hookrightarrow A\Mod_R^{\Lambda}$. 
Denote the counit by $\epsilon^{\leq \nu}:\id \rightarrow \tau^{\leq \nu}$. 

Set $\sO^{A,\leq \nu}_R=\sO^{A}_R\cap A\Mod_R^{\Lambda,\leq \nu}$. 
Any object in $\sO^A_R$ admits a finite decomposition with each direct factors contained in $\sO^{A,\leq \nu}_R$ for some $\nu$. 
The category $\sO^{A}_R$ may not admit enough projective objects, but $\sO^{A,\leq \nu}_R$ always does. 
Indeed, the modules of the form 
\begin{equation}\label{equ 2.9} 
Q'^{A}(\mu)^{\leq \nu}_R:= A\otimes_{A^{\geq}}\big((A^{\geq}/ \sum_{\lambda\nleq \nu-\mu} A^{\geq}_{\lambda}) \otimes_{A^0} R_\mu \big) 
\end{equation} 
provide enough projective objects in $\sO^{A,\leq \nu}_R$. 
Each projective module $Q$ has a \textit{Verma filtration}, i.e. a finite filtration whose composition factors are Verma modules (called the \textit{Verma factors}). 
Denote by $(Q:M^A(\lambda)_R)$ the multiplicity of $M^A(\lambda)_R$ in the Verma factors, which is independent with the choices of Verma filtrations. 
\begin{lem}[{\cite[Prop 2.6]{Fie03}}]\label{lem 2.2} 
Suppose $R$ is a local complete Noetherian domain, with residue field $\F$. 
Then there exists projective cover $Q^A(\mu)^{\leq \nu}_R$ for any $L^A(\mu)_\F$ in $\sO^{A,\leq \nu}_R$, and $\{Q^A(\mu)^{\leq \nu}_R\}_{\mu\leq \nu}$ forms a complete family of indecomposable projective modules in $\sO^{A,\leq \nu}_R$. 
Moreover $Q^A(\mu)^{\leq \nu}_\F=Q^A(\mu)^{\leq \nu}_R\otimes_R \F$, hence $(Q^A(\mu)^{\leq \nu}_R:M^A(\lambda)_R)=(Q^A(\mu)^{\leq \nu}_\F:M^A(\lambda)_\F)$. 
\end{lem} 

\subsubsection{Base change}\label{subsect 2.3.4} 
Let $R'$ be a commutative Noetherian $R$-algebra. 
There is a base change functor $-\otimes_R R': \sO^A_R \rightarrow \sO^A_{R'}$. 
Denote by $\sP^{A,\leq \nu}_R$ the subcategory of projective modules in $\sO^{A,\leq \nu}_R$, then it is generated by the direct factors of the modules $Q'^{A}(\mu)^{\leq \nu}_R$. 
The base change functor yields a natural equivalence, see \cite[Prop 2.4]{Fie03}, 
\begin{equation}\label{equ 2.10} 
\sP^{A,\leq \nu}_R\otimes_R R'\xs \sP^{A,\leq \nu}_{R'}. 
\end{equation} 
It induces a map on the centers 
$$Z(\sO^{A,\leq\nu}_R)=Z(\sP^{A,\leq \nu}_R)\xrightarrow{-\otimes_R R'} Z(\sP^{A,\leq \nu}_{R'})=Z(\sO^{A,\leq\nu}_{R'}),$$ 
which is an inclusion if $R\rightarrow R'$ is injective. 
Since $\sO^{A}_R$ is the direct limit of $\sO^{A,\leq\nu}_R$, it shows that $Z(\sO^{A}_R)=\varprojlim_{\nu} Z(\sO^{A,\leq\nu}_R)$, therefore we have an $R$-algebra homomorphism 
\begin{equation} 
-\otimes_R R':\ Z(\sO^A_R) \rightarrow Z(\sO^A_{R'}), 
\end{equation} 
which is an inclusion if $R\rightarrow R'$ is injective. 

\subsubsection{ } 
Let $A_q$ be one of the algebras in \textsection \ref{subsect 2.1.2}, and $A_\zeta$ be its specialization. 
If the deformation ring $R$ for $A_q$ satisfies $\pi: A^0_q\rightarrow A^0_\zeta\rightarrow R$, then by definition $A_q\Mod^\Lambda_R=A_\zeta\Mod^\Lambda_R$. 
We will not distinguish the categories from now on. 

Let $R$ be a local complete Noetherian domain, with residue field $\F$. 
We write $L^A(\lambda)_\F$ as 
$$E(\lambda)_\F\text{ for }U^\hb_q, \quad L(\lambda)_\F\text{ for }u_\zeta, \quad L^b(\lambda)_\F\text{ for }\fU^b_\zeta.$$ 
We abbreviate $M(\lambda)_R=M^A(\lambda)_R$ and $Q(\mu)^{\leq \nu}_R=Q^A(\mu)^{\leq \nu}_R$ for $A=U^\hb_q$. 
We abbreviate $\sO_{R}=\sO^A_R$ for $A=U^\hb_q$.

\section{Category $\sO$ of hybrid quantum group} 
Let $R$ be a deformation ring for $U^\hb_q$. 
In this section, we study some basic properties for the category $\sO_{R}$. 
\subsection{Projective and simple modules} 
%In the following \textsection \ref{subsect 3.1.1} and \textsection \ref{subsect 3.1.2}, we will assume $G$ is simply connected. 
\subsubsection{Simple modules}\label{subsect 3.1.1} 
Denote the set of $l$-restricted dominant weights by 
$$\Lambda_l^+=\{\mu\in \Lambda|\ 0\leq \langle \mu, \check{\alpha}_i \rangle <l,\ \forall i\in \I \}.$$ 
By \cite[Prop 7.1]{Lus89}, for any $\lambda^0\in \Lambda_l^+$, the simple module $L(\lambda^0)_\k$ of $u_\zeta$ can be extend to a $U_\zeta$-module. 
We view $L(\lambda^0)_\k$ as a $U^\hb_\zeta$-module via $U^\hb_\zeta\rightarrow U_\zeta$. 
For any $\nu\in \Lambda$, there is a trivial $U_\zeta^\hb$-module $\k_{l\nu}$ supported on the weight $l\nu$. 
\begin{lem}\label{lem simple} 
We have $E(\lambda)_\k=L(\lambda^0)_\k \otimes \k_{l\lambda^1}$ for any $\lambda\in \Lambda$, where $\lambda=\lambda^0+l\lambda^1$ is the unique decomposition such that $\lambda^0\in \Lambda_l^+$. 
\end{lem} 
\begin{proof} 
Note that $E(\lambda)_\k$ appears as a simple factor of $L(\lambda^0)_\k \otimes \k_{l\lambda^1}$. 
So it is enough to show that $L(\lambda^0)_\k \otimes \k_{l\lambda^1}$ is simple. 
The $U^\hb_\zeta$-action on $L(\lambda^0)_\k \otimes \k_{l\lambda^1}$ factors through $U^\hb_\zeta\twoheadrightarrow U_\zeta^{\frac{1}{2}}$, and $U_\zeta^{\frac{1}{2}}$ contains $u_\zeta$ as a subalgebra. 
Note that $L(\lambda^0)_\k \otimes \k_{l\lambda^1}$ is simple as a $u_\zeta$-module, so is it as a $U^\hb_\zeta$-module. 
\end{proof} 

\subsubsection{Projective modules}\label{subsect 3.1.2} 
Suppose $R$ is an $S$-algebra that is a local Noetherian domain with residue field $\F$. 
In \cite[\textsection 3.3.9]{BBASV}, the authors define a module $Q(\lambda)_R$ in $U^\hb_{\zeta}\Mod^{\Lambda}_R$ by 
\begin{equation}\label{equ 3.1} 
Q(\lambda)_R:= U^\hb_\zeta\otimes_{\fU^b_\zeta} P^b(\lambda)_R , \quad \lambda\in \Lambda, 
\end{equation} 
where $P^b(\lambda)_R$ is the projective cover for $L^b(\lambda)_\F$ in $\fU^b_\zeta\mod^{\Lambda}_R$. 
If $\lambda\in -\rho+l\Lambda$, we have $P^b(\lambda)_R=M(\lambda)_R$ as $\fU^b_\zeta$-modules, hence 
\begin{equation}\label{equ 3.20} 
\begin{aligned} 
Q(\lambda)_R&=U^\hb_\zeta\otimes_{\fU^b_\zeta} M(\lambda)_R
=U^\hb_\zeta \otimes_{U^{\hb,\geq}_\zeta} (U^{\hb,\geq}_\zeta\otimes_{\fU^{b,\geq}_\zeta} R_{\lambda})\\ 
&=U^\hb_\zeta \otimes_{U^{\hb,\geq}_\zeta} (U\n\otimes R_{\lambda}) , 
\end{aligned} 
\end{equation} 
where $U\n\otimes R_{\lambda}$ is viewed as a $U^{\hb,\geq}_\zeta$-module by $U^{\hb,\geq}_\zeta= U_\zeta^+\rtimes \fU_\zeta^0\xrightarrow{\Fr\rtimes (\pi\circ \tau_{\lambda})} U\n \otimes R_\lambda$. 
\begin{lem}{\cite[Lem 3.7]{BBASV}}\label{lem 3.2} 
\begin{enumerate} 
\item The functor $\Hom_{U^\hb_{\zeta}\Mod^{\Lambda}_R}(Q(\lambda)_R,-)$ on $\sO_{R}$ is exact; 
\item For any $\nu \geq \lambda$, the truncation $\tau^{\leq \nu}Q(\lambda)_R$ is projective  in $\sO^{\leq \nu}_{R}$. 
	It surjects to $E(\lambda)_\F$ and admits a Verma filtration. 
\item For $R=S$, $\tau^{\leq \nu}Q(\lambda)_S$ is the projective cover $Q(\lambda)_S^{\leq \nu}$ for $E(\lambda)_\k$ in $\sO^{\leq \nu}_{S}$. 
\end{enumerate} 
\end{lem}
\begin{proof}
Part (1) and (2) are in the \textit{loc. cit}. 
(3) For any $\mu\leq \nu$, we have isomorphisms 
\begin{align*} 
\Hom_{U^\hb_\zeta\Mod^{\Lambda}_S}(\tau^{\leq \nu}Q(\lambda)_S,E(\mu)_\k)
&=\Hom_{U^\hb_\zeta\Mod^{\Lambda}_S}(Q(\lambda)_S,E(\mu)_\k)\\ 
&=\Hom_{\fU^b_\zeta\Mod^{\Lambda}_S}(P^b(\lambda)_S,E(\mu)_\k)=\delta_{\lambda,\mu}\k,
\end{align*}
where the last equality is by Lemma \ref{lem simple} and the fact that $L^b(\lambda)_\k=L(\lambda^0)_\k \otimes \k_{l\lambda^1}$. 
\end{proof} 
As a consequence, for any $\lambda\leq \nu$, we have 
$$Q(-\rho+l\lambda)^{\leq -\rho+l\nu}_S= U^\hb_\zeta \otimes_{U^{\hb,\geq}_\zeta} \big((U\n/ \bigoplus_{\mu\nleq \nu-\lambda} (U\n)_{\mu}) \otimes S_{-\rho+l\lambda}\big).$$ 

\subsubsection{BGG reciprocity} 
Denote by $U^\hb_q\fMod^{\Lambda}_R$ the full subcategory of $U^\hb_q\Mod^{\Lambda}_R$ of the module $M$ such that $M_\mu$ is finitely generated projective $R$-module for each $\mu\in \Lambda$. 
We define ${}^{\Lambda}_R\rMod U^\hb_q$ to be the category of $\Lambda$-graded right $U^\hb_q\otimes R$-modules as in \textsection \ref{subsect ModCat}, and define its full subcategory ${}^{\Lambda}_R\Modf U^\hb_q$ similarly. 
Consider the functors 
$$-^{\circledast}:\ {}^{\Lambda}_R\Modf U^\hb_q \xs (U^\hb_q\fMod^{\Lambda}_R)^\op, \quad -^{\circledast}:\ U^\hb_q\fMod^{\Lambda}_R\xs ({}^{\Lambda}_R\Modf U^\hb_q)^\op,$$ 
given by $M^{\circledast}=\bigoplus_{\mu\in \Lambda} M^{\circledast}_\mu =\bigoplus_{\mu\in \Lambda} \Hom_R(M_\mu,R)$. 

We consider the right $U^\hb_q$-module $M^-(\lambda)_R:=R_\lambda \otimes_{\fU_q^{\leq}} U^\hb_q$, and define the \textit{coVerma module} by $M(\lambda)^\vee_R=M^-(\lambda)^{\circledast}_R$. 
It may not be contained in $\sO_{R}$. 
\begin{lem}\label{lem 3.3} 
We have 
$$\Ext^i_{U^\hb_q\Mod^{\Lambda}_R}(M(\lambda)_R, M(\mu)_R^\vee)= 
\begin{cases}
	\delta_{\lambda,\mu}R &  i=0 \\ 
	0 & i=1.
\end{cases}$$ 
If $R=\F$ is a field, then the unique (up to scalar) nonzero homomorphism $M(\lambda)_\F\rightarrow M(\lambda)_\F^\vee$ is factorized into 
$$M(\lambda)_\F\twoheadrightarrow E(\lambda)_\F \hookrightarrow M(\lambda)_\F^\vee.$$
\end{lem}
\begin{proof} 
We study the universal property for $M(\mu)_R^\vee$: for any $M \in U^\hb_q\fMod^{\Lambda}_R$, there are natural isomorphisms 
\begin{equation}\label{equ 3.-1} 
\begin{aligned}
\Hom(M,M(\mu)_R^\vee)&= \Hom(M^-(\mu)_R, M^\circledast) \\ 
&= \{f \in \Hom_R(M_\mu,R)\ |\ f((\fU_q^-)_+M)=0 \} = \Hom_R\big((M/(\fU_q^-)_+M)_\mu, R\big), 
\end{aligned} 
\end{equation}
where $(\fU_q^-)_+$ is the augmentation ideal for $\fU_q^-$. 
Note that $M(\lambda)_R/(\fU_q^-)_+M(\lambda)_R=R_\lambda$, hence $\Hom(M(\lambda)_R, M(\mu)_R^\vee)=\delta_{\lambda,\mu}R$. 
If $R=\F$ is a field, the surjection $M(\lambda)_\F\twoheadrightarrow E(\lambda)_\F$ induces an isomorphism 
$$M(\lambda)_\F/(\fU_q^-)_+M(\lambda)_\F\xs E(\lambda)_\F/(\fU_q^-)_+E(\lambda)_\F=\F_\lambda.$$ 
Applying the natural isomorphisms (\ref{equ 3.-1}), we get the desired factorization. 
The vanishing for $\Ext^1$ is proved similarly as in \cite[Thm 3.3(d)]{Hum08}. 
\end{proof}

\begin{prop}\label{prop BGG} 
Suppose $R$ is a local complete Noetherian domain, with residue field $\F$. 
For any $\nu \geq \mu \geq \lambda$, we have an equality 
$$(Q(\lambda)^{\leq \nu}_R: M(\mu)_R)=[M(\mu)_\F: E(\lambda)_\F].$$ 
\end{prop}
\begin{proof}
Note that $M(\mu)^\vee_\F$ is a union of the submodules contained in $\sO_{\F}$. 
By Lemmas \ref{lem 2.2} and \ref{lem 3.3}, we have equalities 
\begin{align*}
(Q(\lambda)^{\leq \mu}_R: M(\mu)_R)&=(Q(\lambda)^{\leq \mu}_\F: M(\mu)_\F)\\ 
&=\dim_\F \Hom(Q(\lambda)^{\leq \mu}_\F, M(\mu)_\F^\vee)=[M(\mu)^\vee_\F: E(\lambda)_\F]. 
\end{align*} 
Since $M(\mu)^\vee_\F$ and $M(\mu)_\F$ have the same characters, it implies that $[M(\mu)^\vee_\F: E(\lambda)_\F]=[M(\mu)_\F: E(\lambda)_\F]$. 
\end{proof} 

\subsection{Jantzen filtration} 
In this subsection, we define the Jantzen filtration for Verma modules in $\sO_\k$ and prove the Jantzen sum formula. 
A crucial point is to compute the Shapovalov's determinant formula for $U^\hb_q$. 
As an application, we deduce a semisimple criterion for $\sO_R$ when $R$ is the residue field of a prime ideal in $\hat{S}$. 
\subsubsection{Shapovalov's formula}\label{subsect 3.3.1} 
Consider the bilinear pairing on $U^\hb_q$ given by 
$$\{\ ,\ \}:\ U^\hb_q\otimes U^\hb_q\xrightarrow{\text{mult}} 
U^\hb_q =\mathfrak{U}^-_q \otimes \mathfrak{U}^0_q \otimes U^+_q \xrightarrow{\epsilon \otimes 1\otimes \epsilon} \mathfrak{U}^0_q.$$ 
It clearly satisfies 
$$\{u_1u,u_2\}=\{u_1,uu_2\}, \quad \forall u,u_1,u_2\in U^\hb_q.$$ 
Let $R$ be a deformation ring for $U^\hb_q$ with structure map $\pi$. 
Let $\lambda\in \Lambda$. 
Consider the composition 
$$U^\hb_q \otimes U^\hb_q\xrightarrow{\{\ ,\ \}} \mathfrak{U}^0_q \xrightarrow{\tau_\lambda} \mathfrak{U}^0_q \xrightarrow{\pi} R,$$ 
which extends $R$-linearly to a self-pairing $\{\ ,\ \}_{\lambda,R}$ on $U^\hb_q\otimes R$. 
One can show that it descends to a pairing 
$$\{\ ,\ \}_{\lambda,R}: M^-(\lambda)_R\otimes M(\lambda)_R \rightarrow R,$$ 
satisfying $\{m_1.u,m_2\}=\{m_1,u.m_2\}$, for any $u\in U^\hb_q, m_1\in M^-(\lambda)_R$ and $m_2\in M(\lambda)_R$. 
The restriction 
\begin{equation}\label{equ 3.7} 
\{\ ,\ \}_{\lambda,R}: M^-(\lambda)_{R,\lambda-\eta'}\otimes M(\lambda)_{R,\lambda-\eta} \rightarrow R 
\end{equation}
is nonzero only if $\eta'=\eta$. 
Therefore, the pairing $\{\ ,\ \}_{\lambda,R}$ yields a morphism in $U^\hb_q\Mod^\Lambda_R$, 
\begin{equation}\label{equ 3.100} 
\varpi_{\lambda,R}: M(\lambda)_R \rightarrow M(\lambda)_R^\vee. 
\end{equation} 

We choose $R$-basis $\{m^-_s\}_s$ and $\{m_s\}_s$ in $M^-(\lambda)_{R,\lambda-\eta}$ and $M(\lambda)_{R,\lambda-\eta}$, respectively. 
Define the determinant of the pairing (\ref{equ 3.7}) (for $\eta'=\eta$) by 
$${\det}_\eta(\lambda)_R:=\big| ( \{m^-_s,m_t\}_{\lambda,R})_{s,t} \big| \ \in R,$$ 
which is independent on the choice of $R$-basis up to an invertible scalar in $R$. 
We abbreviate $\det_\eta(\lambda)=\det_\eta(\lambda)_R$ if $R=\fU^0_q$, and $\det_\eta=\det_\eta(\lambda)$ if moreover $\lambda=0$. 
In general, we have $\det_\eta(\lambda)_R=(\pi\circ \tau_\lambda)(\det_\eta)$. 

We can consider similar constructions above when $U^\hb_q$ is replaced by De Concini--Kac quantum group $\fU_q$. 
In this case, we write $\det^{DK}_\eta$ for the determinant $\det_\eta$. 
According to \cite[Prop. 1.9]{DeCK90}, we have 
$${\det}^{DK}_\eta=\prod_{\beta \in \Phi^+}\prod_{m\geq 1} ( [m]_{q_\beta} 
[K_{\beta}; \langle\rho, \check{\beta}\rangle- m]_{q_\beta})^{|\Par(\eta-m\beta)|}.$$ 
Since $\{E^{(\wp)}\}_{\wp\in \Par(\eta)}$ forms a $\bA$-basis for $(U^+_{q})_{\eta}$ while $\{E^{\wp}\}_{\wp\in \Par(\eta)}$ forms a $\bA$-basis for $(\fU^+_q)_{\eta}$ (see \textsection\ref{subsect PBW}), 
we deduce that ${\det}_\eta$ is given by the \textit{Shapovalov's formula}, 
\begin{align*}
{\det}_\eta &= \big({\prod_{\wp\in \Par(\eta)} [\wp]!\big)^{-1}} 
\prod_{\beta \in \Phi^+}\prod_{m\geq 1} ( [m]_{q_\beta} 
[K_{\beta}; \langle\rho, \check{\beta}\rangle- m]_{q_\beta})^{|\Par(\eta-m\beta)|} \\ 
&= \prod_{\beta \in \Phi^+}\prod_{m\geq 1} ( 
[K_{\beta}; \langle\rho, \check{\beta}\rangle- m]_{q_\beta})^{|\Par(\eta-m\beta)|} , 
\end{align*}
where $[\wp]!:= \prod\limits_{\beta\in \Phi^+} [\wp(\beta)]_{q_\beta}!$. 
Hence $\det_\eta(\lambda)_R$ is the image of the following element in $R$, 
\begin{equation}\label{equ ShapFor} 
{\det}_\eta(\lambda)= {\tau_\lambda}({\det}_\eta)= \prod_{\beta \in \Phi^+}\prod_{m\geq 1} ([K_{\beta}; \langle \lambda+ \rho, \check{\beta}\rangle- m]_{q_\beta} )^{|\Par(\eta-m\beta)|} .
\end{equation} 

\subsubsection{Generic deformation} 
For a prime ideal $\p$ in $\hat{S}$, we denote its residue field by $\Bbbk_\p:=\hat{S}_\p/\p \hat{S}_\p$. 
We view $\check{\Phi}_{l,\af}$ as a subset of $\hat{S}\simeq \k[\t_\af^*]_{\widehat{0}}$. 
\begin{lem}\label{lem 3.6} 
Suppose $\p\cap \check{\Phi}_{l,\re}=\emptyset$, then the category $\sO_{\Bbbk_\p}$ is semisimple. 
In particular, the categories $\sO_{\K}$ and $\sO_{\hat{\K}}$ are semisimple. 
\end{lem} 
\begin{proof} 
By Lemma \ref{lem 3.3}, it is enough to show that $\varpi_{\lambda,\Bbbk_\p}:M(\lambda)_{\Bbbk_\p}\rightarrow M(\lambda)_{\Bbbk_\p}^\vee$ is an isomorphism, for any $\lambda\in \Lambda$. 
To that end, we show $\det_\eta(\lambda)$ is nonzero in $\Bbbk_\p$ for any $\lambda$ and $\eta$. 
Indeed, the element $[K_{\beta}; \langle \lambda+ \rho, \check{\beta}\rangle- m]_{q_\beta}$ is not invertible in $\hat{S}$ if and only if $l|\langle \lambda+ \rho, \check{\beta}\rangle- m$. 
In this case, we have 
\begin{equation}
[K_{\beta}; \langle \lambda+ \rho, \check{\beta}\rangle- m]_{q_\beta}= \check{\beta}+\frac{\langle \lambda+ \rho, \check{\beta}\rangle- m}{l} l\delta \ \in \check{\Phi}_{l,\re}, 
\end{equation}
up to an invertible scalar in $\hat{S}$. 
By our hypothesis that $\p\cap \check{\Phi}_{l,\re}=\emptyset$, the element $[K_{\beta}; \langle \lambda+ \rho, \check{\beta}\rangle- m]_{q_\beta}$ does not belong to $\p$. 
It finishes the proof. 
\end{proof} 

\subsubsection{Jantzen filtration} 
Consider the deformation ring $\k[[t]]$ for $U^\hb_\zeta$ with the structure map 
$$\pi:\ \fU^0_\zeta=\k[T]\xrightarrow{\check{\rho}} \k[\Gm]\hookrightarrow \k[\Gm]_{\hat{1}}=\k[[t]],$$ 
where $\check{\rho}$ is a dominant coweight for $T$. 
We abbreviate all the subscript $\k[[t]]$ by $t$, and set 
$$M^i(\lambda)_{t}:=\varpi_{\lambda,t}^{-1}(t^i M(\lambda)_{t}^\vee).$$ 
It gives an exhausting filtration $M(\lambda)_t=M^0(\lambda)_t\supset M^1(\lambda)_t \supset \cdots$, and we define the \textit{Jantzen filtration} for $M(\lambda)_\k$ by taking the images at $t=0$, 
$$M(\lambda)_\k=M^0(\lambda)_\k\supset M^1(\lambda)_\k \supset \cdots.$$ 
Note that $M^1(\lambda)_\k=\ker (\varpi_{\lambda,\k})$, which coincides with the maximal proper submodule of $M(\lambda)_\k$. 

For any $M\in U^\hb_q\fMod^\Lambda_R$, we define its \textit{character} by 
$$\ch M=\sum_{\lambda\in \Lambda} \rk_R(M_{\lambda})\cdot \exp(\lambda),$$ where $\exp(\lambda)$ is viewed as a formal variable labeled by $\lambda$. 
For example, the character for the Verma module $M(\lambda)_R$ is $\ch M(\lambda)_R=\sum_{\eta}|\Par(\eta)|\cdot \exp({\lambda-\eta})$. 

\begin{prop}[Jantzen sum formula]\label{prop JSF} 
There is an equality 
\begin{equation}\label{equ 3.9} 
\sum_{i>0}\ch M^i(\lambda)_\k=
\sum_{{\beta}\in \Phi^+}\sum_{\langle \lambda+ \rho, \check{\beta}\rangle+kl\geq 1}
\ch M(s_{\beta,kl}\bullet \lambda)_\k. 
\end{equation}
\end{prop}
\begin{proof}
Denote by $\val$ the valuation on $\k[[t]]$.
For any $\eta>0$, we have 
\begin{equation}\label{equ 3.10} 
\sum_{i>0} \dim M^i(\lambda)_{\k,\lambda-\eta}= \val \big(\pi({\det}_\eta(\lambda))\big). 
\end{equation}
Note that the element $\pi([K_{\beta}; \langle \lambda+ \rho, \check{\beta}\rangle- m]_{q_\beta})$ is not invertible in $\k[[t]]$ if and only if $\langle \lambda+ \rho, \check{\beta}\rangle- m=kl$ for some $k\in \Z$, and in this case its valuation is $1$. 
Hence by (\ref{equ 3.10}) and the formula (\ref{equ ShapFor}), we have 
\begin{align*}
\sum_{i>0}\ch M^i(\lambda)_\k &=\sum_{\eta} \sum_{\beta} \sum_{\langle \lambda+ \rho, \check{\beta}\rangle-kl\geq 1} |\Par\big(\eta-(\langle \lambda+ \rho, \check{\beta}\rangle-kl)\beta\big)|\cdot \exp({\lambda-\eta})\\ 
&=\sum_{\beta} \sum_{\langle \lambda+ \rho, \check{\beta}\rangle-kl\geq 1} \sum_{\eta} 
|\Par(\eta)|\cdot \exp({\lambda-(\langle \lambda+ \rho, \check{\beta}\rangle-kl)\beta -\eta})\\ 
&=\sum_{\beta} \sum_{\langle \lambda+ \rho, \check{\beta}\rangle-kl\geq 1} \ch M\big({\lambda-(\langle \lambda+ \rho, \check{\beta}\rangle-kl)\beta}\big)_\k. \qedhere 
\end{align*} 
\end{proof} 

\subsection{Linkage principle and block decomposition} 
Consider the partial order $\uparrow$ in $\Lambda$ generated by (see \cite[II \textsection 6]{Jan03}) 
$$s_{\alpha}\bullet \lambda \uparrow \lambda \quad \text{if}\ s_{\alpha}\bullet \lambda\leq \lambda ,\quad \lambda\in \Lambda,\ \alpha\in \check{\Phi}_{l,\re}.$$ 
\begin{prop}[Linkage principle]\label{prop LP}  
\begin{enumerate} 
\item We have 
\begin{equation}\label{equ 3.3} 
[M(\lambda)_\k:E(\mu)_\k]\neq 0 \quad \text{if and only if}\quad \mu \uparrow \lambda.
\end{equation} 
\item For any commutative Noetherian $\hat{S}$-algebra $R$, there is a block decomposition of the category $\sO_{R}$, 
\begin{equation}\label{equ 3.5} 
\sO_{R}= \bigoplus_{\omega\in \Xi_\sc} \sO^{\omega}_{R}, 
\end{equation} 
such that the Verma module $M(\lambda)_R$ is contained in $\sO^{\omega}_{R}$ if and only if $\lambda\in W_{l,\af}\bullet \omega$. 
When $R=\k$, the blocks $\sO^{\omega}_{\k}$ in (\ref{equ 3.5}) are indecomposable. 
\end{enumerate} 
\end{prop}
\begin{proof} 
Our argument is adopted from \cite[Proof of Thm 2]{KK79}. 

(1) Note that both $[M(\lambda)_\k:E(\mu)_\k]\neq 0$ and $\mu \uparrow \lambda$ implies $\lambda \geq \mu$. 
We do increase induction on the height of $\lambda-\mu$. 
The case $\lambda=\mu$ is trivial. 
Suppose $\lambda>\mu$, then $[M(\lambda)_\k:E(\mu)_\k]\neq 0$ if and only if $[M^1(\lambda)_\k: E(\mu)_\k]\neq 0$. 
By the formula (\ref{equ 3.9}), it is equivalent to $[M(s_{\beta,kl}\bullet \lambda)_\k:E(\mu)_\k]\neq 0$ for some $\beta\in \Phi^+$ and $k\in \Z$ such that $\langle \lambda+ \rho, \check{\beta}\rangle+kl\geq 1$. 
Since $s_{\beta,kl}\bullet \lambda\uparrow \lambda$ and $s_{\beta,kl}\bullet \lambda< \lambda$, by induction hypothesis, it is equivalent to $\mu \uparrow s_{\beta,kl}\bullet \lambda$ for some such $\beta$ and $k$, namely, $\mu\uparrow \lambda$. 

(2) By (1) and Proposition \ref{prop BGG}, it follows that the Verma factors of $Q(\lambda)^{\leq \nu}_{\k}$ are of the form $M(\mu)_{\k}$ for some $\mu\in W_{l,\af}\bullet\lambda$. 
Hence by (1) again, $\Hom(Q(\lambda)^{\leq \nu}_{\k},Q(\lambda')^{\leq \nu'}_{\k})\neq 0$ only if $\lambda'\in W_{l,\af}\bullet\lambda$. 
It yields a block decomposition 
\begin{equation}\label{equ 3.2} 
\sO_{\k}= \bigoplus_{\omega\in \Xi_\sc} \sO^{\omega}_{\k}, 
\end{equation} 
such that $Q(\lambda)^{\leq \nu}_\k$ is contained in $\sO^{\omega}_{\k}$ if and only if $\lambda\in W_{l,\af}\bullet \omega$. 
Note that the $W_{\l,\af}$-orbits of $\Lambda$ are the connected components of $\Lambda$ given by the partial order $\uparrow$. 
By (\ref{equ 3.3}), we deduce that $\sO^{\omega}_{\k}$ is indecomposable. 
Since $\hat{S}$ is local, (\ref{equ 3.2}) lifts to a block decomposition for $\sO_{\hat{S}}$ and then extends to (\ref{equ 3.5}). 
\end{proof} 

We will abbreviate $\sO^{\omega,\leq \nu}_{R}=\sO^{\omega}_{R}\cap \sO^{\leq \nu}_{R}$.

\subsection{Translation functors} 
Let $R$ be a commutative Noetherian $\hat{S}$-algebra. 
For $\omega_1,\omega_2 \in \Xi_\sc$, there is a unique dominant element $\nu$ in $W(\omega_2-\omega_1)$. 
Denote by $V(\nu)_q$ the Weyl module for $U_q$ of highest weight $\nu$. 
We define the \textit{translation functors} by 
$$T_{\omega_1}^{\omega_2} : \sO^{ \omega_1}_{R} \rightarrow \sO^{ \omega_2}_{R}, \quad M \mapsto \pr_{\omega_2} (M \otimes V(\nu)_q),$$ 
$$T^{\omega_1}_{\omega_2} : \sO^{ \omega_2}_{R} \rightarrow \sO^{ \omega_1}_{R}, \quad M \mapsto \pr_{\omega_1} (M \otimes V(\nu)^*_q),$$ 
where $\pr_{\omega_i}$ is the natural projection to the block $\sO^{ \omega_i}_{R}$. 
\begin{lem}\cite[II \textsection 7.8]{Jan03} \label{lem 5.0} 
\begin{enumerate}
\item $T_{\omega_1}^{\omega_2}$ and $T^{\omega_1}_{\omega_2}$ are exact and biadjoint to each other. 
\item For any $x\in W_{l,\af}$, the module $T_{\omega_1}^{\omega_2}M(x\bullet \omega_1)_R$ admits Verma factors $M(xy\bullet \omega_2)_R$, where $y$ runs through a system of representatives for 
$$W_{l,\omega_1}/W_{l,\omega_1}\cap W_{l,\omega_2},$$ 
each of which appears once. 
\end{enumerate}
\end{lem} 
\begin{proof}
(1) The isomorphism $V(\nu)_q\xs V(\nu)^{**}_q$ given by $K_{2\rho}$ yields a biadjunction for $-\otimes V(\nu)_q$ and $-\otimes V(\nu)^*_q$, which induces a biadjunction for $T_{\omega_1}^{\omega_2}$ and $T^{\omega_1}_{\omega_2}$. 

(2) It is proved by the same arguments in \cite[II \textsection 7.8]{Jan03}. 
\end{proof}

Suppose that $\omega_2$ is contained in the closure of the $\omega_1$-facet, i.e. $W_{l,\omega_1}\subseteq W_{l,\omega_2}$. 
Then by Lemma \ref{lem 5.0}(2) we have $T_{\omega_1}^{\omega_2}T^{\omega_1}_{\omega_2} M(\lambda)_R=M(\lambda)_R^{\oplus |W_{l,\omega_2}/W_{l,\omega_1}|}$ for any $\lambda\in W_{l,\af}\bullet \omega_2$, since $M(\lambda)_R$ only has trivial self-extension. 
It shows that $T_{\omega_1}^{\omega_2}T^{\omega_1}_{\omega_2}$ is by multiplying $|W_{l,\omega_2}/W_{l,\omega_1}|$ on the Grothendieck group of $\sO^{\omega_2}_{R}$. 
In fact, we have the following proposition, whose proof will be given in Appendix \ref{App C}. 
\begin{prop}\label{prop 5.0} 
For any $\omega_1,\omega_2\in \Xi_\sc$ such that $W_{l,\omega_1}\subseteq W_{l,\omega_2}$, there is a natural isomorphism 
$$\Upsilon_{\omega_2}^{\omega_1}:\ \id^{\oplus |W_{l,\omega_2}/W_{l,\omega_1}|} \xs T_{\omega_1}^{\omega_2}T^{\omega_1}_{\omega_2}$$ 
of functors on $\sO^{\omega_2}_{R}$. 
\end{prop} 

\subsubsection{The $l\Lambda$-symmetry} 
For any $\lambda\in \Lambda$, we denote by $[\lambda]$ its class in $\Xi=\Lambda/(W_{l,\ex},\bullet)$. 
There is a decomposition 
\begin{equation}\label{equ 3.13} 
\sO_{R}=\bigoplus_{\omega\in \Xi}\sO^{[\omega]}_{R}, \quad \text{where}\quad \sO^{[\omega]}_{R}=\bigoplus\limits_{[\omega']=[\omega]}\sO^{\omega'}_{R}. 
\end{equation} 
Now we let $R$ be a commutative Noetherian $S$-algebra. 
The trivial $U_\zeta^\hb$-module $\k_{l\nu}$ yields an auto-equivalence of $\sO_{R}$ by 
$$-\otimes \k_{l\nu}: \sO_{R}\xs \sO_{R}.$$ 
If $\omega_1 ,\omega_2 \in \Xi_\sc$ satisfy $[\omega_1]=[\omega_2]$, then there are $\lambda_i\in W_{l,\af}\bullet \omega_i$ ($i=1,2$) such that $\lambda_1-\lambda_2\in l\Lambda$. 
It gives an equivalence $$-\otimes \k_{\lambda_2-\lambda_1}: \sO^{\omega_1}_{R}\xs \sO^{\omega_2}_{R}.$$ 
Therefore $\sO^{[\omega]}_{R}\simeq (\sO^{\omega}_{R})^{\oplus \pi_1}$.

\section{Deformed category $\sO$}\label{sect 6} 
Throughout the section, we fix $\alpha\in {\Phi}^+$. 
Let $S_{\alpha} :=S_{(\alpha)}$ be the localization by the prime ideal $(\alpha)$, with residue field ${\Bbbk_\alpha}:= S_\alpha/\alpha S_\alpha$. 
We study the category $\sO_{S_\alpha}$, and construct the endomorphism algebra of the ``big projective module" in $\sO_{S_\alpha}$, which will help us to compare the center and the cohomology ring. 

\subsection{Block decomposition} 
Let $W_l(\alpha)$ be the subgroup of $W_{l,\af}$ generated by $s_{\alpha,nl}$, $n\in \Z$. 
Define a partial order $\uparrow_\alpha$ on $\Lambda$ generated by  
$$s_{\alpha,nl}\bullet \lambda \uparrow_\alpha \lambda \quad \text{if}\ s_{\alpha,nl}\bullet \lambda\leq \lambda ,\quad \lambda\in \Lambda.$$ 
Recall the morphism $\varpi_{\lambda,S_\alpha}:M(\lambda)_{S_\alpha}\rightarrow M(\lambda)_{S_\alpha}^\vee$ defined in (\ref{equ 3.100}). 
We define 
$$M^i(\lambda)_{S_\alpha}:=\varpi_{\lambda,S_\alpha}^{-1}(\alpha^i M(\lambda)_{S_\alpha}^\vee), \quad i\geq 0,$$ 
which form an exhausting filtration $M(\lambda)_{S_\alpha}=M^0(\lambda)_{S_\alpha}\supset M^1(\lambda)_{S_\alpha}\supset \cdots$. 
We define the \textit{Jantzen filtration} of $M(\lambda)_{{\Bbbk_\alpha}}$ by taking the image over ${\Bbbk_\alpha}$, 
\begin{equation}\label{equ 6.1} 
M(\lambda)_{{\Bbbk_\alpha}}=M^0(\lambda)_{{\Bbbk_\alpha}}\supset M^1(\lambda)_{{\Bbbk_\alpha}}\supset \cdots. 
\end{equation} 

\begin{prop}\label{prop 6.1} 
\begin{enumerate}
\item The following \textit{Jantzen sum formula} holds 
\begin{equation}\label{equ 6.2} 
\sum_{i>0}\ch M^i(\lambda)_{{\Bbbk_\alpha}}= 
\sum_{\langle \lambda+ \rho, \check{\alpha}\rangle+kl\geq 1}
\ch M(s_{\alpha,kl}\bullet \lambda)_{{\Bbbk_\alpha}}. 
\end{equation} 
\item We have linkage principle: 
\begin{equation}\label{equ 6.0} 
[M(\lambda)_{{\Bbbk_\alpha}}:E(\mu)_{{\Bbbk_\alpha}}]\neq 0 \quad \text{if and only if}\quad \mu \uparrow_\alpha \lambda.
\end{equation} 
\item There is a block decomposition 
\begin{equation}\label{equ 6.3}
\sO_{S_\alpha} = \bigoplus_{\sX\in (W_{l}(\alpha),\bullet) \backslash \Lambda} \sO^{\sX}_{S_\alpha}, 
\end{equation} 
such that the Verma module $M(\lambda)_{S_\alpha}\in \sO^{\sX}_{S_\alpha}$ if and only if $\lambda\in \sX$. 
\end{enumerate}
\end{prop}
\begin{proof}
It is proved similarly as in Propositions \ref{prop JSF} and \ref{prop LP}. 
\end{proof} 

Let $\omega\in \Lambda$. 
The equivalence class of the block $\sO^{\sX}_{S_\alpha}$ containing $M(\omega)_{S_\alpha}$ depends on whether 
\begin{itemize} 
	\item \textit{$\omega$ is $\alpha$-singular}, that is, $l\ |\ \langle \omega+\rho, \check{\alpha} \rangle$; 
	\item \textit{$\omega$ is $\alpha$-regular}, that is, $l \nmid \langle \omega+\rho , \check{\alpha} \rangle$. 
\end{itemize} 
We set $\sX^\omega_\alpha:=W_{l}(\alpha)\bullet \omega$, then 
$$\sX^\omega_\alpha = 
\begin{cases}
	(\omega +l\Z \alpha)\bigsqcup (s_\alpha\bullet \omega +l\Z \alpha) & \text{if $\omega$ is $\alpha$-regular} \\ 
	\omega +l\Z \alpha & \text{if $\omega$ is $\alpha$-singular}. 
\end{cases}$$ 
We denote by $\pr_{\sX^\omega_\alpha}$ the natural projection to the block $\sO^{\sX^{\omega}_\alpha}_{S_\alpha}$. 

\subsection{The $\alpha$-singular case}\label{subsect 6.1.1} 
In this subsection, we study the blocks $\sO^{\sX^{\omega}_\alpha}_{S_\alpha}$ for $\alpha$-singular weights $\omega$. 
We only consider the case $\omega=-\rho$, without loss of generality. 
For simplicity, we abbreviate 
$$\bm{n}_\alpha=-\rho + ln\alpha \in \Lambda, \quad n\in \Z.$$ 
\subsubsection{Projective modules} 
\begin{lem}\label{lem 6.2} 
The Jantzen filtration (\ref{equ 6.1}) for $M(\bm{n}_\alpha)_{{\Bbbk_\alpha}}$ coincides with 
$$M(\bm{n}_\alpha)_{{\Bbbk_\alpha}}\supset M(\bm{n-1}_\alpha)_{{\Bbbk_\alpha}}\supset M(\bm{n-2}_\alpha)_{{\Bbbk_\alpha}}\supset \cdots.$$ 
In particular, there is a short exact sequence 
$$0\rightarrow M(\bm{n-1}_\alpha)_{{\Bbbk_\alpha}} \rightarrow M(\bm{n}_\alpha)_{{\Bbbk_\alpha}} \rightarrow E(\bm{n}_\alpha)_{{\Bbbk_\alpha}} \rightarrow 0. $$ 
\end{lem} 
\noindent 
To show this lemma, we need the following claim, whose proof is technical and is postponed to \textsection \ref{subsect 4.1.0}. 

\begin{claim}\label{claim 6.3} 
For any $n\geq 1$, there is a homogeneous element $F_n(\alpha)\in Z_\Fr^-\otimes S_\alpha\backslash \alpha (Z_\Fr^-\otimes S_\alpha)$ of degree $-nl\alpha$, such that 
\begin{equation}\label{equ 6.5} 
\{U^\hb_\zeta\otimes S_\alpha, F_n(\alpha)\}_{\lambda,S_\alpha}\subset \alpha^nS_\alpha, \quad \forall \lambda\in \Lambda. 
\end{equation} 
\end{claim}

\begin{proof}[Proof of the Lemma \ref{lem 6.2}] 
Choose a highest weight generator $1_{\bm{n}_\alpha}$ for $M(\bm{n}_\alpha)_{S_\alpha}$. 
By Claim \ref{claim 6.3}, for any $i\geq 1$, we have an element $F_i(\alpha).1_{\bm{n}_\alpha}\in M(\bm{n}_\alpha)_{S_\alpha}\backslash \alpha M(\bm{n}_\alpha)_{S_\alpha}$ such that 
$$\{M^-(\bm{n}_\alpha)_{S_\alpha},F_i(\alpha).1_{\bm{n}_\alpha}\}_{\bm{n}_\alpha,S_\alpha}\subset \alpha^iS_\alpha.$$ 
Hence $F_i(\alpha).1_{\bm{n}_\alpha}\in M^i(\bm{n}_\alpha)_{S_\alpha}\backslash \alpha M(\bm{n}_\alpha)_{S_\alpha}$ and its image in $M^i(\bm{n}_\alpha)_{{\Bbbk_\alpha}}$ is nonzero, which shows that $\dim M^i(\bm{n}_\alpha)_{{\Bbbk_\alpha},\bm{n-i}_\alpha}\geq 1$. 
Since $M(\bm{n}_\alpha)_{{\Bbbk_\alpha}}$ is a torsion free $\fU^-_\zeta\otimes {\Bbbk_\alpha}$-module, we have 
$$M(\bm{n-i}_\alpha)_{{\Bbbk_\alpha}}\cong (\fU^-_\zeta\otimes {\Bbbk_\alpha}) F_i(\alpha).1_{\bm{n}_\alpha} \subset M^i(\bm{n}_\alpha)_{{\Bbbk_\alpha}}$$ 
as modules in $\fU^-_\zeta\Mod^\Lambda_{{\Bbbk_\alpha}}$. 
Therefore, we have 
\begin{equation}\label{equ 6.7}
\ch M^i(\bm{n}_\alpha)_{{\Bbbk_\alpha}}-\ch M(\bm{n-i}_\alpha)_{{\Bbbk_\alpha}}\in \prod_{\lambda\in \Lambda} \Z_{\geq 0}\exp(\lambda), \quad \forall i\geq 1. 
\end{equation}
Applying the Jantzen sum formula (\ref{equ 6.2}) for $M(\bm{n}_\alpha)_{{\Bbbk_\alpha}}$, we get 
$$\sum_{i>0} \ch M^i(\bm{n}_\alpha)_{{\Bbbk_\alpha}}=\sum_{i>0} \ch M(\bm{n-i}_\alpha)_{{\Bbbk_\alpha}},$$ 
which forces the LHS of (\ref{equ 6.7}) to vanish. 
Again, since $M^i(\bm{n}_\alpha)_{{\Bbbk_\alpha}}$ is a torsion free $\fU^-_\zeta\otimes {\Bbbk_\alpha}$-module, it follows that $M^i(\bm{n}_\alpha)_{{\Bbbk_\alpha}}= M(\bm{n-i}_\alpha)_{{\Bbbk_\alpha}}$ as modules in $U^\hb_\zeta\Mod_{{\Bbbk_\alpha}}^\Lambda$. 
The short exact sequence follows from the fact that $M^1(\bm{n}_\alpha)_{{\Bbbk_\alpha}}$ coincides with the maximal proper submodule of $M(\bm{n}_\alpha)_{{\Bbbk_\alpha}}$. 
\end{proof}

\begin{corollary}\label{cor 6.4} 
The module $Q(\bm{n}_\alpha)^{\leq \bm{m}_\alpha}_{S_\alpha}$ admits a Verma flag with composition factors 
\begin{equation}\label{equ 6.8} 
M(\bm{n}_\alpha)_{S_\alpha}, M(\bm{n+1}_\alpha)_{S_\alpha}, \cdots, M(\bm{m}_\alpha)_{S_\alpha}, 
\end{equation}
each of which appears once. 
\end{corollary} 
\begin{proof}
From Lemma \ref{lem 6.2}, we deduce that $[M(\bm{n}_\alpha)_{{\Bbbk_\alpha}}:E(\bm{n'}_\alpha)_{{\Bbbk_\alpha}}]=1$ for any $n'\leq n$. 
The assertion follows from Proposition \ref{prop BGG}. 
\end{proof}

\subsubsection{Proof of Claim \ref{claim 6.3}}\label{subsect 4.1.0} 
We begin with a lemma. 
\begin{lem}\label{lem 4.2} 
\begin{enumerate}
\item For any $i\in \I$, the adjoint operator $[E_{i}^{(l)}, -]$ on $U^\hb_\zeta$ preserves the subalgebra $Z_\Fr^{\leq}$. 
\item It induces a $U\n^-$-action on $Z_\Fr^{\leq}$ such that $f_{\alpha_i}$ acts via $-K_i^l [E_{i}^{(l)}, -]$ for any $i\in \I$. 
There is a $U\n^-$-isomorphism 
\begin{equation}\label{equ 4.1} 
Z_\Fr^{\leq} \xs \k[B^-\times_T T], 
\end{equation} 
compatible with (\ref{equ 2.7}), where the base change $T\rightarrow T$ is by $t\mapsto t^{2}$ for any $t\in T$, and the $U\n^-$-action on $\k[B^-\times_T T]$ is induced from the $N^-$-conjugation on $B^-$. 
\end{enumerate} 
\end{lem} 
\begin{proof} 
(1) Since $[E_{i}^{(l)},-]$ is trivial on $Z^0_\Fr$, 
it remains to show that $[E_{i}^{(l)}, F_\beta^l]\in Z_\Fr^{\leq}$ for any $\beta\in \Phi^+$. 
By \cite[\textsection 3.4]{DeCK90}, the operator $\underline{X}_{i}:=[E_{i}^{(l)},-]$ on $U_q$ preserves the subspace $\fU_q$, and its specialization on $\fU_\zeta$ preserves $Z_\Fr$. 
It follows that in the integral form $U_q^\hb$, we have 
	$$ [E_{i}^{(l)}, F_\beta^l] \in \sum_{\wp, \wp'\in (l\N)^{\Phi^+}, \lambda\in \Lambda} \k \ F^\wp K^l_\lambda E^{\wp'} \quad \mathrm{mod} \ (q_e-\zeta_e)\fU_q .$$ 
Hence in $U^\hb_\zeta$, we have $[E_{i}^{(l)}, F_\beta^l]\in Z_\Fr^{\leq}$. 
	
(2) The derivation $\underline{X}_{i}$ on $Z_\Fr$ yields a tangent field $\underline{X}_{i} \in \Gamma(G^*,\sT_{G^*})$ via the isomorphism (\ref{equ 2.7}). 
Consider the unramified covering from $G^*$ onto the big open cell $N^-TN$, 
$$\kappa: G^* \rightarrow G, \quad (n_-,t ,n_+)\mapsto n_- t^2 n_+^{-1}, \quad \forall n_{-}\in N^-,\ n_+\in N,\ t\in T,$$ 
which gives a pullback of tangent fields $\kappa^*: \Gamma(G,\sT_{G})\rightarrow \Gamma(G^*,\sT_{G^*})$. 
A remarkable theorem given by De Concini--Kac--Procesi in {\cite[\textsection 5]{DeCKP92}} shows that 
\begin{equation}\label{equ 4.9re} 
\kappa^*(f_{\alpha_i})= -K_i^l \underline{X}_{i},
\end{equation}
where $f_{\alpha_i}$ is viewed as a Killing vector field on $G$ (i.e. the vector fields induced by the conjugate action). 
Denote by $(Z_\Fr^+)_+$ the augmentation ideal of $Z_\Fr^+$. 
Since $\underline{X}_{i}$ on $Z_\Fr$ preserves $(Z_\Fr^+)_+$, it descents to an operator on $Z_\Fr^{\leq}= Z_\Fr/ Z_\Fr(Z_\Fr^+)_+$, which coincides with the action $[E_{i}^{(l)},-]$ on $Z_\Fr^{\leq}$ in the algebra $U^\hb_\zeta$. 

Consider the Cartesian diagram 
$$\begin{tikzcd}
\Spec Z_\Fr^{\leq} \arrow[r,two heads, "\kappa'"] \arrow[d,hook] & B^- \arrow[d,hook] \\ 
\Spec Z_\Fr \arrow[r,two heads,"\kappa"] & N^- T N.
\end{tikzcd}$$ 
By the definition of $\kappa$ and (\ref{equ 2.7}), $\kappa$ is compatible with the isomorphism $\Spec Z_\Fr\simeq (N^- T N) \times_T T$, where $T\rightarrow T$ is by $t\mapsto t^{2}$ for any $t\in T$. 
Hence $\kappa'$ is compatible with the isomorphism 
\begin{equation}\label{equ 4.10re}
Z_\Fr^{\leq} \xs \k[B^-\times_T T]. 
\end{equation}
Note that the Killing vector field $f_{\alpha_i}$ on $G$ is tangent to $B^-$. 
By (\ref{equ 4.9re}), the map $f_{\alpha_i}\mapsto -K_i^l [E_{i}^{(l)}, -]$ defines a $U\n^{-}$-action on $Z_\Fr^{\leq}$ such that (\ref{equ 4.10re}) is an isomorphism of $U\n^{-}$-algebras. 
\end{proof} 

Following Lemma \ref{lem 4.2}(2), we consider the algebra $A=\k[B^-\times_T T]\rtimes U\n^-$. 
As in \textsection \ref{subsect 3.3.1}, there is a bilinear pairing 
$$\{\ ,\ \}':\ A\otimes A\xrightarrow{\text{mult}} A=\k[N^-]\otimes \k[T] \otimes U\n^- \xrightarrow{\mathrm{ev_1} \otimes 1\otimes \epsilon_1} \k[T],$$ 
where $\mathrm{ev_1}$ is by evaluation at $1\in N^-$, and $\epsilon_1$ is the counit of $U\n^-$. 
The pairing kills the following subspace of $A\otimes A$, 
$$\big(\ker(\ev_1)\otimes U\n^-\big)\otimes A+A\otimes \big(\k[B^-\times_T T]\otimes \ker\epsilon_1\big),$$ 
hence it descents to a pairing 
$$\{\ ,\ \}':\ (\k[T] \otimes U\n^-)\otimes \k[B^-\times_T T]\rightarrow \k[T].$$ 
The restriction on the weight space $(U\n^-)_{\eta'} \otimes \k[N^-]_{\eta}$ is nonzero only if $\eta'=-\eta$. 
By the exponential map $\b^-\rightarrow B^-$, we can identify the algebras endowed with $B^-$-actions 
$$\k[N^-]\otimes \k[T]_{\widehat{1}} \xs \k[\n^-]\otimes \k[\t]_{\widehat{0}}.$$ 
Let $\{\varphi_\beta\}_{\beta}$ be the dual basis of $\{f_\beta\}_{\beta}$ in $\n^-$, then $\k[\n^-]$ is a polynomial ring generated by $\{\varphi_\beta\}_{\beta}$. 

\begin{proof}[Proof of Claim \ref{claim 6.3}] 
The pairing $\{\ ,\ \}'$ extends $S_\alpha$-linearly to a pairing $\{\ ,\ \}'_\alpha$ on $A\otimes_{\k[T]}S_\alpha$. 
To prove the assertion, we firstly show a similar one ------ there exists a homogeneous element $\varphi_n(\alpha)$ in $(\k[N^-]\otimes S_\alpha) \backslash \alpha(\k[N^-]\otimes S_\alpha)$ of degree $n\alpha$, such that 
$$\{A\otimes S_\alpha, \varphi_n(\alpha)\}'_\alpha \subset \alpha^n S_\alpha.$$ 
By degree consideration, it is enough to consider the restriction of $\{\ ,\ \}'_\alpha$ on $(U\n^-)_{-n\alpha}\otimes S_\alpha$ and $\k[N^-]_{n\alpha}\otimes S_\alpha$. 
Note that $S_\alpha$ is a discrete valuation ring, and the Smith normal form of the matrix associated with this $S_\alpha$-linear pairing is equal to the one of the transposed matrix. 
Hence if we show that 
\begin{equation}\label{equ 4.10} 
\{f^n_\alpha,\k[N^-]_{n\alpha}\otimes S_\alpha\}'_\alpha\subset \alpha^n S_\alpha, 
\end{equation}
and note that $f^n_\alpha\notin \alpha(U\n^-\otimes S_\alpha)$, then the existence of $\varphi_n(\alpha)$ follows. 
To that end, we compute $\{f^n_\alpha,\prod\limits_{\beta\in \Phi^+}\varphi^{\wp(\beta)}_{\beta}\}'_\alpha$ for any $\wp\in \Par(n\alpha)$. 
Indeed, by definition we have 
\begin{equation}\label{equ 6.4} 
\begin{aligned} 
\{f^n_\alpha,\prod\varphi^{\wp(\beta)}_{\beta}\}'_\alpha
&=(\mathrm{ev_1} \otimes 1\otimes \epsilon_1)(f^n_\alpha\prod\varphi^{\wp(\beta)}_{\beta})\\ 
&=\sum_{\wp(\gamma)\geq 1} (\mathrm{ev_1} \otimes 1\otimes \epsilon_1)
\big(f^{n-1}_\alpha[f_\alpha,\varphi_{\gamma}]\prod\varphi^{(\wp\backslash \gamma)(\beta)}_{\beta}
+f^{n-1}_\alpha\prod\varphi^{\wp(\beta)}_{\beta}f_\alpha\big) \\ 
&=\sum_{\wp(\gamma)\geq 1} \{f^{n-1}_\alpha,[f_\alpha,\varphi_{\gamma}]\prod\varphi^{(\wp\backslash \gamma)(\beta)}_{\beta}\}'_\alpha, 
\end{aligned}
\end{equation} 
where $\wp\backslash \gamma$ is the Kostant partition obtain from $\wp$ by decreasing $\wp(\gamma)$ by $1$. 
By computation, we have $[f_\alpha,\varphi_\alpha]\in \alpha S_\alpha$, and for $\beta\neq \alpha$, $[f_\alpha,\varphi_{\beta}]\in \k \varphi_{\beta-\alpha}$ where $\varphi_{\beta-\alpha}:=0$ if $\beta-\alpha\notin \Phi^+$               . 
One can show by increase induction on $n$ that (\ref{equ 6.4}) vanishes unless $\wp(\alpha)=n$, and in this case $\{f^n_\alpha,\varphi^n_\alpha\}_\alpha\in \alpha^n S_\alpha$, which implies (\ref{equ 4.10}). 

Let $F_n(\alpha)\in Z_\Fr^{-}\otimes S_\alpha\backslash \alpha(Z_\Fr^{-}\otimes S_\alpha)$ be the element corresponding to $\varphi_n(\alpha)$ under the isomorphism (\ref{equ 4.1}). 
We show that $F_n(\alpha)$ has the desired property. 
Consider the elements $\prod_{j}E^{(m_{i_j})}_{{i_j}}$, which generate $U^{\hb,\geq}_\zeta$ as a $\fU^0_\zeta$-module. 
Since $E_i$ commutes with $Z^{\leq}_\Fr$ for any $i\in \rI$, we have 
\begin{equation}\label{equ 4.11} 
\big\{\prod_{j}E^{(m_{i_j})}_{i_j}, F_n(\alpha)\big\}_{\lambda, S_\alpha}
\end{equation}
is nonzero only if $l|m_{i_j}$ for any $j$.  
In this case, by Lemma \ref{lem 4.2} and since $\tau_\lambda$ acts trivially on $Z_\Fr^0$, the element (\ref{equ 4.11}) is equal to 
$$\prod_j K^{m_{i_j}}_{i_j}\cdot 
\big\{\prod_{j}\frac{-f^{m_{i_j}/l}_{\alpha_{i_j}}}{(m_{i_j}/l)!} ,\varphi_n(\alpha) \big\}'_\alpha,$$ 
which belongs to $\alpha^nS_\alpha$. 
\end{proof} 

\subsubsection{Endomorphism algebra of the ``big projective object"} 

\begin{claim}\label{claim 4.6*} 
There is a homogeneous element $E(\alpha)\in (U^+_{\zeta}\otimes S_\alpha)\backslash \alpha(U^+_{\zeta}\otimes S_\alpha)$ of degree $l\alpha$, such that 
$$\{E(\alpha),\fU^-_{\zeta} \otimes S_\alpha\}_{\bm{0}_\alpha,S_\alpha}\subset \alpha S_\alpha. $$ 
\end{claim}
\begin{proof}
Consider the $S_\alpha$-linear pairing 
$$\{\ ,\ \}_{\bm{0}_\alpha,S_\alpha}: (U^+_{\zeta,l\alpha}\otimes S_\alpha)\otimes_{S_\alpha} (\fU^-_{\zeta,-l\alpha} \otimes S_\alpha) \rightarrow S_\alpha.$$ 
Since the Smith normal form of the matrix associated with this $S_\alpha$-linear pairing is equal to the one of the transposed matrix, the existence of $E(\alpha)$ follows from Claim \ref{claim 6.3}. 
\end{proof}

By Claim \ref{claim 4.6*}, we have $E(\alpha).M(\bm{0}_\alpha)_{S_\alpha,\bm{-1}_\alpha}\subset \alpha M(\bm{0}_\alpha)_{S_\alpha,\bm{0}_\alpha}$, and hence $E(\alpha).M(\bm{0}_\alpha)_{{\Bbbk_\alpha},\bm{-1}_\alpha}$ $=0$. 
In general, since $M(\bm{n}_\alpha)_{{\Bbbk_\alpha}}=M(\bm{0}_\alpha)_{{\Bbbk_\alpha}}\otimes \k_{ln\alpha}$, we have 
\begin{equation}\label{equ 6.9} 
E(\alpha).M(\bm{n}_\alpha)_{{\Bbbk_\alpha},\bm{n-1}_\alpha}=0, \quad \forall n\in \Z. 
\end{equation} 

For any $S_\alpha$-algebra $R$, we define the following projective module in $\sO^{\leq \bm{m}_\alpha}_{R}$ as in (\ref{equ 2.9}), 
$$Q'(\bm{n}_\alpha)^{\leq \bm{m}_\alpha}_{R}:= 
U^\hb_\zeta \otimes_{U^{\hb,\geq}_\zeta}\big((U^+_\zeta/\bigoplus_{\nu\nleq (m-n)l\alpha}U^+_{\zeta,\nu} ) \otimes R_{\bm{n}_\alpha}\big).$$ 
It is a cyclic $U^\hb_\zeta\otimes R$-module generated by the element $1_\zeta \otimes 1_+\otimes 1_{\bm{n}_\alpha}$, where $1_\zeta, 1_+, 1_{\bm{n}_\alpha}$ are the identities of $U^\hb_\zeta$, $U^+_\zeta$ and $R_{\bm{n}_\alpha}$. 
Note that $Q'(\bm{n}_\alpha)^{\leq \bm{m}_\alpha}_{S_\alpha}$ admits Verma factors of the form $M(\bm{n}_\alpha+\eta)_{S_\alpha}$ with $0\leq \eta\leq (m-n)l\alpha$, among which the factor $M(\bm{n}_\alpha)_{S_\alpha}$ appears once. 
We decompose 
\begin{equation}\label{equ 4.15re}
Q'(\bm{n}_\alpha)^{\leq \bm{m}_\alpha}_{S_\alpha}
=\pr_{\sX^{-\rho}_{\alpha}}(Q'(\bm{n}_\alpha)^{\leq \bm{m}_\alpha}_{S_\alpha}) \oplus Q''_{\alpha, n,m},
\end{equation}
where $Q''_{\alpha, n,m}$ is a module outside the block $\sO^{\sX^{-\rho}_{\alpha}}_{S_\alpha}$. 
By the multiplicity of Verma factors of projective modules in $\sO^{\sX^{-\rho}_{\alpha},\leq \bm{m}_\alpha}_{S_\alpha}$ shown in Corollary \ref{cor 6.4}, it follows a decomposition 
\begin{equation}\label{equ 6.10} 
\pr_{\sX^{-\rho}_{\alpha}}(Q'(\bm{n}_\alpha)^{\leq \bm{m}_\alpha}_{S_\alpha}) = Q(\bm{n}_\alpha)^{\leq \bm{m}_\alpha}_{S_\alpha} \oplus Q'_{\alpha, n,m} ,
\end{equation} 
where $Q'_{\alpha, n,m}$ is a direct sum of $Q(\bm{n_1}{}_\alpha)^{\leq \bm{m}_\alpha}_{S_\alpha}$'s for some $n_1> n$. 

For any $n\leq n'\leq m$, there is a morphism 
$$\iota^{m}_{n',n,S_\alpha}: 
Q'(\bm{n'}_\alpha)^{\leq \bm{m}_\alpha}_{S_\alpha} \rightarrow  
Q'(\bm{n}_\alpha)^{\leq \bm{m}_\alpha}_{S_\alpha} ,\quad 
1_\zeta \otimes 1_+\otimes 1_{\bm{n'}_\alpha} \mapsto 
	1_\zeta \otimes E(\alpha)^{n'-n}.$$ 
Consider the composition 
$$\iota^{m}_{n',n,\alpha}: 
Q(\bm{n'}_\alpha)^{\leq \bm{m}_\alpha}_{S_\alpha} \hookrightarrow 
Q'(\bm{n'}_\alpha)^{\leq \bm{m}_\alpha}_{S_\alpha} 
\xrightarrow{\iota^{m}_{n',n,S_\alpha}} 
Q'(\bm{n}_\alpha)^{\leq \bm{m}_\alpha}_{S_\alpha} \twoheadrightarrow  
Q(\bm{n}_\alpha)^{\leq \bm{m}_\alpha}_{S_\alpha}.$$ 
Note that $\iota^{m}_{n',n,\alpha}$ depends on the decomposition (\ref{equ 6.10}). 
However the lemma below is stated for any choice of $\iota^{m}_{n',n,\alpha}$. 
\begin{lem}\label{lem 6.5} 
The morphism $\iota^{m}_{n',n,\alpha}$ is an embedding, and fits into a short exact sequence 
\begin{equation}\label{equ 6.12} 
0\rightarrow Q(\bm{n'}_\alpha)^{\leq \bm{m}_\alpha}_{S_\alpha} 
\xrightarrow{\iota^{m}_{n',n,\alpha}} Q(\bm{n}_\alpha)^{\leq \bm{m}_\alpha}_{S_\alpha} 
\xrightarrow{\epsilon^{\leq \bm{n'-1}_{\alpha}}} Q( \bm{n}_\alpha )^{\leq \bm{n'-1}_{\alpha}}_{S_\alpha} \rightarrow 0 . 
\end{equation} 
In particular, there is a short exact sequence 
\begin{equation}
0\rightarrow Q(\bm{n+1}_{\alpha})^{\leq \bm{m}_\alpha}_{S_\alpha} \rightarrow Q(\bm{n}_\alpha)^{\leq \bm{m}_\alpha}_{S_\alpha} \rightarrow M( \bm{n}_\alpha )_{S_\alpha} \rightarrow 0 . 
\end{equation}
\end{lem} 
We need the following claim. 

\begin{claim}\label{claim 6.6} 
The element $E(\alpha)$ is not nilpotent in $U^+_\zeta\otimes {\Bbbk_\alpha}$. 
\end{claim}
\begin{proof} 
We prove the claim by showing that the image of $E(\alpha)$ in $U\n\otimes {\Bbbk_\alpha}$ is nonzero. 
Recall that in (\ref{equ 3.1}) we define $Q(\omega)_{{\Bbbk_\alpha}}=U^\hb_\zeta \otimes_{\fU^{b}_\zeta} P^b(\omega)_{{\Bbbk_\alpha}}$, for any $\omega\in \Lambda$. 
When $\omega$ is an $\alpha$-singular weight, we have $P^b(\omega)_{{\Bbbk_\alpha}}=M(\omega)_{{\Bbbk_\alpha}}$ as $\fU^{b}_\zeta\otimes {\Bbbk_\alpha}$-modules, and in this case (c.f. (\ref{equ 3.20}))   
$$Q(\omega)_{{\Bbbk_\alpha}}=U^\hb_\zeta \otimes_{\fU^{b}_\zeta}M(\omega)_{{\Bbbk_\alpha}}=U^\hb_\zeta \otimes_{U^{\hb,\geq}_\zeta}(U\n\otimes ({\Bbbk_\alpha})_{\omega}).$$ 
By Lemma \ref{lem 3.2}(2), $\tau^{\leq \bm{1}_\alpha}Q(\bm{0}_\alpha)_{{\Bbbk_\alpha}}$ is projective in $\sO^{\leq \bm{1}_\alpha}_{{\Bbbk_\alpha}}$, therefore the natural surjection 
$$\pi:Q'(\bm{0}_\alpha)^{\leq \bm{1}_\alpha}_{{\Bbbk_\alpha}}\twoheadrightarrow 
\tau^{\leq \bm{1}_\alpha}Q(\bm{0}_\alpha)_{{\Bbbk_\alpha}}=U^\hb_\zeta \otimes_{U^{\hb,\geq}_\zeta}\big((U\n/\bigoplus_{\nu\nleq \alpha}(U\n)_\nu) \otimes ({\Bbbk_\alpha})_{-\rho}\big)$$ 
splits. 
Abbreviate $v_{-\rho}:=1_\zeta\otimes 1_+\otimes 1_{-\rho}\in Q'(\bm{0}_\alpha)^{\leq \bm{1}_\alpha}_{{\Bbbk_\alpha}}$. 
To show our assertion, it is enough to show that $\pi(E(\alpha).v_{-\rho})$ is nonzero. 

By (\ref{equ 4.15re}) and (\ref{equ 6.10}), we have a decomposition 
\begin{equation}\label{equ 4.18re}
Q'(\bm{0}_\alpha)^{\leq \bm{1}_\alpha}_{{\Bbbk_\alpha}}= Q(\bm{0}_\alpha)^{\leq \bm{1}_\alpha}_{{\Bbbk_\alpha}} \oplus M(\bm{1}_\alpha)^{\oplus j}_{{\Bbbk_\alpha}} \oplus  (Q''_{\alpha,0,1}\otimes_{S_\alpha}{\Bbbk_\alpha}).
\end{equation}
Note that $\tau^{\leq \bm{1}_\alpha}Q(\bm{0}_\alpha)_{{\Bbbk_\alpha}}$ admits Verma factors of the form $M(\bm{0}_\alpha+l\eta)_{{\Bbbk_\alpha}}$ with $0\leq \eta \leq \alpha$, among which the factor $M(\bm{0}_\alpha)_{{\Bbbk_\alpha}}$ appears once. 
As a direct summand of  $Q'(\bm{0}_\alpha)^{\leq \bm{1}_\alpha}_{{\Bbbk_\alpha}}$, the module $\tau^{\leq \bm{1}_\alpha}Q(\bm{0}_\alpha)_{{\Bbbk_\alpha}}$ admits a decomposition 
$$\tau^{\leq \bm{1}_\alpha}Q(\bm{0}_\alpha)_{{\Bbbk_\alpha}}= Q(\bm{0}_\alpha)^{\leq \bm{1}_\alpha}_{{\Bbbk_\alpha}} \oplus M(\bm{1}_\alpha)^{\oplus j'}_{{\Bbbk_\alpha}} \oplus  Q'',$$ 
where $j'\leq j$ and $Q''$ is a direct summand of $Q''_{\alpha,0,1}\otimes_{S_\alpha}{\Bbbk_\alpha}$. 
We can adjust the decomposition (\ref{equ 4.18re}) such that $\pi$ is by identity on $Q(\bm{0}_\alpha)^{\leq \bm{1}_\alpha}_{{\Bbbk_\alpha}}$, by killing $(j-j')$ factors for the middle term, and by a projection on the last term. 

If $\pi(E(\alpha).v_{-\rho})=0$, then $E(\alpha).v_{-\rho}\in M(\bm{1}_\alpha)^{\oplus j}_{{\Bbbk_\alpha}} \oplus (Q''_{\alpha,0,1}\otimes_{S_\alpha}{\Bbbk_\alpha})$. 
Note that $(Q''_{\alpha,0,1})_{\bm{1}_\alpha}=0$, it shows that $E(\alpha).v_{-\rho}$ is a nonzero element in $M(\bm{1}_\alpha)^{\oplus j}_{{\Bbbk_\alpha}}$. 
By (\ref{equ 6.9}), we have $E(\alpha).v_{-\rho}\in E(\alpha). M(\bm{1}_\alpha)^{\oplus j}_{{\Bbbk_\alpha},\bm{0}_\alpha}=0$, which is absurd. 
\end{proof} 

\ 

\begin{proof}[Proof of Lemma \ref{lem 6.5}] 
We do increase induction on $m-n'$. 
We firstly suppose $n'=m$. 
Then $Q'(\bm{m}_\alpha)^{\leq \bm{m}_\alpha}_{S_\alpha}=Q(\bm{m}_\alpha)^{\leq \bm{m}_\alpha}_{S_\alpha}=M(\bm{m}_\alpha)_{S_\alpha}$, and 
$$\iota^{m}_{n',n,S_\alpha}: M(\bm{m}_\alpha)_{S_\alpha} \rightarrow Q'(\bm{n}_\alpha)^{\leq \bm{m}_\alpha}_{S_\alpha} ,\quad 
1_{\bm{m}_\alpha} \mapsto 1_\zeta \otimes E(\alpha)^{m-n},$$ 
which is an inclusion and so is its specialization at ${\Bbbk_\alpha}$, thanks to Claim \ref{claim 6.6}. 
The element $\iota^{m}_{n',n,S_\alpha}(1_{\bm{m}_\alpha})$ is contained in the direct summand 
	$$\pr_{\sX^{-\rho}_{\alpha}}(Q'(\bm{n}_\alpha)^{\leq \bm{m}_\alpha}_{S_\alpha})= Q(\bm{n}_\alpha)^{\leq \bm{m}_\alpha}_{S_\alpha} \oplus Q'_{\alpha, n,m}.$$ 
From (\ref{equ 6.9}) and Corollary \ref{cor 6.4}, we deduce that $E(\alpha)^{m-n}.Q(\bm{n_1}{}_\alpha)^{\leq \bm{m}_\alpha}_{{\Bbbk_\alpha},\bm{n}_\alpha}=0$ if $n_1>n$. 
Since $\iota^{m}_{n',n,S_\alpha}(1_{\bm{m}_\alpha})$ lies in $E(\alpha)^{m-n}.Q'(\bm{n}_\alpha)^{\leq \bm{m}_\alpha}_{S_\alpha,\bm{n}_\alpha}$, it follows that the image of $1_{\bm{m}_\alpha}$ under the map 
$$\iota^{m}_{n',n,\alpha}\otimes_{S_\alpha}{\Bbbk_\alpha}: M(\bm{m}_\alpha)_{{\Bbbk_\alpha}} \rightarrow Q(\bm{n}_\alpha)^{\leq \bm{m}_\alpha}_{{\Bbbk_\alpha}}$$ 
	is nonzero. 
Thus $\iota^{m}_{n',n,\alpha}\otimes_{S_\alpha}{\Bbbk_\alpha}$ is an injection, and $\iota^{m}_{n',n,\alpha}$ identifies $M(\bm{m}_\alpha)_{S_\alpha}$ with the Verma factor of $Q(\bm{n}_\alpha)^{\leq \bm{m}_\alpha}_{S_\alpha}$ in (\ref{equ 6.8}). 
	
In general, suppose $n'<m$. 
One can choose direct inclusions and projections  
$$ Q(\bm{n}_\alpha)^{\leq \bm{m'}_\alpha}_{S_\alpha} \hookrightarrow 
Q'(\bm{n}_\alpha)^{\leq \bm{m'}_\alpha}_{S_\alpha} , \quad 
Q'(\bm{n}_\alpha)^{\leq \bm{m'}_\alpha}_{S_\alpha} \twoheadrightarrow  
Q(\bm{n}_\alpha)^{\leq \bm{m'}_\alpha}_{S_\alpha}, \quad m'=m,m-1,$$ 
that are compatible with the truncations 
$$\epsilon^{\leq \bm{m-1}_\alpha}:Q'(\bm{n}_\alpha)^{\leq \bm{m}_\alpha}_{S_\alpha}\rightarrow Q'(\bm{n}_\alpha)^{\leq \bm{m-1}_\alpha}_{S_\alpha}, \quad 
\epsilon^{\leq \bm{m-1}_\alpha}:Q(\bm{n}_\alpha)^{\leq \bm{m}_\alpha}_{S_\alpha}\rightarrow Q(\bm{n}_\alpha)^{\leq \bm{m-1}_\alpha}_{S_\alpha}.$$ 
Then there is a commutative diagram 
$$\begin{tikzcd} 
0 \arrow[r] & 
M(\bm{m}_\alpha)_{S_\alpha} \arrow[r]\arrow[d,equal] & 
 Q(\bm{n'}_\alpha)^{\leq \bm{m}_\alpha}_{S_\alpha} \arrow[r,"\epsilon^{\leq \bm{m-1}_{\alpha}}"] \arrow[d,"\iota^{m}_{n',n,\alpha}"] &  
 Q(\bm{n'}_\alpha)^{\leq \bm{m-1}_{\alpha}}_{S_\alpha} \arrow[r] \arrow[d,"\iota^{m-1}_{n',n,\alpha}"] & 0 \\ 
 0 \arrow[r] & M(\bm{m}_\alpha)_{S_\alpha} \arrow[r] 
 & Q(\bm{n}_\alpha)^{\leq \bm{m}_\alpha}_{S_\alpha} \arrow[r,"\epsilon^{\leq \bm{m-1}_{\alpha}}"] \arrow[d,"\epsilon^{\leq \bm{n'-1}_{\alpha}}"] &  Q(\bm{n}_\alpha)^{\leq \bm{m-1}_{\alpha}}_{S_\alpha} \arrow[r] \arrow[d,"\epsilon^{\leq \bm{n'-1}_{\alpha}}"] & 0 \\
 & & Q(\bm{n}_\alpha)^{\leq \bm{n'-1}_{\alpha}}_{S_\alpha} \arrow[r,equal] 
 & Q(\bm{n}_\alpha)^{\leq \bm{n'-1}_{\alpha}}_{S_\alpha} ,
 \end{tikzcd}$$ 
where the upper two rows and the right column are short exact sequences. 
By induction hypothesis that $\iota^{m-1}_{n',n,\alpha}$ is an inclusion, the map $\iota^{m}_{n',n,\alpha}$ is injective by Snake Lemma and the middle column is a short exact sequence. 
\end{proof}

For any $n\leq n' \leq m' \leq m$, the lemma above realizes the module $Q(\bm{n'}_\alpha)^{\leq \bm{m'}_\alpha}_{S_\alpha}$ as a sub-quotient of $Q(\bm{n}_\alpha)^{\leq \bm{m}_\alpha}_{S_\alpha}$ that is extended by the Verma factors $\{M(\bm{j}_\alpha)_{S_\alpha}\}_{j=n'}^{m'}$, and such realization is unique up to an automorphism of $Q(\bm{n'}_\alpha)^{\leq \bm{m'}_\alpha}_{S_\alpha}$. 
Since $\Ext^i(M(\lambda)_R,M(\mu)_R)\neq 0$ only if $\lambda\leq \mu$, for $i=0,1$, any endomorphism of $Q(\bm{n}_\alpha)^{\leq \bm{m}_\alpha}_{S_\alpha}$ induces a one for the sub-quotient $Q(\bm{n'}_\alpha)^{\leq \bm{m'}_\alpha}_{S_\alpha}$. 
It gives an algebra homomorphism 
\begin{equation}\label{equ 6.14+1} 
\End(Q(\bm{n}_\alpha)^{\leq \bm{m}_\alpha}_{S_\alpha})\rightarrow \End(Q(\bm{n'}_\alpha)^{\leq \bm{m'}_\alpha}_{S_\alpha}), 
\end{equation} 
which is unique up to conjugations by $\Aut(Q(\bm{n'}_\alpha)^{\leq \bm{m'}_\alpha}_{S_\alpha})$. 
By the algebra embedding 
\begin{equation}\label{equ 4.17} 
\End(Q(\bm{n}_\alpha)^{\leq \bm{m}_\alpha}_{S_\alpha})\hookrightarrow \End(Q(\bm{n}_\alpha)^{\leq \bm{m}_\alpha}_{S_\alpha}\otimes_{S_\alpha}\K) =\prod_{j=n}^m \End(M(\bm{j}_\alpha)_\K),
\end{equation}
it follows that $\End(Q(\bm{n}_\alpha)^{\leq \bm{m}_\alpha}_{S_\alpha})$ is a commutative algebra. 
Hence (\ref{equ 6.14+1}) is unique, and it yields a projective limit 
\begin{equation}\label{equ 6.-1} 
\End(Q(\bm{-\infty}_\alpha )^{\leq \bm{\infty}_\alpha}_{S_\alpha})
:= \varprojlim_{n \leq m} \End(Q( \bm{n}_\alpha )^{\leq \bm{m}_\alpha}_{S_\alpha}), 
\end{equation}
where the LHS is regarded as a single symbol. 
We regard it as as the ``endomorphism algebra of big projective object". 

\subsection{The $\alpha$-regular case}\label{subsect 6.1.2} 
In this subsection, we study the blocks $\sO^{\sX^{\omega}_\alpha}_{S_\alpha}$ for $\alpha$-regular weights $\omega$. 
We may assume $0<\langle\omega+\rho,\check{\alpha}\rangle<l$, without loss of generality.  
For simplicity, we abbreviate 
$$\bm{2n}_{\alpha,\omega}=\omega + ln\alpha, \quad \bm{2n-1}_{\alpha,\omega}= s_\alpha\bullet \omega + ln\alpha , \quad n\in \Z.$$ 

Let $V$ be the Weyl module for $U_\zeta$ with extreme weight $\omega+\rho$. 
We define the \textit{translation functors} by 
$$T^{\omega}_{-\rho,\alpha}: \sO^{\sX^{-\rho}_\alpha}_{S_\alpha} \rightarrow \sO^{\sX^\omega_\alpha}_{S_\alpha}, \quad M \mapsto \pr_{\sX^\omega_\alpha}(M \otimes V) ,$$ 
$$T_{\omega,\alpha}^{-\rho} : \sO^{\sX^\omega_\alpha}_{S_\alpha} \rightarrow \sO^{\sX^{-\rho}_\alpha}_{S_\alpha}, \quad M \mapsto \pr_{\sX^{-\rho}_\alpha} (M \otimes V^*).$$ 
Note that $T^{\omega}_{-\rho,\alpha}$ and $T_{\omega,\alpha}^{-\rho}$ are exact and biadjoint to each other. 
Till the end of the section, we will abbreviate $\sT=T_{\omega,\alpha}^{-\rho}$ and $\sT'=T^{\omega}_{-\rho,\alpha}$. 

\begin{lem}\label{lem 6.7} 
Let $n\in \Z$. 
\begin{enumerate}
\item $\sT M(\bm{2n}_{\alpha,\omega})_{S_\alpha}=\sT M(\bm{2n-1}_{\alpha,\omega})_{S_\alpha}= M(\bm{n}_{\alpha})_{S_\alpha}$. 
\item $\sT' M(\bm{n}_{\alpha})_{S_\alpha}$ fits into a short exact sequence 
$$0\rightarrow M(\bm{2n}_{\alpha,\omega})_{S_\alpha} \rightarrow 
\sT' M(\bm{n}_{\alpha})_{S_\alpha} \rightarrow M(\bm{2n-1}_{\alpha,\omega})_{S_\alpha} \rightarrow 0 .$$ 
\end{enumerate} 
\end{lem} 
\begin{proof} 
We consider the Verma factors of $M(\bm{2n}_{\alpha,\omega})_{S_\alpha}\otimes V^*$, whose highest weights belonging to $-\rho+ \Z \alpha$ are of the form $M(-\rho+j\alpha)_{S_\alpha}$ for $ln\leq j\leq ln+\langle \omega+\rho, \check{\alpha}\rangle$. 
By our assumption $0<\langle \omega+\rho, \check{\alpha}\rangle<l$, it follows that $\sT M(\bm{2n}_{\alpha,\omega})_{S_\alpha}=M(\bm{n}_{\alpha})_{S_\alpha}$. 
The case for $\sT M(\bm{2n-1}_{\alpha,\omega})_{S_\alpha}$ and Part (2) are proved similarly. 
\end{proof}

\begin{lem}\label{lem 6.8} 
Let $n\leq n' \leq m$. 
\begin{enumerate} 
\item We have $Hom(M(\bm{n-1}_{\alpha,\omega})_{{\Bbbk_\alpha}},M(\bm{n}_{\alpha,\omega})_{{\Bbbk_\alpha}})={\Bbbk_\alpha}$, and any nonzero element gives an embedding $M(\bm{n-1}_{\alpha,\omega})_{{\Bbbk_\alpha}} \hookrightarrow M(\bm{n}_{\alpha,\omega})_{{\Bbbk_\alpha}}$.  
\item There is an isomorphism $\sT'Q(\bm{n}_\alpha)^{\leq \bm{m}_\alpha}_{S_\alpha}= 
	Q(\bm{2n-1}_{\alpha,\omega})^{\leq \bm{2m}_{\alpha,\omega}}_{S_\alpha}$, and the module admits Verma factors 
\begin{equation}\label{equ 6.13} 
M(\bm{2n-1}_{\alpha,\omega})_{S_\alpha}, M(\bm{2n}_{\alpha,\omega})_{S_\alpha}, \cdots ,M(\bm{2m}_{\alpha,\omega})_{S_\alpha},
\end{equation}
each of which appears once. 
\item There is a short exact sequence 
$$ 0\rightarrow 
Q(\bm{2n'-1}_{\alpha,\omega})^{\leq \bm{2m}_{\alpha,\omega}}_{S_\alpha} \rightarrow 
Q(\bm{2n-1}_{\alpha,\omega})^{\leq \bm{2m}_{\alpha,\omega}}_{S_\alpha} \rightarrow 
Q(\bm{2n-1}_{\alpha,\omega})^{\leq \bm{2(n'-1)}_{\alpha,\omega}}_{S_\alpha} \rightarrow 0 , $$ 
where the last map is by $\epsilon^{\leq \bm{2(n'-1)}_{\alpha,\omega}}$. 
\end{enumerate} 
\end{lem} 
\begin{proof} 
(1) Applying the Jantzen sum formula (\ref{equ 6.2}) to $M(\bm{n}_{\alpha,\omega})_{{\Bbbk_\alpha}}$, we get 
$$\sum_{i>0}\ch M^i(\bm{n}_{\alpha,\omega})_{{\Bbbk_\alpha}}=\sum_{i>0}\ch M(\bm{(n+1-2i)}_{\alpha,\omega})_{{\Bbbk_\alpha}}.$$ 
Let $\nu\nless \bm{n-1}_{\alpha,\omega}$. 
Consider the coefficients of $\exp(\nu)$ on both sides, we obtain that 
$$\sum_{i>0} \dim M^i(\bm{n}_{\alpha,\omega})_{{\Bbbk_\alpha},\nu} = \dim M(\bm{(n-1)}_{\alpha,\omega})_{{\Bbbk_\alpha},\nu}=\delta_{\bm{(n-1)}_{\alpha,\omega},\nu}.$$ 
Hence $M^i(\bm{n}_{\alpha,\omega})_{{\Bbbk_\alpha},\nu}\neq 0$ only if $i=1$ and $\nu=\bm{n-1}_{\alpha,\omega}$, and in this case it has dimension $1$. 
Thus 
$$\Hom(M(\bm{n-1}_{\alpha,\omega})_{{\Bbbk_\alpha}},M(\bm{n}_{\alpha,\omega})_{{\Bbbk_\alpha}})=\Hom(M(\bm{n-1}_{\alpha,\omega})_{{\Bbbk_\alpha}},M^1(\bm{n}_{\alpha,\omega})_{{\Bbbk_\alpha}})={\Bbbk_\alpha}.$$ 
The second assertion follows from that $\fU^-_\zeta$ is a torsion free ring. 

(2) By Lemma \ref{lem 6.7}, the translation functors $\sT$ and $\sT'$ restrict on the truncated categories 
$$\sT: \sO^{\sX^\omega_\alpha, \leq \bm{2m}_{\alpha,\omega}}_{S_\alpha} \rightarrow \sO^{\sX^{-\rho}_\alpha, \leq \bm{m}_{\alpha}}_{S_\alpha}, 
    \quad 
\sT': \sO^{\sX^{-\rho}_\alpha, \leq \bm{m}_{\alpha}}_{S_\alpha} \rightarrow \sO^{\sX^\omega_\alpha, \leq \bm{2m}_{\alpha,\omega}}_{S_\alpha}.$$ 
Since $\sT$ and $\sT'$ are biajoint, they send projective objects to projective objects. 
In particular, $\sT'Q(\bm{n}_\alpha)^{\leq \bm{m}_\alpha}_{S_\alpha}$ is projective in $\sO^{\sX^\omega_\alpha, \leq \bm{2m}_{\alpha,\omega}}_{S_\alpha}$. 
By Corollary \ref{cor 6.4} and Lemma \ref{lem 6.7}(2), it admits Verma factors as in (\ref{equ 6.13}). 
It shows that $\sT'Q(\bm{n}_\alpha)^{\leq \bm{m}_\alpha}_{S_\alpha}$ contains $Q(\bm{2n-1}_{\alpha,\omega})^{\leq \bm{2m}_{\alpha,\omega}}_{S_\alpha}$ as a direct factor. 
From the linkage principle (\ref{equ 6.0}) and Proposition \ref{prop BGG}, one deduces that the Verma factors in (\ref{equ 6.13}) appear in $Q(\bm{2n-1}_{\alpha,\omega})^{\leq \bm{2m}_{\alpha,\omega}}_{S_\alpha}$ at least once. 
Hence $\sT'Q(\bm{n}_\alpha)^{\leq \bm{m}_\alpha}_{S_\alpha}= 
	Q(\bm{2n-1}_{\alpha,\omega})^{\leq \bm{2m}_{\alpha,\omega}}_{S_\alpha}$. 

Part (3) follows from applying $\sT'$ on (\ref{equ 6.12}).  
\end{proof} 

As in the last subsection, we define the ``endomorphism algebra of big projective object" via the projective limit 
\begin{equation}\label{equ 6.-2} 
\End(Q(\bm{-\infty}_{\alpha,\omega} )^{\leq \bm{\infty}_{\alpha,\omega}}_{S_\alpha}):= \varprojlim_{n \leq m} \End(Q(\bm{2n-1}_{\alpha,\omega})^{\leq \bm{2m}_{\alpha,\omega}}_{S_\alpha}). 
\end{equation}

\section{Center of category $\sO_S$} 
In this section, we relate the center of $\sO_S$ to the equivariant cohomology of $\zeta$-fixed locus of the affine Grassmannian $\Gr$. 

\subsection{Affine flag variety and its cohomology} 
\subsubsection{Affine flag varieties}\label{nsubsect 5.1.1} 
Let $\check{G}((t))=\check{G}\big(\C(\!(t)\!)\big)$ and $\check{G}[[t]]=\check{G}\big(\C[[t]]\big)$ be the \textit{loop group} and the \textit{arc group} of $\check{G}$. 
For any subset $J\subset \check{\Sigma}_\af$, the parabolic subgroup $W_J\subset W_\ex$ corresponds to a parahoric subgroup $P^J$ of $\check{G}((t))$. 
The \textit{partial affine flag variety} is the fpqc quotient $\Fl^J=\check{G}((t))/P^J$. 
For any $x\in W_\ex$, we fix its lifting $\dot{x}$ in $\check{G}((t))$, then there is an identification 
\begin{equation}\label{nequ 5.1}
(\Fl^J)^{\check{T}} = \{ \dot{x}P^J/P^J\}_{x\in W^J_\ex} =W^J_\ex. 
\end{equation}
Let $\Fl^{J,\circ}$ be the connected component of $\Fl^J$ containing the base point $P^J/P^J$. 
Then (\ref{nequ 5.1}) identifies $(\Fl^{J,\circ})^{\check{T}}=W_\af^J$. 
We have an isomorphism $\Fl^J=\pi_1\times \Fl^{J,\circ}$, which at the level of $\check{T}$-fixed points recovers the identification $W^J_\ex=\pi_1\times W_\af^J$. 

If $J=\check{\Sigma}$, we have $P^{\check{\Sigma}}=\check{G}[[t]]$. 
The \textit{affine Grassmannian} of $\check{G}$ is the fpqc quotient $\Gr=\check{G}((t))/\check{G}[[t]]$. 
For $\lambda \in \Lambda$, we denote by $t^\lambda$ the image of $t$ under the following morphism 
$$ \lambda\big(\C(\!(t)\!)\big):\ \Gm\big(\C(\!(t)\!)\big) =\C(\!(t)\!)^\times \rightarrow \check{T}\big(\C(\!(t)\!)\big), $$
and set $\delta_\lambda=t^\lambda \check{G}[[t]]/\check{G}[[t]]$. 
In the case of affine Grassmannian, (\ref{nequ 5.1}) is the equality 
\begin{equation}\label{nequ 5.2} 
\Gr^{\check{T}}=\{\delta_\lambda |\ \lambda \in \Lambda\}=\Lambda.
\end{equation}

\subsubsection{$\zeta$-fixed locus of affine Grassmannian} 
We can substitute $t^l$ for $t$ in \textsection\ref{nsubsect 5.1.1} and replace $W_\ex$ by $W_{l,\ex}$. 
For $\omega\in \Xi_\sc$, let $P^\omega$ be the parahoric subgroup of $\check{G}((t^l))$ associated with the parabolic subgroup $W_{l,\omega}\subset W_{l,\ex}$, and let $\Fl^{\omega}=\check{G}((t^l))/P^\omega$ be the associated partial affine flag variety. 
Let $\Fl^{\omega,\circ}$ be the connected component of $P^\omega/P^\omega$. 
By (\ref{nequ 5.1}), we have $(\Fl^{\omega})^{\check{T}}=W^\omega_{l,\ex}$, $(\Fl^{\omega,\circ})^{\check{T}}=W^\omega_{l,\af}$. 

There is a $\Gm$-action on $\C(\!(t)\!)$ by rotating $t$, which induces a $\Gm$-action on $\Gr$. 
Let $\Gr^\zeta$ be the fixed locus of $\zeta\in \Gm$ on $\Gr$. 
By \cite[\textsection 4]{RW22}, there is an isomorphism 
\begin{equation}\label{equ 1.1} 
\bigsqcup_{\omega\in \Xi} \Fl^{\omega}=\Gr^\zeta, \quad gP^\omega/P^\omega \mapsto g\delta_{\omega+\rho}, \quad \forall g\in \check{G}((t^l)), 
\end{equation} 
where for each $\omega\in \Xi$ we fix a lifting $\omega\in \Xi_\sc$. 
At the level of $\check{T}$-fixed points, (\ref{equ 1.1}) recovers the identity $\bigsqcup_{\omega\in \Xi} W^\omega_{l,\ex} =\Lambda$, where $x \in W^\omega_{l,\ex}$ is mapped to $x \bullet \omega +\rho$. 

\subsubsection{Cohomology of affine flag varieties} 
Let $T'=\check{T}\times \Gm$, $\check{T}$, $\Gm$ or the trivial group. 
The cohomology group $H^\bullet_{T'}(\Fl^J)$ is freely generated over $H^\bullet_{T'}(\pt)$ by the fundamental classes $[\Fl^{J,x}]_{T'}$ of the finite codimensional Schubert varieties $\Fl^{J,x}$ labelled by $x\in W^{J}_{\ex}$. 
For any $H^\bullet_{T'}(\pt)$-algebra $R$, we denote by 
$$H^\bullet_{T'}(\Fl^J)^\wedge_R$$ 
the space of formal series of $[\Fl^{J,x}]_{T'}$ with coefficients in $R$. 
We will drop the subscript $R$ if $R=H^\bullet_{T'}(\pt)$. 

The restriction to $\check{T}$-fixed points induces embeddings 
$$H^\bullet_{\check{T}}(\Fl^J) \hookrightarrow \Fun(W^J_\ex,S'), \quad H^\bullet_{\check{T}\times \Gm}(\Fl^J) \hookrightarrow \Fun(W^J_\ex,\hat{S}'). $$ 
For any finite subset $\sX\subset W^J_\ex$, there are finitely many $y\in W^J_\ex$ such that the restriction of $[\Fl^{J,y}]_{\check{T}}$, resp. $[\Fl^{J,y}]_{\check{T}\times \Gm}$, at $\dot{x}P^J/P^J$ is nonzero for some $x\in \sX$. 
Hence we have embeddings 
$$H^\bullet_{\check{T}}(\Fl^J)^\wedge\hookrightarrow \Fun(W^J_\ex,S'), \quad H^\bullet_{\check{T}\times \Gm}(\Fl^J)^\wedge\hookrightarrow \Fun(W^J_\ex,\hat{S}').$$ 

The discussion above is immediately adopted to $\Fl^\omega$ ($\omega\in \Xi_\sc$), after replacing the superscripts $J$ by $\omega\in \Xi_\sc$ and replacing $W^{J}_{\ex}$ by $W^\omega_{l,\ex}$. 
For any $H^\bullet_{T'}(\pt)$-algebra $R$, we set 
$$H^\bullet_{T'}(\Gr^\zeta)^\wedge_R :=\prod_{\omega\in \Xi}H^\bullet_{T'}(\Fl^\omega)^\wedge_R.$$

\subsection{The map from cohomology to center} 
Consider the evaluation map 
$$\chi_R:\ Z(\sO_{R})\rightarrow \prod_{\lambda\in \Lambda}\End_{\sO_{R}} (M(\lambda)_R)=\Fun(\Lambda,R).$$ 
\begin{prop}[{\cite[Prop 4.8]{BBASV}}]\label{prop 3.7} 
There is a commutative diagram of algebra homomorphisms 
$$\begin{tikzcd}
H_{\check{T}\times\Gm}^\bullet(\Gr^\zeta) \arrow[r,"\hat{\bb}"] \arrow[d]
& Z(\sO_{\hat{S}}) \arrow[r,"\chi_{\hat{S}}"] \arrow[d] 
& \Fun(\Lambda,\hat{S}) \arrow[d] \\ 
H_{\check{T}}^\bullet(\Gr^\zeta) \arrow[r,"\bb"] & Z(\sO_{S}) \arrow[r,"\chi_{S}"] & \Fun(\Lambda,S),
\end{tikzcd}$$ 
such that the compositions $\chi_{\hat{S}}\circ \hat{\bb}$ and $\chi_{S}\circ {\bb}$ coincide with the restrictions on the $\check{T}$-fixed points $\{\delta_{\lambda+\rho}\}_{\lambda\in \Lambda}$. 
In particular, the map $\bb$ and $\hat{\bb}$ are embeddings and are compatible with the decompositions (\ref{equ 1.1}) and (\ref{equ 3.13}), yielding homomorphisms 
$$\bb:H_{\check{T}}^\bullet(\Fl^{\omega})\rightarrow Z(\sO^{[\omega]}_{S}), \quad \hat{\bb}:H_{\check{T}\times\Gm}^\bullet(\Fl^{\omega})\rightarrow Z(\sO^{[\omega]}_{\hat{S}}),$$ 
for each $\omega\in \Xi$. 
\end{prop}

In the following two lemmas, we suppose $R$ satisfies $S\subset R \subset \K$, resp. $\hat{S}\subset R\subset \hat{\K}$. 
\begin{lem}\label{lem 3.7} 
The map $\chi_R$ is an inclusion, and the image of $Z(\sO_{R})$ is a complete subspace of $\Fun(\Lambda, R)$ with respect to the product topology. 
\end{lem} 
\begin{proof} 
We only prove the assertion for $S\subset R \subset \K$. 
Since $\sO_{\K}$ is semisimple by Lemma \ref{lem 3.6} and the Verma modules form a complete set of simple objects, we have $\chi_\K:Z(\sO_{\K})\xs \Fun(\Lambda, \K)$. 
Since $\chi_R$ is compatible with $\chi_\K$ with respect to the inclusion $Z(\sO_{R})\subset Z(\sO_{\K})$ by base change, it shows that $\chi_R$ is an injection. 
For the second assertion, note that elements in $Z(\sO_{R})$ are determined by their restrictions on projective modules in truncated categories, and projective modules admit finite Verma factors. 
Hence the subspace $Z(\sO_{R})$ is closed under the pro-finite completion.  
\end{proof} 

\begin{lem}\label{lem 3.11} 
The map $\bb$, resp. $\hat{\bb}$, extends to an algebra homomorphism 
\begin{equation}\label{equ 3.16} 
\bb: H_{\check{T}}^\bullet(\Gr^\zeta)^{\wedge}_{R} \rightarrow Z(\sO_{R}), \quad \text{resp.}\quad 
\hat{\bb}: H_{\check{T}\times \Gm}^\bullet(\Gr^\zeta)^{\wedge}_{R} \rightarrow Z(\sO_{R}). 
\end{equation} 
For $R=S$, it induces a compatible algebra homomorphism 
\begin{equation}\label{equ 3.18} 
\overline{\bb}: H^\bullet(\Gr^\zeta)^{\wedge} \rightarrow Z(\sO_{\k}). 
\end{equation} 
\end{lem}
\begin{proof} 
We only show the extension for $\bb$. 
Recall that $H_{\check{T}}^\bullet(\Gr^\zeta)$ is a free $S'$-module generated by the fundamental classes $\{[\Fl^{\omega,x}]_{\check{T}}\}_{x\in W^\omega_{l,\ex}}$. 
Recall again that for any finite subset $\sX\subset W^\omega_{l,\ex}$, there are finitely many $y\in W^\omega_{l,\ex}$ such that the restriction of $[\Fl^{\omega,y}]_{\check{T}}$ on $\dot{x}P^\omega/P^\omega$ is nonzero for some $x\in \sX$. 
Hence the pro-finite completion of $H_{\check{T}}^\bullet(\Gr^\zeta)\otimes_{S'}R$ in $\Fun(\Lambda,R)$ is the subspace of formal sums of $\{[\Fl^{\omega,x}]_{\check{T}}\}_{x}$. 
By Lemma \ref{lem 3.7}, $Z(\sO_{R})$ is complete, so the map $\bb$ extends to the one in (\ref{equ 3.16}). 

For the second assertion, recall that any projective module $Q$ in some truncated category of $\sO_{\k}$ has a lifting in $\sO_{S}$ that admits finitely many Verma factors. 
Hence the action of $H_{\check{T}}^\bullet(\Gr^\zeta)^{\wedge}_{S}$ on $Q$ factorizes through a truncation of finite sum. 
It induces a homomorphism (\ref{equ 3.18}) that is compatible with (\ref{equ 3.16}). 
\end{proof} 

\ 

The rest of the section is devoted to proving the following main theorem. 
\begin{thm}\label{thm 3.11} 
There are algebra isomorphisms 
$$\bb: H_{\check{T}}^\bullet(\Gr^\zeta)^{\wedge}_{S} \xs Z(\sO_{S}) \quad \text{and} \quad  \hat{\bb}: H_{\check{T}\times \Gm}^\bullet(\Gr^\zeta)^{\wedge}_{\hat{S}} \xs Z(\sO_{\hat{S}}).$$ 
\end{thm} 

\subsection{Cohomology of subtorus-fixed locus} 
In this subsection, we reduce the cohomology of $\Gr^\zeta$ to the $SL_2$-case by taking the fixed points of $\Gr^\zeta$ of codimension $1$ subtorus of $\check{T}$, and introduce some precise cohomology classes. 

\subsubsection{Subtorus-fixed locus} 
The root $\alpha$ defines a codimension $1$ subtorus $\check{T}_\alpha=(\ker \check{\alpha})^\circ$ of $\check{T}$ (where $\circ$ denotes the neutral component). 
The $\check{T}_\alpha$-fixed locus of $\Gr^\zeta$ decomposes as  
\begin{equation}\label{equ 6.B16} 
(\Gr^\zeta)^{\check{T}_\alpha}=\bigsqcup_{\sX\in W_l(\alpha) \backslash \Lambda} \Fl^{\sX}_{_{SL_2}}, 
\end{equation} 
where $Fl^{\sX}_{_{SL_2}}$ is the connected component with $(\Fl^{\sX}_{_{SL_2}})^{\check{T}}=\sX$. 
We view $\check{\alpha}$ as the positive root of $SL_2$, and let $T_2\subset SL_2$ be the standard torus. 
Let $\Fl_{_{SL_2}}$ and $\Gr_{_{SL_2}}$ be the affine flag variety and the affine Grassmannian associated with the loop group $SL_2((t^l))$, and identify their fix points 
$$(\Fl_{_{SL_2}})^{T_2}=W_l(\alpha) \quad \text{and} \quad (\Gr_{_{SL_2}})^{T_2}=W_l(\alpha)/\langle s_\alpha\rangle=l\Z\alpha,$$ 
respectively. 
There are isomorphisms of $\check{T}$-varieties 
$$\Fl^{\sX}_{_{SL_2}}\simeq 
\begin{cases}
\Fl_{_{SL_2}} & \text{for any bijection $\sX=W_l(\alpha)$ of $W_l(\alpha)$-sets,} \\ 
\Gr_{_{SL_2}} & \text{for any bijection $\sX=l\Z\alpha$ of $W_l(\alpha)$-sets,} 
\end{cases}$$ 
via suitable isomorphism $\check{T}/\ker \check{\alpha}=T_2/\{\pm I_2\}$, such that the identifications of $W_l(\alpha)$-sets coincide with the isomorphisms restricted at the $\check{T}$-fixed points. 

\begin{prop}\label{prop B.1}
\begin{enumerate} 
\item The restriction induces an isomorphism 
$$H^\bullet_{\check{T}}(\Gr^\zeta)^\wedge_{S_\alpha} \xs H^\bullet_{\check{T}}\big((\Gr^\zeta)^{\check{T}_\alpha}\big)^\wedge_{S_\alpha}. $$
\item As subspaces in $\Fun(\Lambda,\K)$, there is an equality 
$$H^\bullet_{\check{T}}(\Gr^\zeta)^\wedge_{S} =\Fun(\Lambda,S)\cap 
\bigcap_{\alpha\in \Phi^+} H^\bullet_{\check{T}}\big((\Gr^\zeta)^{\check{T}_\alpha}\big)^\wedge_{S_\alpha}.$$ 
\end{enumerate} 
\end{prop} 
\begin{proof} 
By decomposition (\ref{equ 1.1}), it can be reduced to the corresponding assertion for partial affine flag varieties, which is shown in Proposition \ref{prop B.0}. 
\end{proof} 

By Lemma \ref{lem 3.11}, the map $\bb: H^\bullet_{\check{T}}(\Gr^\zeta)\rightarrow Z(\sO_{S_\alpha})$ extends to an algebra homomorphism 
\begin{equation}\label{equ 6.B15} 
\bb: H^\bullet_{\check{T}}(\Gr^\zeta)^\wedge_{S_\alpha}\rightarrow Z(\sO_{S_\alpha}). 
\end{equation} 
It is compatible with the decomposition (\ref{equ 6.3}) and the following one via (\ref{equ 6.B16}), 
$$H^\bullet_{\check{T}}(\Gr^\zeta)^\wedge_{S_\alpha}=\prod_{\sX\in W_l(\alpha) \backslash \Lambda} 
H^\bullet_{\check{T}}(\Fl^{\sX}_{_{SL_2}})^\wedge_{S_\alpha},$$ 
which yields an algebra homomorphism for each $\sX\in W_l(\alpha)\backslash \Lambda$, 
\begin{equation}\label{equ 6.B17} 
\bb: H^\bullet_{\check{T}}(\Fl^{\sX}_{_{SL_2}})^\wedge_{S_\alpha} \rightarrow Z(\sO^{\sX-\rho}_{S_\alpha}). 
\end{equation} 

\subsubsection{Cohomology in $SL_2$ case}\label{subsect 6.2.2} 
The $l$-affine Weyl group $W_l(\alpha)$ for $PGL_2$ is generated by the $l$-simple reflections $s_1=s_{\alpha}$ and $s_0=s_{\alpha,l}$. 
Under the restriction to $T_2$-fixed points, 
$$H^\bullet_{T_2}(\Fl_{_{SL_2}})\rightarrow \Fun(W_l(\alpha),\k[\alpha]),\quad H^\bullet_{T_2}(\Gr_{_{SL_2}})\rightarrow \Fun(l\Z\alpha ,\k[\alpha]),$$ 
the Schubert classes $[\Fl_{_{SL_2}}^{s_1}]_{\check{T}}$, $[\Fl_{_{SL_2}}^{s_0}]_{\check{T}}$ and $[\Gr_{_{SL_2}}^{s_0}]_{\check{T}}$ correspond to the functions 
$$[\Fl_{_{SL_2}}^{s_1}]_{\check{T}}:\ \tau_{nl\alpha} \mapsto -n\alpha,\ \tau_{nl\alpha}s_{\alpha} \mapsto -(n-1)\alpha, $$ 
$$[\Fl^{s_0}_{_{SL_2}}]_{\check{T}}:\ \tau_{nl\alpha} \mapsto -n\alpha,\ \tau_{nl\alpha}s_{\alpha} \mapsto -n\alpha, \quad [\Gr_{_{SL_2}}^{s_0}]_{\check{T}}:\ \tau_{nl\alpha} \mapsto -n\alpha, \quad \forall n\in \Z,$$ 
see \cite[\textsection 4.B.1]{Shim14}. 
We abbreviate 
$$\psi_\alpha=-[\Fl_{_{SL_2}}^{s_1}]_{\check{T}},\quad 
\phi'_\alpha=-[\Fl_{_{SL_2}}^{s_0}]_{\check{T}}, \quad \text{and}\quad 
\phi_\alpha=-[\Gr_{_{SL_2}}^{s_0}]_{\check{T}}.$$ 

Any $\omega\in \Lambda$ gives an isomorphism of $W_l(\alpha)$-sets 
$$\sX^\omega_\alpha+\rho=
\begin{cases} 
W_l(\alpha) \quad \text{by}\quad \omega+\rho \mapsto \text{unit}, & \text{if $\omega$ is $\alpha$-regular,} \\ 
l\Z\alpha \quad \text{by}\quad \omega+\rho \mapsto 0, & \text{if $\omega$ is $\alpha$-singular,} 
\end{cases}$$ 
and a compatible isomorphism of $\check{T}$-varieties 
$$\Fl^{\sX^\omega_\alpha+\rho}_{_{SL_2}}\simeq 
\begin{cases}
\Fl_{_{SL_2}} & \text{if $\omega$ is $\alpha$-regular,} \\ 
\Gr_{_{SL_2}} & \text{if $\omega$ is $\alpha$-singular.} 
\end{cases}$$ 
Consider the composition 
$$H^\bullet_{\check{T}}(\Fl^{\sX^\omega_\alpha+\rho}_{_{SL_2}})^\wedge_{S_\alpha} \xrightarrow{\bb} Z(\sO^{\sX^\omega_\alpha}_{S_\alpha})\xrightarrow{\chi_{S_\alpha}} \Fun(\sX^\omega_\alpha,S_\alpha).$$ 
When $\omega$ is $\alpha$-regular, the elements $\psi_\alpha$, $\phi'_\alpha$ correspond to the functions given by  
$$\psi_\alpha(\omega+ln\alpha)
=\psi_\alpha(s_\alpha\bullet \omega +ln\alpha)+\alpha=n\alpha, \quad \phi'_\alpha(\omega+ln\alpha)=\phi'_\alpha(s_\alpha\bullet \omega +ln\alpha)=n\alpha, \quad \forall n\in \Z.$$ 
When $\omega$ is $\alpha$-singular, the element $\phi_\alpha$ corresponds to the function given by $\phi_\alpha(\omega+ln\alpha)=n\alpha$, for any $n\in \Z$.

\subsection{Endomorphism of projective modules} 
The goal of this subsection is to show that (\ref{equ 6.B17}) is an isomorphism. 
To that end, we show that the cohomology rings map surjectively onto the ``endomorphism algebras of big projective object" defined in (\ref{equ 6.-1}) and (\ref{equ 6.-2}), by using the elements $\psi_\alpha$, $\phi'_\alpha$ and $\phi_\alpha$ introduced in \textsection \ref{subsect 6.2.2}. 

\subsubsection{The $\alpha$-singular case}\label{subsect 6.2.1} 
As in \textsection \ref{subsect 6.1.1}, we only consider the $\alpha$-singular weight $-\rho$. 
\begin{prop}\label{prop 6.9} 
For any $n\leq m$, the following composition is surjective, 
$$H^\bullet_{\check{T}}(\Gr_{_{SL_2}})^\wedge_{S_\alpha} \rightarrow 
Z(\sO^{ \sX^{-\rho}_\alpha}_{S_\alpha}) \rightarrow 
\End(Q(\bm{n}_\alpha)^{\leq \bm{m}_\alpha}_{S_\alpha}).$$ 
\end{prop}
\begin{proof} 
We prove the proposition by showing that the image of $\phi_\alpha$ generates $\End(Q( \bm{n}_\alpha )^{\leq \bm{m}_\alpha}_{S_\alpha})$ as an $S_\alpha$-algebra. 
In the paragraphs below, we will not distinguish $\phi_\alpha$ with its image under any map. 
Without loss of generality, we may assume $n=0$. 
Let us start by the following claim. 
\begin{claim}\label{claim 6.10} 
There is an algebra isomorphism 
$\End(Q(\bm{0}_\alpha)^{\leq \bm{1}_\alpha}_{{\Bbbk_\alpha}})= {\Bbbk_\alpha}[x]/x^2$, where $x$ represents the composition 
$$Q(\bm{0}_\alpha)^{\leq \bm{1}_\alpha}_{{\Bbbk_\alpha}} \twoheadrightarrow M(\bm{0}_\alpha)_{{\Bbbk_\alpha}} \hookrightarrow M(\bm{1}_\alpha)_{{\Bbbk_\alpha}} \hookrightarrow Q(\bm{0}_\alpha)^{\leq \bm{1}_\alpha}_{{\Bbbk_\alpha}}.$$
The image of $\phi_\alpha$ in $\End(Q(\bm{0}_\alpha)^{\leq \bm{1}_\alpha}_{S_\alpha})$ provides a lifting of $x$ (up to scalar). 
In particular, the proposition holds when $m-n=1$. 
\end{claim} 
\begin{proof}[Proof of the Claim] 
As in (\ref{equ 6.13}), $Q(\bm{0}_\alpha)^{\leq \bm{1}_\alpha}_{S_\alpha}$ fits into a short exact sequence 
$$0\rightarrow M( \bm{1}_\alpha )_{S_\alpha} \rightarrow Q(\bm{0}_\alpha)^{\leq \bm{1}_\alpha}_{S_\alpha} \xrightarrow{p} M( \bm{0}_\alpha )_{S_\alpha} \rightarrow 0.$$ 
So by Proposition \ref{prop BGG}, we have 
\begin{align*}
\dim \End(Q(\bm{0}_\alpha)^{\leq \bm{1}_\alpha}_{{\Bbbk_\alpha}}) &= 
[Q(\bm{0}_\alpha)^{\leq \bm{1}_\alpha}_{{\Bbbk_\alpha}}: L(\bm{0}_\alpha)_{{\Bbbk_\alpha}}] \\ 
&= [M(\bm{0}_\alpha)_{{\Bbbk_\alpha}} :L(\bm{0}_\alpha)_{{\Bbbk_\alpha}}]+ [M(\bm{1}_\alpha)_{{\Bbbk_\alpha}} :L(\bm{0}_\alpha)_{{\Bbbk_\alpha}}] =2. 
\end{align*} 
Hence the map ${\Bbbk_\alpha}[x]/x^2\rightarrow \End(Q(\bm{0}_\alpha)^{\leq \bm{1}_\alpha}_{{\Bbbk_\alpha}})$ is an isomorphism of algebras. 
Let $\tilde{x}$ be a lifting of $x$ in $\End(Q(\bm{0}_\alpha)^{\leq \bm{1}_\alpha}_{S_\alpha})$ such that $p\circ \tilde{x}=0$. 
By Nakayama's Lemma, $\End(Q(\bm{0}_\alpha)^{\leq \bm{1}_\alpha}_{S_\alpha})$ is a free $S_\alpha$-module generated by $1$ and $\tilde{x}$. 
Under the embedding 
$$\End(Q(\bm{0}_\alpha)^{\leq \bm{1}_\alpha}_{S_\alpha}) \hookrightarrow \End (M(\bm{0}_\alpha)_\K) \times \End (M(\bm{1}_\alpha)_\K)=\K \times \K,$$ 
the coordinate of $\tilde{x}$ is $(0,\alpha^c)$ (up to scalar) for some integer $c\geq 1$. 
Since the image of $\phi_\alpha$ is $(0, \alpha)$, it follows that $c=1$ and $\phi_\alpha$ provides a lifting of $x$ (up to scalar). 
\end{proof} 

Back to the general case. 
We choose a lifting $\tilde{x}_j$ for $0\leq j\leq m$ as the following diagram, 
\begin{equation}\label{equ 5.10} 
\begin{tikzcd}
Q(\bm{0}_\alpha)^{\leq \bm{m}_\alpha}_{S_\alpha} \arrow[r,dashrightarrow,"\tilde{x}_j"] \arrow[d,two heads]
& Q(\bm{j}_\alpha)^{\leq \bm{m}_\alpha}_{S_\alpha} \arrow[d,two heads]\\ 
M(\bm{0}_\alpha)_{{\Bbbk_\alpha}} \arrow[r,hook] & M(\bm{j}_\alpha)_{{\Bbbk_\alpha}}, 
\end{tikzcd}
\end{equation} 
and abuse $\tilde{x}_j$ as an endomorphism of $Q(\bm{0}_\alpha)^{\leq \bm{m}_\alpha}_{S_\alpha}$ via the embedding $\iota^{m}_{j,0,\alpha}:Q(\bm{j}_\alpha)^{\leq \bm{m}_\alpha}_{S_\alpha} \hookrightarrow Q(\bm{0}_\alpha)^{\leq \bm{m}_\alpha}_{S_\alpha}$ in (\ref{equ 6.12}). 
Since $Q(\bm{0}_\alpha)^{\leq \bm{m}_\alpha}_{{\Bbbk_\alpha}}$ is projective in $\sO_{{\Bbbk_\alpha}}^{\leq\bm{m}_\alpha}$ and it admits the Verma factors in (\ref{equ 6.8}), there is an (non-canonical) isomorphism of ${\Bbbk_\alpha}$-vector spaces 
\begin{equation}\label{equ 5.9} 
\End(Q(\bm{0}_\alpha)^{\leq \bm{m}_\alpha}_{{\Bbbk_\alpha}})=\bigoplus_{j=0}^m \Hom(Q(\bm{0}_\alpha)^{\leq \bm{m}_\alpha}_{{\Bbbk_\alpha}},M(\bm{j}_\alpha)_{{\Bbbk_\alpha}})={\Bbbk_\alpha}^{\oplus (m+1)}. 
\end{equation} 
By the commutative diagram (\ref{equ 5.10}), one can choose an isomorphism (\ref{equ 5.9}) such that the specializations of $\{\tilde{x}_j\}_{j=0}^m$ correspond to the standard basis of the RHS. 
By Nakayama's Lemma, $\{\tilde{x}_j\}_{j=0}^m$ generates $\End(Q( \bm{n}_\alpha )^{\leq \bm{m}_\alpha}_{S_\alpha})$ as a free $S_\alpha$-module. 

We show that $\tilde{x}_j$ can be generated by $\phi_\alpha$, then the proposition follows. 
Indeed, as a function in $\Fun(\sX^{-\rho}_\alpha,S_\alpha)$, we have $\phi_\alpha(\bm{i}_\alpha)=i\alpha$, so $\phi_\alpha-i\alpha$ kills $M(\bm{i}_\alpha)_{S_\alpha}$. 
Hence $\phi_\alpha-i\alpha$ maps $Q(\bm{i}_\alpha )^{\leq \bm{m}_\alpha}_{S_\alpha}$ to its submodule $Q(\bm{i+1}_\alpha )^{\leq \bm{m}_\alpha}_{S_\alpha}$. 
There is a commutative diagram for any $0\leq i< m$, 
$$\begin{tikzcd} 
Q(\bm{i}_\alpha)^{\leq\bm{m}_\alpha}_{S_\alpha} 
\arrow[r,two heads,"\epsilon^{\leq \bm{i+1}_\alpha}"] \arrow[d,"\phi_\alpha-i\alpha"'] & 
Q(\bm{i}_\alpha)^{\leq\bm{i+1}_\alpha}_{S_\alpha} 
\arrow[r,two heads] \arrow[d,"\phi_\alpha-i\alpha"'] & 
M(\bm{i}_\alpha)_{{\Bbbk_\alpha}} \arrow[d,hook] \\ 
Q( \bm{i+1}_\alpha )^{\leq \bm{m}_\alpha}_{S_\alpha} 
\arrow[r,two heads,"\epsilon^{\leq \bm{i+1}_\alpha}"] & M(\bm{i+1}_\alpha)_{S_\alpha} \arrow[r,two heads] & M(\bm{i+1}_\alpha)_{{\Bbbk_\alpha}}, 
\end{tikzcd}$$ 
where the left square commutes as in Claim \ref{claim 6.10}. 
It yields a commutative diagram 
\begin{equation}\label{equ 6.15} 
\begin{tikzcd} 
Q(\bm{0}_\alpha)^{\leq \bm{m}_\alpha}_{S_\alpha} 
    \arrow[r,dashrightarrow,"\phi_\alpha"'] \arrow[d,two heads] \arrow[rrr,dashrightarrow,bend left=10,"\tilde{x}_j"] & 
Q( \bm{1}_\alpha )^{\leq \bm{m}_\alpha}_{S_\alpha} 
\arrow[r,dashrightarrow,"\phi_\alpha-\alpha"']\arrow[d,two heads] & 
\cdots \arrow[r,dashrightarrow,"\phi_\alpha-(j-1)\alpha"'] & 
 Q( \bm{j}_\alpha )^{\leq \bm{m}_\alpha}_{S_\alpha} \arrow[d,two heads] \\ 
 M(\bm{0}_\alpha)_{{\Bbbk_\alpha}} \arrow[r,hook] & 
 M(\bm{1}_\alpha)_{{\Bbbk_\alpha}} \arrow[r,hook] &  \cdots \arrow[r,hook] & 
 M(\bm{j}_\alpha)_{{\Bbbk_\alpha}} ,\end{tikzcd}	
\end{equation}
hence $\tilde{x}_j$ can be chosen as $\tilde{x}_j = (\phi_\alpha-(j-1)\alpha) \circ \cdots \circ (\phi_\alpha-\alpha) \circ \phi_\alpha$.  
\end{proof} 

By the embedding (\ref{equ 4.17}), any endomorphism of $Q(\bm{n}_\alpha)^{\leq \bm{m}_\alpha}_{S_\alpha}$ is determined by its restrictions on each $M(\bm{j}_\alpha)_{S_\alpha}$. 
Hence the natural projection 
$$\End(Q(\bm{-\infty}_\alpha)^{\leq \bm{\infty}_\alpha}_{S_\alpha})\rightarrow \prod_n \End (M(\bm{n}_\alpha)_{S_\alpha})= \Fun(\sX^{-\rho}_\alpha, S_\alpha)$$ 
is injective. 
Consider the following chain of inclusions 
$$H^\bullet_{\check{T}}(\Gr_{_{SL_2}})^\wedge_{S_\alpha} \xrightarrow{\bb} 
Z(\sO^{\sX^{-\rho}_\alpha }_{S_\alpha})\rightarrow \End(Q(\bm{-\infty}_\alpha )^{\leq \bm{\infty}_\alpha}_{S_\alpha}) 
\rightarrow \Fun(\sX^{-\rho}_\alpha, S_\alpha),$$ 
where middle arrow is induced by the maps $Z(\sO^{\sX^{-\rho}_\alpha }_{S_\alpha})\rightarrow \End(Q(\bm{n}_\alpha)^{\leq \bm{m}_\alpha}_{S_\alpha})$ for any $n\leq m$, and the composition of the last two maps coincides with $\chi_{S_\alpha}$. 
\begin{corollary}\label{cor 6.11} 
The map  
$\bb: H^\bullet_{\check{T}}(\Gr_{_{SL_2}})^\wedge_{S_\alpha} \rightarrow Z(\sO^{\sX^{-\rho}_\alpha }_{S_\alpha})$ 
is an isomorphism. 
\end{corollary} 
\begin{proof} 
It is enough to show that the inclusion 
\begin{equation}\label{equ 6.16} 
H^\bullet_{\check{T}}(\Gr_{_{SL_2}})^\wedge_{S_\alpha} \hookrightarrow \End(Q(\bm{-\infty}_\alpha )^{\leq \bm{\infty}_\alpha}_{S_\alpha})
\end{equation}
is surjective. 
By Proposition \ref{prop 6.9}, the composition 
$$H^\bullet_{\check{T}}(\Gr_{_{SL_2}})^\wedge_{S_\alpha} \hookrightarrow \End(Q(\bm{-\infty}_\alpha )^{\leq \bm{\infty}_\alpha}_{S_\alpha}) \rightarrow \End (Q( \bm{n}_\alpha )^{\leq \bm{m}_\alpha}_{S_\alpha})$$ 
is surjective for any $n\leq m$. 
Since $H^\bullet_{\check{T}}(\Gr_{_{SL_2}})^\wedge_{S_\alpha}$ is pro-finitely complete in $\Fun(\sX^{-\rho}_\alpha, S_\alpha)$, and $\End(Q(\bm{-\infty}_\alpha )^{\leq \bm{\infty}_\alpha}_{S_\alpha})$ is by definition the projective limit of $\End(Q(\bm{n}_\alpha )^{\leq \bm{m}_\alpha}_{S_\alpha})$, it shows the surjectivity of (\ref{equ 6.16}). 
\end{proof}

\ 

\subsubsection{The $\alpha$-regular case} 
As in \textsection \ref{subsect 6.1.2}, we consider the $\alpha$-regular weight $\omega$ satisfying $0<\langle \omega+\rho, \check{\alpha} \rangle <l$.  

\begin{prop}\label{prop Chap7-15} 
For any $n\leq m$, the following composition is surjective, 
$$H^\bullet_{\check{T}}(\Fl_{_{SL_2}})^\wedge_{S_\alpha}  
\rightarrow Z(\sO^{\sX^\omega_\alpha}_{S_\alpha}) \rightarrow 
	\End(Q(\bm{2n-1}_{\alpha,\omega})^{\leq \bm{2m}_{\alpha,\omega}}_{S_\alpha}).$$ 
\end{prop}
\begin{proof} 
We omit the subscripts ``$\alpha,\omega$". 
Without loss of generality, we may assume $n=0$. 
We prove the proposition by showing that as an $S_\alpha$-algebra, 
$\End(Q(\bm{-1})^{\leq \bm{2m}}_{S_\alpha})$ can be generated by the image of $\phi'_\alpha$ and $\psi_\alpha$. 
In the paragraphs below, we will not distinguish $\phi'_\alpha, \psi_\alpha$ with their images under any map. 
For a morphism $x$ in $\sO_{ S_\alpha}$, we denote by $\overline{x}$ its specialization in $\sO_{ {\Bbbk_\alpha}}$. 

We study two classes of endomorphisms of $Q(\bm{-1})^{\leq \bm{2m}}_{S_\alpha}$ as follows. 

\textit{Class 1. The $\sT'\tilde{x}_j$'s.} 
Applying the translation functor $\sT'$ to (\ref{equ 6.15}), we obtain a commutative diagram 
\begin{equation}\label{equ 6.18} 
 \begin{tikzcd} 
 Q(\bm{-1})^{\leq \bm{2m}}_{S_\alpha} 
 \arrow[r,dashrightarrow,"\sT'\phi_\alpha"'] \arrow[d,two heads] \arrow[rrr,dashrightarrow,bend left=10,"\sT'\tilde{x}_j"] & 
 Q(\bm{1})^{\leq \bm{2m}}_{S_\alpha} 
 \arrow[r,dashrightarrow,"\sT'\phi_\alpha-\alpha"']\arrow[d,two heads] & 
 \cdots \cdots \arrow[r,dashrightarrow,"\sT'\phi_\alpha-(j-1)\alpha"'] & 
 Q(\bm{2j-1})^{\leq \bm{2m}}_{S_\alpha} \arrow[d,two heads] \\ 
 Q(\bm{-1})^{\leq \bm{0}}_{{\Bbbk_\alpha}} 
 \arrow[r,hook] & 
 Q(\bm{1})^{\leq \bm{2}}_{{\Bbbk_\alpha}} \arrow[r,hook] & \cdots \cdots \arrow[r,hook] & 
 Q(\bm{2j-1})^{\leq \bm{2j}}_{{\Bbbk_\alpha}}. \end{tikzcd} 
 \end{equation}
Note that $\sT'\phi_\alpha$ acts on $Q(\bm{2k-1})^{\leq \bm{2k}}_{S_\alpha}=\sT'M(\bm{k}_\alpha)_{S_\alpha}$ via scalar $\phi_\alpha(\bm{k}_\alpha)=k\alpha$. 
Therefore, the coordinate for $\sT'\phi_\alpha$ in 
$$\End(Q(\bm{2j-1})^{\leq \bm{2m}}_{S_\alpha}) \otimes_{S_\alpha} \K 
	=\prod_{i=2j-1}^{2m} \End(M(\bm{i})_{\K})=\prod_{i=2j-1}^{2m}\K $$ 
reads as $((2j-1)\alpha,(2j-1)\alpha,\cdots ,m\alpha,m\alpha)$, which coincides with the one for $\phi'_\alpha$. 
Hence $\sT'\phi_\alpha=\phi'_\alpha$ as endomorphisms of $Q(\bm{2j-1})^{\leq \bm{2m}}_{S_\alpha}$. 

We will need the following claim. 
\begin{claim}\label{claim 6.15} 
The composition 
\begin{equation}\label{equ 6.17} 
Q(\bm{-1})^{\leq \bm{0}}_{{\Bbbk_\alpha}} \hookrightarrow Q(\bm{2j-1})^{\leq \bm{2j}}_{{\Bbbk_\alpha}} \twoheadrightarrow M(\bm{2j-1})_{{\Bbbk_\alpha}}
\end{equation}
is nonzero. 
\end{claim}
\begin{proof}[Proof of Claim \ref{claim 6.15}] 
Applying $\sT$ to the maps above, we get $M(\bm{0}_\alpha)_{{\Bbbk_\alpha}}^{\oplus 2}\hookrightarrow M(\bm{j}_\alpha)_{{\Bbbk_\alpha}}^{\oplus 2} \twoheadrightarrow M(\bm{j}_\alpha)_{{\Bbbk_\alpha}}$. 
By Lemma \ref{lem 6.2}, we have $\dim\Hom(M(\bm{0}_\alpha)_{{\Bbbk_\alpha}},M(\bm{j}_\alpha)_{{\Bbbk_\alpha}})=1$, hence the composition above is nonzero, so is (\ref{equ 6.17}). 
\end{proof} 

\textit{Class 2. The $\psi_\alpha\circ \sT'\tilde{x}_j$'s.} 
We view $\psi_\alpha$ as an endomorphism of $Q(\bm{2j-1})^{\leq \bm{2m}}_{S_\alpha}$, then there is a morphism ${\psi_\alpha}\circ {\sT'\tilde{x}_j}: 
Q(\bm{-1})^{\leq \bm{2m}}_{S_\alpha}\rightarrow Q(\bm{2j-1})^{\leq \bm{2m}}_{S_\alpha}$. 
Note that the truncation $\tau^{\leq \bm{2j}}(\overline{\psi_\alpha})$ as an endomorphism of $Q(\bm{2j-1})^{\leq \bm{2j}}_{{\Bbbk_\alpha}}$ is also provided by $\overline{\psi_\alpha}$. 
By Lemma \ref{lem 6.8}(3), there is a short exact sequence 
$$0\rightarrow M(\bm{2j})_{{\Bbbk_\alpha}} \rightarrow Q(\bm{2j-1})^{\leq \bm{2j}}_{{\Bbbk_\alpha}} \rightarrow M(\bm{2j-1})_{{\Bbbk_\alpha}} \rightarrow 0.$$ 
As in Claim \ref{claim 6.10}, $\tau^{\leq \bm{2j}}(\overline{\psi_\alpha})$ can be factorized into 
\begin{equation}\label{equ 6.19} 
Q(\bm{2j-1})^{\leq \bm{2j}}_{{\Bbbk_\alpha}} \twoheadrightarrow M(\bm{2j-1})_{{\Bbbk_\alpha}} \hookrightarrow M(\bm{2j})_{{\Bbbk_\alpha}} \hookrightarrow Q(\bm{2j-1})^{\leq \bm{2j}}_{{\Bbbk_\alpha}}.
\end{equation} 

\ 

We now abuse $\sT'\tilde{x}_j$ and $\psi_\alpha\circ \sT'\tilde{x}_j$ as endomorphisms of $Q(\bm{-1})^{\leq \bm{2m}}_{S_\alpha}$ via the inclusion $Q(\bm{2j-1})^{\leq \bm{2m}}_{S_\alpha}\hookrightarrow Q(\bm{-1})^{\leq \bm{2m}}_{S_\alpha}$ given by Lemma \ref{lem 6.8}(3). 
Since $Q(\bm{-1})^{\leq \bm{2m}}_{{\Bbbk_\alpha}}$ is projective in $\sO^{\leq \bm{2m}}_{{\Bbbk_\alpha}}$ and it admits the Verma factors in (\ref{equ 6.13}), there is an (non-canonical) isomorphism of ${\Bbbk_\alpha}$-vector spaces 
\begin{equation}\label{equ 5.16} 
\End(Q(\bm{-1})^{\leq \bm{2m}}_{{\Bbbk_\alpha}})=\bigoplus_{j=-1}^{2m} \Hom(Q(\bm{-1})^{\leq \bm{2m}}_{{\Bbbk_\alpha}},M(\bm{j})_{{\Bbbk_\alpha}})={\Bbbk_\alpha}^{\oplus 2(m+1)}, 
\end{equation} 
By (\ref{equ 6.18}), (\ref{equ 6.17}) and (\ref{equ 6.19}), one can find an isomorphism (\ref{equ 5.16}) such that $\{\overline{\sT'\tilde{x}_j},\overline{\psi_\alpha}\circ \overline{\sT'\tilde{x}_j} \}_{j=0}^m$ correspond to the standard basis of the RHS. 
By Nakayama's Lemma, $\{{\sT'\tilde{x}_j}, {\psi_\alpha}\circ {\sT'\tilde{x}_j} \}_{j=0}^m$ forms a free $S_\alpha$-basis of $\End(Q(\bm{-1})^{\leq \bm{2m}}_{S_\alpha})$, and note that these elements can be generated by $\phi'_\alpha$ and $\psi_\alpha$. 
It completes the proof. 
\end{proof} 

Similar arguments as in Corollary \ref{cor 6.11} show that the composition of the inclusions 
$$H^\bullet_{\check{T}}(\Fl_{_{SL_2}})^\wedge_{S_\alpha} \rightarrow Z(\sO^{\sX^\omega_\alpha}_{S_\alpha}) \rightarrow 
	\End(Q(\bm{-\infty}_{\omega,\alpha} )^{\leq \bm{\infty}_{\omega,\alpha}}_{S_\alpha})$$
is surjective, which implies that 
\begin{corollary}\label{cor 6.14} 
The map 
$\bb: H^\bullet_{\check{T}}(\Fl_{_{SL_2}})^\wedge_{S_\alpha} \rightarrow Z(\sO^{\sX^{\omega}_\alpha }_{S_\alpha})$ 
is an isomorphism. 
\end{corollary} 

\ 

We are now in the position to prove Theorem \ref{thm 3.11}. 
\subsection{Center of $\sO_{S}$} 
The following lemma is clear. 
\begin{lem}\label{lem 6.160}
For any height $1$ prime ideal $\p$ of $S$, either the set $\p\cap \Phi^+$ is empty or $\p=(\alpha)$ for some $\alpha\in {\Phi}^+$. 
In the latter case, $\p\cap \Phi^+=\{\alpha\}$. 
\end{lem} 

\ 

\begin{proof}[Proof of the isomorphism $\bb$ in Theorem \ref{thm 3.11}] 
By Corollary \ref{cor 6.11} and Corollary \ref{cor 6.14}, the map (\ref{equ 6.B17}) is an isomorphism for any $\sX\in W_l(\alpha)\backslash \Lambda$. 
Hence (\ref{equ 6.B15}) is an isomorphism, 
$$\bb: H^\bullet_{\check{T}}(\Gr^\zeta)^\wedge_{S_\alpha}\xs Z(\sO_{S_\alpha}).$$ 
By \cite[Thm 38]{Mat1970}, we have $S=\cap_{\hgt(\p)=1}S_\p$ in $\K$. 
Since any projective module $Q$ in the truncated category $\sO^{\leq\nu}_{S}$ is free as an $S$-module, we have $Q=\cap_{\hgt(\p)=1}Q\otimes_S S_\p\subset Q\otimes_S \K$. 
It implies that 
\begin{equation}\label{equ 6.23} 
Z(\sO_{S})=\bigcap_{\mathrm{ht}(\p)=1} Z(\sO_{S_\p})
\end{equation}
as subspaces of $Z(\sO_{\K})=\Fun(\Lambda,\K)$. 
By Lemma \ref{lem 3.6}, for any prime ideal $\p$ of $S$ such that $\p\cap \Phi^+=\emptyset$, the category $\sO_{\Bbbk_\p}$ is semisimple, and there is an isomorphism $\chi_{S_\p}: Z(\sO_{S_\p})\xs \Fun(\Lambda, S_\p)$. 
Finally, using Lemma \ref{lem 6.160} and Proposition \ref{prop B.1}, we have the following equalities 
\begin{align*} 
Z(\sO_{S}) &= \bigcap_{\mathrm{ht}(\p)=1} Z(\sO_{S_\p} ) \\ 
&= \bigcap_{\p\cap \Phi^+=\emptyset} \Fun(\Lambda , S_\p)\cap \bigcap_{\alpha\in \Phi^+} Z(\sO_{S_\alpha} ) \\ 
&= \Fun(\Lambda , S)\cap \bigcap_{\alpha\in \Phi^+} H^\bullet_{\check{T}}(\Gr^\zeta)^\wedge_{S_\alpha} = H^\bullet_{\check{T}}(\Gr^\zeta)^\wedge_{S}, 
\end{align*} 
as subspaces in $\Fun(\Lambda,\K)$. 
\end{proof}

\subsection{Center of $\sO_{\hat{S}}$} 
We prove the isomorphism $\hat{\bb}$ in Theorem \ref{thm 3.11}. 
In this case, the subgeneric deformations are similar with the classical one in \cite{Fie03}. 
We briefly sketch the proof by following the line in the \textit{loc. cit}. 

For any prime ideal $\p$ of $\hat{S}$, we denote by $\Bbbk_\p$ the residue field of $\hat{S}_\p$. 
If $\p=(\alpha)$ for some $\alpha\in \check{\Phi}_{l,\re}$, then we abbreviate $\hat{S}_\alpha=\hat{S}_\p$ and denote its residue field by $\hat{\Bbbk}_\alpha$. 
The following lemma is clear. 
\begin{lem}\label{lem 6.16} 
For any height $1$ prime ideal $\p$ of $\hat{S}$, either the set $\p\cap \check{\Phi}^+_{l,\re}$ is empty or $\p=(\alpha)$ for some $\alpha\in \check{\Phi}^+_{l,\re}$. 
In the latter case, $\p\cap \check{\Phi}^+_{l,\re}=\{\alpha\}$. 
\end{lem}
Let $\alpha\in \check{\Phi}_{l,\re}$, and consider the subgroup $\langle s_\alpha \rangle$ generated by $s_\alpha$ in $W_{l,\af}$. 
\begin{lem}\label{lem 6.17} 
There is a block decomposition 
\begin{equation}\label{equ 6.21} 
\sO_{\hat{S}_\alpha}= \bigoplus_{\sX\in (\langle s_\alpha \rangle, \bullet) \backslash \Lambda} \sO^{\sX}_{\hat{S}_\alpha}, 
\end{equation} 
such that the Verma module $M(\lambda)_{\hat{S}}$ lies in $\sO^{\sX}_{\hat{S}}$ if and only if $\lambda\in \sX$. 
If $s_\alpha\bullet \lambda=\lambda$, then $M(\lambda)_{\hat{S}_\alpha}$ is projective in $\sO^{\{\lambda\}}_{\hat{S}}$. 
If $s_\alpha\bullet \lambda>\lambda$, then $M(s_\alpha\bullet \lambda)_{\hat{S}_\alpha}$ is projective in $\sO^{\{\lambda, s_\alpha\bullet \lambda\}}_{\hat{S}_\alpha}$, and the projective cover $Q(\lambda)_{\hat{S}_\alpha}$ of $E(\lambda)_{\hat{\Bbbk}_\alpha}$ fits into a short exact sequence 
\begin{equation}\label{equ 6.22} 
0\rightarrow M(s_\alpha\bullet \lambda)_{\hat{S}_\alpha}\rightarrow Q(\lambda)_{\hat{S}_\alpha} \rightarrow M(\lambda)_{\hat{S}_\alpha}\rightarrow 0.
\end{equation} 
\end{lem}
\begin{proof} 
As in Proposition \ref{prop 6.1}, one can show the Jantzen sum formula and the linkage principle for $\sO_{\hat{S}_\alpha}$, which imply the block decomposition (\ref{equ 6.21}). 
The short exact sequence (\ref{equ 6.22}) is proved as in \cite[Prop 3.4]{Fie03}, using the Jantzen sum formula. 
\end{proof} 

\begin{lem}\label{lem 6.19} 
The inclusion $\chi_{\hat{S}_\alpha}:Z(\sO_{\hat{S}_\alpha})\rightarrow \Fun(\Lambda, \hat{S}_\alpha)$ induces an isomorphism 
$$Z(\sO_{\hat{S}_\alpha})=\{ f\in \Fun(\Lambda, \hat{S}_\alpha)|\ f(\lambda)-f(s_\alpha\bullet \lambda)\in \alpha \hat{S}_\alpha,\ \forall \lambda\in \Lambda \}.$$ 
\end{lem}
\begin{proof}
It is proved similarly as in \cite[Cor 3.5]{Fie03}. 
\end{proof}

\ 

\begin{proof}[Proof of the isomorphism $\hat{\bb}$ in Theorem \ref{thm 3.11}] 
As in (\ref{equ 6.23}), there is an equality as subspaces in $Z(\sO_{\hat{\K}})=\Fun(\Lambda,\hat{\K})$, 
$$Z(\sO_{\hat{S}})= \bigcap_{\mathrm{ht}(\p)=1} Z(\sO_{\hat{S}_\p}).$$ 
By Lemma \ref{lem 3.6}, the category $\sO_{\Bbbk_\p}$ is semisimple if $\p\cap \check{\Phi}^+_{l,\re}=\emptyset$, and in this case there is an isomorphism $\chi_{\hat{S}_\p}:Z(\sO_{\hat{S}_\p})\xs \Fun(\Lambda,\hat{S}_\p)$. 
Using Lemma \ref{lem 6.19}, we have 
\begin{align*} 
Z(\sO_{\hat{S}}) &= \bigcap_{\mathrm{ht}(\p)=1} Z(\sO_{\hat{S}_\p}) \\ 
&=\bigcap_{\p\cap \check{\Phi}^+_{l,\re}=\emptyset} \Fun(\Lambda,\hat{S}_\p)\cap \bigcap_{\alpha\in \check{\Phi}^+_{l,\re}} Z(\sO_{\hat{S}_\alpha}) \\ 
&=\{f\in \Fun(\Lambda, \hat{S})|\ f(\lambda)-f(s_\alpha\bullet \lambda)\in \alpha \hat{S},\ \forall \lambda\in \Lambda, \forall \alpha\in \check{\Phi}^+_{l,\re}\}. 
\end{align*} 
The last term is known as the GKM description for the cohomology $H^\bullet_{\check{T}\times \Gm}(\Gr^\zeta)$, which coincides with the completion $H^\bullet_{\check{T}\times \Gm}(\Gr^\zeta)^\wedge_{\hat{S}}$, see \cite{KoKu86}. 
\end{proof}

\begin{rmk}
Note that the proof above does not use the existence of the homomorphism $\hat{\bb}$ in Proposition \ref{prop 3.7}. 
Note also that the map ${\bb}$ can be induced by $\hat{\bb}$. 
Therefore, our result gives an alternative proof of Proposition \ref{prop 3.7}. 
\end{rmk}

\ 

\appendix 
\section{The GKM theory for affine flag varieties}\label{app B} 

Let $J$ be a subset of $\check{\Sigma}_\af$. 
It associates a partial affine flag variety $\Fl^J$, and recall $\Fl^{J,\circ}$ is the neutral component of $\Fl^J$, see \textsection\ref{nsubsect 5.1.1}. 
For $\alpha\in \Phi^+$, recall the codimension one subtorus $\check{T}_\alpha=(\ker \check{\alpha})^\circ$ in $\check{T}$. 
In this appendix, we prove the following proposition. 

\begin{prop}\label{prop B.0} 
\begin{enumerate} 
\item The restriction induces an isomorphism 
$$H^\bullet_{\check{T}}(\Fl^{J,\circ})^\wedge_{S_\alpha} \xs H^\bullet_{\check{T}}\big((\Fl^{J,\circ})^{\check{T}_\alpha}\big)^\wedge_{S_\alpha}. $$
\item As subspaces in $\Fun(W^J_\af,\K)$, there is an equality 
$$H^\bullet_{\check{T}}(\Fl^{J,\circ})^\wedge_{S} =\Fun(W^J_\af,S)\cap 
\bigcap_{\alpha\in \Phi^+} H^\bullet_{\check{T}}\big((\Fl^{J,\circ})^{\check{T}_\alpha}\big)^\wedge_{S_\alpha}.$$ 
\end{enumerate} 
\end{prop} 

For $\alpha\in \Phi^+$, we set $W(\alpha)$ as the subgroup of $W_{\af}$ generated by $s_{\alpha,n}$, $n\in \Z$. 
The $\check{T}_\alpha$-fixed locus of $\Fl^{J,\circ}$ decomposes as 
$$(\Fl^{J,\circ})^{\check{T}_\alpha}=\bigsqcup_{\sX\in W(\alpha) \backslash W^J_{\af}} \Fl^{\sX}_{_{SL_2}},$$ 
where $\Fl^{\sX}_{_{SL_2}}$ is the connected component with $\sX= (\Fl^{\sX}_{_{SL_2}})^{\check{T}}$. 
By abuse of notation, we let $\Fl_{_{SL_2}}$ and $\Gr_{_{SL_2}}$ be the affine flag variety and the affine Grassmannian associated with $SL_2((t))$. 
Then there are isomorphisms of $\check{T}$-varieties 
$$\Fl^{\sX}_{_{SL_2}}\simeq 
\begin{cases}
\Fl_{_{SL_2}} & \text{for any bijection $\sX=W(\alpha)$ of $W(\alpha)$-sets,} \\ 
\Gr_{_{SL_2}} & \text{for any bijection $\sX=\Z\alpha$ of $W(\alpha)$-sets,} 
\end{cases}$$ 
via suitable isomorphism $\check{T}/\ker \check{\alpha}=T_2/\{\pm I_2\}$, such that the identifications of $W(\alpha)$-sets coincide with the variety isomorphisms restricting at the $\check{T}$-fixed points. 

\subsection{GKM theory} 
We recall the GKM theory for the equivariant cohomology of $\Fl^{J,\circ}$, which describes the images of the embeddings 
\begin{equation}\label{equ B.1} 
H^\bullet_{\check{T}}(\Fl^{J,\circ})\hookrightarrow \Fun(W^J_\af,S') , \quad H^\bullet_{\check{T}\times \Gm}(\Fl^{J,\circ})\hookrightarrow \Fun(W^J_\af,\hat{S}') . 
\end{equation} 
We firstly introduce some notations. 

\subsubsection{GKM descriptions} 
Let $R$ be an $\hat{S}'$-algebra. 
Let $\alpha\in \check{\Phi}_{\re}^+$ and let $\sX$ be a set endowed with an action of $\Z/2\Z=\langle s_\alpha\rangle$. 
We define the subspace (see e.g. \cite{Shim14}) 
\begin{equation}\label{equ A.2}
{\GKM}(\langle s_\alpha\rangle,\sX)_R:=\{f\in \Fun(\sX, R)|\ f(x)-f(s_\alpha(x))\in \alpha R,\ \forall x\in \sX\}.
\end{equation}
These modular equalities are referred to as the \textit{GKM descriptions} with respect to $\alpha$. 
If $\sX$ is moreover a $W_{\af}$-set, we define 
$${\GKM}(W_\af,\sX)_R:= \bigcap_{\alpha\in \check{\Phi}_{\re}^+}{\GKM}(\langle s_\alpha\rangle,\sX)_R. $$ 
We denote by ${\GKM}_f(W_\af,\sX)_{\hat{S}'}$ the $\hat{S}'$-submodule generated by homogeneous elements in ${\GKM}(W_\af,\sX)_{\hat{S}'}$. 

Now we let $R$ be an $S'$-algebra. 
Let $\alpha\in \Phi^+$ and let $\sX$ be a $W(\alpha)$-set.  
Following \cite[\textsection 4.2]{Shim14}, we define the subspace 
\begin{equation}\label{equ B.0} 
{\sGKM}(W(\alpha),\sX)_R := 
	\left\{ 
	\begin{gathered} 
    f\in \Fun( \sX , R )
    \end{gathered}  \left |\ 
    \begin{gathered} 
	f(s_{\alpha,m}(x))-f(x)\in \alpha R , \quad \forall m\in \Z \\ 
	f((1 - \tau_{\alpha})^d x)\in \alpha^d R , \quad \forall d\geq0 \\ 
	f((1 - \tau_{\alpha})^{d-1}(1-s_\alpha) x)\in \alpha^d R , \quad \forall d\geq0 
	\end{gathered} \quad \forall x\in \sX 
    \right. \right\} ,
\end{equation}
where we view $\Fun(\sX,R)= \Hom_{R}(R[\sX], R)$. 
The modular equalities above are referred to as the \textit{small torus GKM descriptions} with respect to $\alpha$ in the \textit{loc. cit}. 
If $\sX$ is moreover a $W_{\af}$-set, we define 
$${\sGKM}(W_\af,\sX)_R:= \bigcap_{\alpha\in \Phi^+}{\sGKM}(W(\alpha),\sX)_R. $$ 
We define ${\sGKM}_f(W_\af,\sX)_{S'}$ similarly as above. 

\begin{prop}[{\cite[Thm 3.45, Thm 4.3]{Shim14}}]\label{prop B.2} 
The restrictions to $\check{T}$-fixed points (\ref{equ B.1}) give isomorphisms 
\begin{equation}\label{equ B.2} 
H^\bullet_{\check{T}}(\Fl^{J,\circ})={\sGKM}_f(W_\af,W^J_\af)_{S'}, \quad H^\bullet_{\check{T}\times \Gm}(\Fl^{J,\circ})={\GKM}_f(W_\af,W^J_\af)_{\hat{S}'}. 
\end{equation} 
\end{prop} 

\subsubsection{Solution of GKM conditions}\label{A.1.2} 
In \cite[Prop 3.25]{Shim14} the author constructed a family $\{\xi^{J,x}\}_{x\in W^J_\af}$ in ${\GKM}(W_\af,W^J_\af)_{\hat{S}'}$, which is uniquely characterized by the following properties: for any $x,y\in W^J_\af$, 
\begin{equation}\label{equ B.3} 
	\begin{aligned} 
    &\text{(1) $\xi^{J,x}(y)$ is homogeneous of degree $l(x)$; } \\ 
    &\text{(2) $\xi^{J,x}(y)=0$ unless $x\leq y$ (Bruhat order); } \\ 
	&\text{(3) $\xi^{J,x}(x)=\prod\limits_{\alpha \in \check{\Phi}^+_\re,\ s_\alpha x<x} \alpha$. } 
\end{aligned}
\end{equation} 
Denote by $\{\overline{\xi}^{J,x}\}_{x\in W^J_\af}$ the specialization of $\{\xi^{J,x}\}_{x\in W^J_\af}$ in ${\Fun}(W^J_\af,S')$ by the projection $\hat{S}'\rightarrow S'$, $\hbar \mapsto 0$. 

\begin{prop}[{\cite[Thm 3.45, Thm 4.3]{Shim14}}]\label{prop B.3} 
The family $\{\xi^{J,x}\}_{x\in W^J_\af}$ (resp. $\{\overline{\xi}^{J,x}\}_{x\in W^J_\af}$) forms an $\hat{S}'$-basis of ${\GKM}_f(W_\af,W^J_\af)_{\hat{S}'}$ (resp. an $S'$-basis of ${\sGKM}_f(W_\af,W^J_\af)_{S'}$). 
The isomorphisms (\ref{equ B.2}) send the Schubert class $[\Fl^{J,x}]_{\check{T}\times \Gm}$ (resp. $[\Fl^{J,x}]_{\check{T}}$) to $\xi^{J,x}$ (resp. $\overline{\xi}^{J,x}$), for any $x\in W^J_\af$. 
\end{prop} 
%for $\Fl^J=G((t))/P^{J,-}$ ? 

\subsection{Proof of Proposition \ref{prop B.0}} 
Let $R$ be a flat $\hat{S}'$-algebra (resp. $S'$-algebra) such that the structure map $\hat{S}'\rightarrow R$ (resp. $S'\rightarrow R$) is an injection. 
We view $\check{\Phi}^+_\re$ (resp. $\Phi^+$) as a subset in $R$. 

\begin{lem}[{\cite[Prop 3.26, (B.3)]{Shim14}}]\label{lem B.5} 
Suppose $R$ is a UFD, and the elements in $\check{\Phi}^+_\re\backslash R^\times$ (resp. $\Phi^+\backslash R^\times$) are coprime to each other. 
Then we have an equality 
$$\GKM(W_\af,W^J_\af)_R= \prod_{x\in W^J_\af} R\cdot {\xi}^{J,x} ,$$ 
$$\text{(resp.}\quad \sGKM(W_\af,W^J_\af)_R=\prod_{x\in W^J_\af} R\cdot \overline{\xi}^{J,x}\quad \text{)},$$ 
where the elements of the RHS are viewed as formal sums of ${\xi}^{J,x}$'s (resp. $\overline{\xi}^{J,x}$'s) (which are well-defined by Properties (\ref{equ B.3})(2)). 
\end{lem} 
\begin{proof} 
Suppose $R$ is an $\hat{S}'$-algebra. 
Given a nonzero element $f\in \GKM(W_\af,W^J_\af)_R$, one can choose a minimal $x\in W^J_\af$ such that $f(x)\neq 0$, i.e. $f(y)=0$ for any $y<x$. 
Then for any $\alpha\in \check{\Phi}^+_{\re}$ such that $s_\alpha x<x$, we have $f(x)=f(x)-f(s_\alpha x)\in \alpha R$ by (\ref{equ A.2}). 
By our assumptions on $R$ and Properties (\ref{equ B.3})(3), it follows that 
\begin{equation}\label{equ A.6}
f(x)\in {\xi}^{J,x}(x)R. 
\end{equation}
Hence $f- \frac{f(x)}{{\xi}^{J,x}(x)}{\xi}^{J,x}$ belongs to $\GKM(W_\af,W^J_\af)_R$, which vanishes on $x$. 
Keeping this procedure one can express $f$ into a formal sum of ${\xi}^{J,y}$'s, and such expression is unique by Properties (\ref{equ B.3})(2). 

The proof is similar for the case when $R$ is an ${S}'$-algebra, and we refer to \cite[Proof of Thm 4.3]{Shim14} for the proof of (\ref{equ A.6}) in this case. 
\end{proof} 

\ 

\begin{proof}[Proof of Proposition \ref{prop B.0}] 
(1) We have to show the isomorphism 
$$H^\bullet_{\check{T}}(\Fl^{J,\circ})^\wedge_{S_\alpha} \xs 
\prod_{\sX} H^\bullet_{\check{T}}(\Fl^{\sX}_{_{SL_2}})^\wedge_{S_\alpha}.$$ 
Indeed, by Propositions \ref{prop B.2}, \ref{prop B.3} and Lemma \ref{lem B.5}, the LHS coincides with $\sGKM(W_\af,W^J_\af)_{S_\alpha}$, and the RHS coincides with $\prod_{\sX} \sGKM(W(\alpha),\sX)_{S_\alpha}$. 
The conclusion follows from the equalities 
$$\sGKM(W_\af,W^J_\af)_{S_\alpha}=\sGKM(W(\alpha),W^J_\af)_{S_\alpha}
=\prod_{\sX}\sGKM(W(\alpha),\sX)_{S_\alpha}.$$ 

(2) By Propositions \ref{prop B.2}, \ref{prop B.3} and Lemma \ref{lem B.5}, we have $H^\bullet_{\check{T}}(\Fl^{J,\circ})^\wedge_{S}=\sGKM(W_\af,W^J_\af)_{S}$. 
The conclusion follows from the proof of (1) and the equalities 
$$\sGKM(W_\af,W^J_\af)_{S}=\Fun(W^J_\af,S)\cap \bigcap_{\alpha\in \Phi^+} \sGKM(W(\alpha),W^J_\af)_{S_\alpha}. \qedhere $$ 
\end{proof}

\section{The natural isomorphism $\Upsilon$}\label{App C} 
In this appendix, we prove Proposition \ref{prop 5.0} (including an enhanced statement in \textsection\ref{subsect B.4}). 
Our construction essentially appeared in the works of Anderson--Jantzen--Soergel \cite{AJS94} for small quantum groups (and also restricted enveloping algebras) and of Fiebig \cite{Fie03} for symmetrizable Kac--Moody algebras. 
Without loss of generality, we fix an element $\omega\in \Xi_\sc$ and prove the proposition for $(\omega_1,\omega_2)=(0,\omega)$. 

\subsection{Preliminaries} 
Recall $W_{l,\omega}=\text{Stab}_{(W_{l,\af},{\bullet})}(\omega)$. 
We set $\check{\Phi}^+_{l,\omega}=\{\alpha\in \check{\Phi}^+_{l,\re}|\ s_\alpha\in W_{l,\omega}\}$. 
Consider the translation functors 
$$T^\omega_0: \sO^{0}_{\hat{S}}\rightarrow \sO^{\omega}_{\hat{S}},\quad 
T_\omega^0: \sO^{\omega}_{\hat{S}}\rightarrow \sO^{0}_{\hat{S}}.$$ 
We will simplify $\sT=T^\omega_0$ and $\sT'=T_\omega^0$. 
We fix the following data: 
\begin{itemize}
\item An adjunction $(\sT,\sT')$, i.e. a collection of isomorphisms 
$$\adj: \Hom(M,\sT'N)\xs \Hom(\sT M,N),$$ 
that is functorial with respect to $M\in \sO^{0}_{\hat{S}}$ and $N\in \sO^{\omega}_{\hat{S}}$; 
\item A family of isomorphisms $\{f_z\}_{z\in W_{l,\af}}$ where 
$$f_{z}: M(z\bullet \omega)_{\hat{S}}\xs \sT M(z\bullet 0)_{\hat{S}} \quad \text{(c.f. Lemma \ref{lem 5.0}(2))} .$$  
\end{itemize} 
They yield a family of homomorphisms $\{f'_z\}_{z\in W_{l,\af}}$ via 
$$f'_{z}:=\adj^{-1}(f_{z}^{-1}):\ M(z\bullet 0)_{\hat{S}} \rightarrow \sT'M(z\bullet \omega)_{\hat{S}}.$$ 

\begin{lem}\label{lem C.1} 
For any $z\in W_{l,\af}$, if $\p\cap z(\check{\Phi}^+_{l,\omega})=\emptyset$, then the morphism 
$$\bigoplus_{x\in W_{l,\omega}}f'_{zx}:\ 
\bigoplus_{x\in W_{l,\omega}} M(zx\bullet 0)_{\hat{S}_\p}\rightarrow \sT'M(z\bullet \omega)_{\hat{S}_\p}$$
is an isomorphism. 
\end{lem}
\begin{proof} 
By Lemma \ref{lem 5.0}, Lemma \ref{equ 6.17} and our assumption that $\p\cap z(\check{\Phi}_{l,\omega})=\emptyset$, the Verma factors $M(zx\bullet 0)_{\hat{S}_\p}$ of $\sT'M(z\bullet \omega)_{\hat{S}_\p}$ belong to different blocks in $\sO_{\hat{S}_\p}$. 
Hence there is an isomorphism 
$$\sT'M(z\bullet \omega)_{\hat{S}_\p}\simeq \bigoplus_{x\in W_{l,\omega}} M(zx\bullet 0)_{\hat{S}_\p}.$$ 
Since $f^{-1}_{zx}$ is a generator of the $\hat{S}_\p$-module 
$$\Hom(\sT M(zx\bullet 0)_{\hat{S}_\p},M(z\bullet \omega)_{\hat{S}_\p})\simeq \End(M(z\bullet \omega)_{\hat{S}_\p})=\hat{S}_\p,$$ 
the element $f'_{zx}=\adj^{-1}(f^{-1}_{zx})$ is a generator of the $\hat{S}_\p$-module $\Hom(M(zx\bullet 0)_{\hat{S}_\p}, \sT'M(z\bullet \omega)_{\hat{S}_\p})$ which is free of rank $1$. 
Hence $\bigoplus_{x\in W_{l,\omega}}f'_{zx}$ is an isomorphism. 
\end{proof}

\subsection{Constructions} 
\subsubsection{The matrix $\sH$}\label{subsect B.2.1} 
Recall the family $\{\xi^{z}\}_{z\in W_{l,\af}}$ in $\Fun(W_{l,\af},\hat{S})$ introduced in \textsection \ref{A.1.2} (here we identify $W_{l,\af}$ with $W_\af$). 
We denote $\bL_\omega:=\prod_{\alpha\in \check{\Phi}^+_{l,\omega}} \alpha$. 
For any $z\in  W_{l,\af}$, we define a $|W_{l,\omega}|\times |W_{l,\omega}|$-matrix with coefficients in $\hat{\K}$ by 
$$\sH_z=\big(zx(\bL^{-1}_\omega)\cdot \xi^{x'}(zx) \big)_{x,x'\in W_{l,\omega}}.$$ 

\begin{lem}\label{lem B..}
There is a matrix $\sH'_z\in GL_{|W_{l,\omega}|}(\hat{S})$ such that $\sH_z=z(\sH_1)\cdot \sH'_z$. 
\end{lem}
\begin{proof} 
We abbreviate $\GKM^\omega=\GKM(W_{l,\af},W^\omega_{l,\af})_{\hat{S}}$. 
We view $\Fun(W^\omega_{l,\af},\hat{S})$ as the subspace of $\Fun(W_{l,\af},\hat{S})$ of the functions constant on the $W_{l,\omega}$-cosets. 
By \cite[Prop 3.42]{Shim14}, we have $\GKM^\omega=\Fun(W^\omega_{l,\af},\hat{S})\cap \GKM^0$.  
It is known that $\{\xi^{x}\}_{x\in W_{l,\omega}}$ is a free basis of $\GKM^0$ as $\GKM^\omega$-module. 
Consider the operation on $\Fun(W_{l,\af},\hat{S})$ given by $(\varphi_y)_{y\in W_{l,\af}}\mapsto (z^{-1}\varphi_{zy})_{y\in W_{l,\af}}$. 
It preserves the subspace $\GKM^\omega$. 
Therefore, $\{z^{-1}\xi^{x}(z\cdot )\}_{x\in W_{l,\omega}}$ is also a free basis of $\GKM^0$ as $\GKM^\omega$-module. 
So there is a family $\{c^z_{x,x'}\}_{x,x'\in W_{l,\omega}}$ in $\GKM^\omega$ such that $z^{-1}\xi^{x}(z\cdot )=\sum_{x'} c^z_{x,x'}\cdot \xi^{x'}$ for any $x,x'\in W_{l,\omega}$. 
Finally, we set $\sH_z':=z\big(c^z_{x',x}(1)\big)_{x,x'}$, then 
\begin{align*}
\sH_z &=z\big(x(\bL^{-1}_\omega)\cdot z^{-1}\xi^{x'}(zx) \big)_{x,x'}\\ 
&=z\big(x(\bL^{-1}_\omega)\cdot \sum_{x''} c^z_{x',x''}(1)\cdot \xi^{x''}(x) \big)_{x,x'}= z(\sH_1)\cdot \sH_z'. \qedhere 
\end{align*}
\end{proof} 

By Properties (\ref{equ B.3})(2), if we index the elements in $W_{l,\omega}$ as $x_1,x_2,\cdots,x_n$ such that $x_i\ngtr x_j$ for any $i<j$, then the matrix $\sH_1$ is upper-triangular. 
So by Properties (\ref{equ B.3})(3), $\sH_1$ is invertible over $\hat{S}[\alpha^{-1}|\alpha\in \check{\Phi}^+_{l,\omega}]$. 
By Lemma \ref{lem B..}, $\sH_z$ is invertible over $\hat{S}[(z\alpha)^{-1}|\alpha\in \check{\Phi}^+_{l,\omega}]$. 

\begin{rmk}
The matrices $\{\sH_z\}_{z\in W_{l,\af}}$ appeared in \cite[Thm 5.9]{Fie03}.   
\end{rmk}

\subsubsection{Construction of $\Upsilon_{\hat{\K}}$} 
For any $z\in W_{l,\af}$, we define a morphism 
$$F_{z}:\ 
\bigoplus_{x\in W_{l,\omega}} M(z\bullet\omega)_{\hat{S}} 
\xrightarrow[\sim]{\bigoplus\limits_{x\in W_{l,\omega}} f_{zx}} 
\bigoplus_{x\in W_{l,\omega}} \sT M(zx\bullet 0)_{\hat{S}} 
\xrightarrow{\bigoplus\limits_{x\in W_{l,\omega}} \sT f'_{zx}} 
\sT\sT' M(z\bullet\omega)_{\hat{S}}.$$ 
Note that $F_{z}$ is independent on the choice of $\{f_z\}_{z\in W_{l,\af}}$. 
By Lemma \ref{lem C.1}, $F_{z}\otimes_{\hat{S}}\hat{\K}$ is an isomorphism. 
We define an isomorphism 
$$\Upsilon_{\hat{\K},z}:\ 
\bigoplus_{x\in W_{l,\omega}}M(z\bullet\omega)_{\hat{\K}}
\xrightarrow[\sim]{\sH_{z}} 
\bigoplus_{x\in W_{l,\omega}}M(z\bullet\omega)_{\hat{\K}} 
\xrightarrow[\sim]{F_z \otimes_{\hat{S}} \hat{\K}} \sT\sT'M(z\bullet\omega)_{\hat{\K}}.$$ 
Note that $\Upsilon_{\hat{\K},z}$ only depends on the class of $z$ in $W^\omega_{l,\af}$. 
Since the category $\sO^{\omega}_{\hat{\K}}$ is semisimple, and $\{M(z\bullet\omega)_{\hat{\K}}\}_{z\in W_{l,\af}}$ forms a complete set of simple objects, the collection $\{\Upsilon_{\hat{\K},z}\}_{z\in W_{l,\af}}$ determines a natural isomorphism 
$$\Upsilon_{\hat{\K}}:\ \id^{\oplus |W_{l,\omega}|}\xs \sT\sT'$$ 
between functors on $\sO^{\omega}_{\hat{\K}}$. 
Here is the main result of this section. 
\begin{prop}\label{prop C.3} 
There is a natural isomorphism 
$$\Upsilon_{\hat{S}}:\id^{\oplus |W_{l,\omega}|}\xs \sT\sT'$$ 
between functors on $\sO^{\omega}_{\hat{S}}$, such that $\Upsilon_{\hat{\K}}=\Upsilon_{\hat{S}}\otimes_{\hat{S}} \hat{\K}$. 
\end{prop} 

To prove Proposition \ref{prop C.3}, it is enough to show that for any $\nu\in \Lambda$, and any projective module $Q$ in $\sO_{\hat{S}}^{\omega,\leq \nu}$, the isomorphism $\Upsilon_{\hat{\K},Q\otimes_{\hat{S}}\hat{\K}}$ restricts to an isomorphism on the subspaces $Q^{^{\oplus |W_{l,\omega}|}}$ and $\sT\sT'Q$. 
Since $Q$ is a free $\hat{S}$-module, by \cite[Thm 38]{Mat1970} there is an equality $Q=\bigcap_{\hgt(\p)=1} Q\otimes_{\hat{S}}{\hat{S}_\p}$ in $Q\otimes_{\hat{S}}{\hat{\K}}$. 
Therefore, we only have to show that for any height $1$ prime ideal $\p$ of $\hat{S}$, the isomorphism $\Upsilon_{\hat{\K},Q\otimes_{\hat{S}}\hat{\K}}$ restricts to an isomorphism on the subspaces $(Q\otimes_{\hat{S}}{\hat{S}_\p})^{^{\oplus |W_{l,\omega}|}}$ and $\sT\sT'Q\otimes_{\hat{S}}{\hat{S}_\p}$. 
Combining Lemmas \ref{lem 3.6}, \ref{lem 6.16} and \ref{lem 6.17}, we reduce Proposition \ref{prop C.3} to the following proposition. 
\begin{prop}\label{prop C.4} 
\begin{enumerate} 
\item For any height $1$ prime ideal $\p$ of $\hat{S}$ and any $z\in W_{l,\af}$, 
the isomorphism $\Upsilon_{\hat{\K}, M(z\bullet\omega)_{\hat{\K}}}$ restricts to an isomorphism on the subspaces 
$$\Upsilon_{\hat{S}_\p,z}: \bigoplus_{x\in W_{l,\omega}} M(z\bullet\omega)_{\hat{S}_\p} \rightarrow \sT\sT'M(z\bullet\omega)_{\hat{S}_\p}.$$ 
\item For any $\alpha\in \check{\Phi}_{l,\re}^+$ and any $z\in W_{l,\af}$ such that $s_\alpha z\bullet \omega> z\bullet \omega$, the isomorphism $\Upsilon_{\hat{\K}, Q(z\bullet\omega)_{\hat{S}_\alpha}\otimes_{\hat{S}_\alpha}\hat{\K}}$ restricts to an isomorphism on the subspaces 
$$\Upsilon_{\alpha,z}: \bigoplus_{x\in W_{l,\omega}}Q(z\bullet\omega)_{\hat{S}_\alpha} \rightarrow \sT\sT'Q(z\bullet\omega)_{\hat{S}_\alpha}.$$ 
\end{enumerate}
\end{prop}

\subsection{Proof of Proposition \ref{prop 5.0}} 
\subsubsection{Lemmas} 
Let $\alpha\in \check{\Phi}^+_{l,\re}$ and $\lambda\in \Lambda$ satisfying $s_\alpha\bullet \lambda>\lambda$. 
Then short exact sequence (\ref{equ 6.22}) splits canonically over $\hat{\K}$, 
$$Q(\lambda)_{\hat{S}_\alpha}\otimes_{\hat{S}_\alpha} \hat{\K}\simeq M(\lambda)_{\hat{\K}} \oplus M(s_\alpha\bullet \lambda)_{\hat{\K}},$$ 
which induces an embedding 
\begin{equation}\label{equ C.2} 
\End(Q(\lambda)_{\hat{S}_\alpha})\hookrightarrow \End(M(\lambda)_{\hat{\K}})\times \End(M(s_\alpha\bullet \lambda)_{\hat{\K}})= \hat{\K}\times \hat{\K}.
\end{equation} 

\begin{lem}\label{lem C.5} 
The embedding (\ref{equ C.2}) identifies the subspaces 
$$\End(Q(\lambda)_{\hat{S}_\alpha})=\{(a,b)\in \hat{S}_\alpha\times \hat{S}_\alpha\ |\ b-a\equiv 0 \text{ mod }\alpha\}= \hat{S}_\alpha\{(1,1),(0,\alpha)\}.$$ 
\end{lem} 
\begin{proof} 
It is proved similarly as \cite[Prop 3.4]{Fie03} using the Jantzen filtration for $M(s_\alpha\bullet \lambda)_{\hat{S}_\alpha}$. 
\end{proof} 

The block decomposition (\ref{equ 6.21}) yields a canonical decomposition of functors on $\sO^{\omega}_{\hat{S}_\alpha}$ and $\sO^{0}_{\hat{S}_\alpha}$ 
\begin{equation}\label{equ 1-1.5} 
\sT'=\bigoplus_{\substack{z\in W_{l,\af},\\ s_\alpha z\bullet 0> z\bullet 0}} \sT'_{z},\quad 
\text{where}\quad 
\sT'_{z}: \sO^{\omega}_{\hat{S}_\alpha}\rightarrow 
\sO^{\{z\bullet 0,s_\alpha z\bullet 0\}}_{\hat{S}_\alpha}, 
\end{equation} 
\begin{equation}\label{equ 1-1.5.1} 
\sT=\bigoplus_{\substack{z\in W_{l,\af},\\ s_\alpha z\bullet 0> z\bullet 0}} \sT_{z} ,\quad 
\text{where}\quad 
\sT_{z}: \sO^{\{z\bullet 0,s_\alpha z\bullet 0\}}_{\hat{S}_\alpha} 
\rightarrow \sO^{\omega}_{\hat{S}_\alpha}. 
\end{equation} 
Note that $\sT'_{z}$ and $\sT_{z}$ depend on $\alpha$. 

Let $z\in W_{l,\af}$ and let $\mu= z\bullet \omega$, $\lambda =z\bullet 0$. 
Suppose $s_\alpha\bullet \mu=\mu$ and $s_\alpha\bullet \lambda>\lambda$. 
Since $\sT$, $\sT'$ are biadjoint to each other, they send projective objects to projective objects. 
By Lemma \ref{lem 6.17}, $M(\mu)_{\hat{S}_\alpha}$ is projective in $\sO^{\omega}_{\hat{S}_\alpha}$, so is $\sT'M(\mu)_{\hat{S}_\alpha}$ in $\sO^{0}_{\hat{S}_\alpha}$. 
As a direct summand, $\sT'_{z} M(\mu)_{\hat{S}_\alpha}$ is projective in $\sO^{\{\lambda,s_{\alpha}\bullet \lambda\}}_{\hat{S}_\alpha}$ and fits into a short exact sequence 
\begin{equation}\label{equ C.6} 
0 \rightarrow M(s_{\alpha}\bullet \lambda)_{\hat{S}_\alpha} \xrightarrow{i} 
\sT'_{z} M(\mu)_{\hat{S}_\alpha} \xrightarrow{j} 
M(\lambda)_{\hat{S}_\alpha} \rightarrow 0. 
\end{equation} 
By Lemma \ref{lem 6.17}, we have $Q(\lambda)_{\hat{S}_\alpha}\simeq \sT'_{z} M(\mu)_{\hat{S}_\alpha}$.

\begin{lem}[{\cite[\textsection 8.11]{AJS94}}]\label{lem C.6} 
There are morphisms 
\begin{equation}\label{equ 1-1.7} 
\begin{tikzcd}
0\arrow[r] 
& M(s_\alpha\bullet \lambda)_{\hat{S}_\alpha} \arrow[r,"f'_{s_\alpha z}"]
& \sT'_{z} M(\mu)_{\hat{S}_\alpha} \arrow[r,"g^{\alpha}_{z}"] \arrow[l,bend left,"g^{\alpha}_{s_\alpha z}"] 
& M(\lambda)_{\hat{S}_\alpha} \arrow[r] \arrow[l,bend left,"f'_{z}"] & 0,
\end{tikzcd}
\end{equation}
such that the arrows pointing right form a short exact sequence, and the following equalities hold 
\begin{equation}\label{equ 1-1.8}
g^{\alpha}_{s_\alpha z}\circ f'_{s_\alpha z}=\alpha\cdot \id, \quad 
g^{\alpha}_{z}\circ f'_{z}=\alpha\cdot \id, \quad 
g^{\alpha}_{s_\alpha z}\circ f_{z}' =0. 
\end{equation}
\end{lem}
\begin{proof} 
We simplify $\lambda_1=s_\alpha\bullet \lambda$, and 
$$f=f_{z},\quad f'=f'_{z}, \quad f_s=f_{s_\alpha z}, \quad f'_s=f'_{s_\alpha z}, \quad g=g^{\alpha}_{z},\quad g_s=g^{\alpha}_{s_\alpha z}, \quad \sT'=\sT'_{z},$$ 
and omit all the subscripts $\hat{S}_\alpha$. 
The proof is separated into three steps. 

\textit{Step 1. We can replace $i$ by $f'_s$ in the short exact sequence (\ref{equ C.6}).} 
Applying $\Hom(M(\lambda_1),-)$ to (\ref{equ C.6}), we have an exact sequence 
$$0\rightarrow \End(M(\lambda_1))\xs \Hom(M(\lambda_1), \sT'M(\mu)) \rightarrow \Hom(M(\lambda_1), M(\lambda))=0.$$ 
Since $i$ is the image of $1$ under the isomorphism 
$$\hat{S}_\alpha= \End(M(\lambda_1))\xs \Hom(M(\lambda_1), \sT'M(\mu)),$$
it is a generator of the latter $\hat{S}_\alpha$-module. 
On the other hand, by definition $f'_s$ is the image of an isomorphism $f$ under 
$$\Hom(\sT M(\lambda_1),M(\mu))\mathop{\xrightarrow{\adj^{-1}}}_{\sim} \Hom(M(\lambda_1), \sT' M(\mu)) ,$$
so it is also a generator of the latter module. 
It shows that $f'_s$ and $i$ are differed by an invertible scalar in $\hat{S}_\alpha$, so the claim follows. 

\textit{Step 2. We can choose $g_s$ such that $g_s\circ f'_s=\alpha\cdot \id$ and $g_s \circ f'=0$.} 
Recall there is an endomorphism of $\sT'M(\mu)$ represented by $(0,\alpha)$ by applying Lemma \ref{lem C.5} to the extension (\ref{equ C.6}) (recall we replaced $i$ by $f'_{s}$ in \textit{Step 1}). 
It can be factorized into 
$$\sT' M(\mu)\xrightarrow{g_s} M(\lambda_1) \xrightarrow{f'_s} \sT' M(\mu),$$
since $(0,\alpha)$ is zero on $M(\lambda)_{\hat{\K}}$. 
Moreover $g_s\circ f'_s= \alpha\cdot \id$. 
Note that the inclusion $\hat{S}_\alpha\subset \hat{\K}$ induces an embedding $\Hom(M(\lambda), M(\lambda_1))\subset \Hom(M(\lambda)_{\hat{\K}}, M(\lambda_1)_{\hat{\K}})$, and note also that the latter hom-space vanishes since $\lambda\neq \lambda_1$. 
It shows that $\Hom(M(\lambda), M(\lambda_1))=0$, hence the second equality follows. 

\textit{Step 3. We can replace $j$ by some $g$ in the short exact sequence (\ref{equ C.6}), such that $g\circ f'=\alpha\cdot \id$.} 
Arguing as in \textit{Step 1}, we can show $j$ and $f'$ are generators for the $\hat{S}_\alpha$-modules 
$$\Hom(\sT'M(\mu),M(\lambda))\quad \text{and}\quad \Hom(M(\lambda),\sT'M(\mu)).$$ 
As in \textit{Step 2}, there is an endomorphism of $\sT' M(\mu)$ represented by $(\alpha,0)$, which can be factorized into 
$$\sT' M(\mu) \xrightarrow{j} M(\lambda) \xrightarrow{cf'} \sT' M(\mu), \quad c\in \hat{S}_\alpha,$$ 
since $(\alpha,0)$ is zero on $M(\lambda_1)$. 
Moreover $j\circ cf'=\alpha\cdot \id$. 

On the other hand, the specialization of $j\circ f'$ in $\sO_{\hat{\Bbbk}_\alpha}$ must be zero, since $\sT'M(\mu)$ is indecomposable. 
Hence $j\circ f'=c'\alpha\cdot \id$ for some $c'\in \hat{S}_\alpha$. 
It follows that $c$ and $c'$ are inverse to each other in $\hat{S}_\alpha$. 
Replace $j$ by $g=cj$ then the claim follows. 
\end{proof} 

The following lemma can be proved as in \cite[Lem 8.12]{AJS94}, using Lemma \ref{lem C.6}. 
\begin{lem}[{\cite[Lem 8.12]{AJS94}}]\label{lem 1-1.8} 
There is an isomorphism 
$$h_{z}^\alpha:\sT\sT'_{z} M(\mu)_{\hat{S}_\alpha} \xs M(\mu)_{\hat{S}_\alpha}\oplus M(\mu)_{\hat{S}_\alpha}$$ 
such that the diagrams 
$$\begin{tikzcd}[ampersand replacement=\&] 
0 \arrow[r] \& \sT M(s_\alpha \bullet \lambda)_{\hat{S}_\alpha} \arrow[r,"\sT f'_{s_\alpha z}"] \arrow[d,"f_{s_\alpha z}^{-1}","\simeq"'] 
\& \sT\sT'_{z} M(\mu)_{\hat{S}_\alpha} \arrow[r,"\sT g^{\alpha}_{z}"] \arrow[d,"h_{z}^\alpha","\simeq"'] 
\& \sT M(\lambda)_{\hat{S}_\alpha} \arrow[r] \arrow[d,"f_{z}^{-1}","\simeq"'] \& 0 \\ 
0 \arrow[r] \& M(\mu)_{\hat{S}_\alpha} \arrow[r,"{\begin{pmatrix}1\\0\end{pmatrix}}"] 
\& M(\mu)_{\hat{S}_\alpha}\oplus M(\mu)_{\hat{S}_\alpha} \arrow[r,"{\begin{pmatrix}0&1\end{pmatrix}}"] 
\& M(\mu)_{\hat{S}_\alpha} \arrow[r] \& 0 
\end{tikzcd}$$ 
and 
$$\begin{tikzcd}[ampersand replacement=\&] 
\sT M(s_\alpha \bullet \lambda)_{\hat{S}_\alpha} \arrow[d,"f_{s_\alpha z}^{-1}","\simeq"'] 
\& \sT\sT'_{z} M(\mu)_{\hat{S}_\alpha} \arrow[l,"\sT g^{\alpha}_{s_\alpha z}"'] \arrow[d,"h_{z}^\alpha","\simeq"'] 
\& \sT M(\lambda)_{\hat{S}_\alpha} \arrow[l,"\sT f'_{z}"'] \arrow[d,"f_{z}^{-1}","\simeq"'] \\ M(\mu)_{\hat{S}_\alpha}
\& M(\mu)_{\hat{S}_\alpha}\oplus M(\mu)_{\hat{S}_\alpha} \arrow[l,"{\begin{pmatrix}\alpha & -1 \end{pmatrix}}"] 
\& M(\mu)_{\hat{S}_\alpha} \arrow[l,"{\begin{pmatrix}1 \\ \alpha \end{pmatrix}}"] 
\end{tikzcd}$$ 
commutes. 
\end{lem} 

\subsubsection{Proof of Proposition \ref{prop C.4}} 
\begin{proof}[Proof of Proposition \ref{prop C.4}(1)] 
By Lemma \ref{lem B..}, we may assume $z=1$, without loss of generality. 
The proof is divided into two cases. 

\textit{Case 1. Suppose $\p\cap \check{\Phi}^+_{l,\omega}=\emptyset$.} 
Then $\sH_1$ is an invertible matrix over $\hat{S}_\p$ by discussions in \textsection \ref{subsect B.2.1}. 
Combining with Lemma \ref{lem C.1}, it follows that the composition 
$$\bigoplus_{x\in W_{l,\omega}}M(\omega)_{\hat{S}_\p}
\xrightarrow{\sH_{1}} 
\bigoplus_{x\in W_{l,\omega}}M(\omega)_{\hat{S}_\p} 
\xrightarrow{F_1 \otimes_{\hat{S}} \hat{S}_\p} \sT\sT'M(\omega)_{\hat{S}_\p}$$ 
is an isomorphism, which is the desired $\Upsilon_{\hat{S}_\p,1}$. 

\textit{Case 2. Suppose $\p=(\alpha)$ for an $\alpha\in \check{\Phi}^+_{l,\omega}$.} 
In view of the block decomposition (\ref{equ 6.21}), the morphism 
$$F_1 \otimes_{\hat{S}} \hat{S}_\alpha: 
\bigoplus_{x\in W_{l,\omega}}M(\omega)_{\hat{S}_\alpha} 
\rightarrow \sT\sT'M(\omega)_{\hat{S}_\alpha}$$ 
is the direct sum of 
$$M(\omega)_{\hat{S}_\alpha}\oplus M(\omega)_{\hat{S}_\alpha}\\ 
\xrightarrow{f_{s_\alpha x}\oplus f_{x}} 
\sT M(s_\alpha x\bullet 0)_{\hat{S}_\alpha} \oplus \sT M(x\bullet 0)_{\hat{S}_\alpha}
\xrightarrow{\sT f'_{s_\alpha x}\oplus \sT f'_{x}} 
\sT\sT'_{x} M(\omega)_{\hat{S}_\alpha},$$ 
with $x$ ringing over the elements in $W_{l,\omega}$ such that $s_\alpha x\bullet 0>x\bullet 0$. 
By Lemma \ref{lem 1-1.8}, the composition 
$$M(\omega)_{\hat{S}_\alpha}\oplus M(\omega)_{\hat{S}_\alpha}\rightarrow \sT\sT'_{x} M(\omega)_{\hat{S}_\alpha} \mathop{\xrightarrow{h_{x}^\alpha}}_\sim M(\omega)_{\hat{S}_\alpha}\oplus M(\omega)_{\hat{S}_\alpha}$$
is represented by the matrix $\begin{pmatrix} 1 & 1 \\ 0 & \alpha \end{pmatrix}$. 
We order the elements of $W_{l,\omega}$ as $x_1, s_\alpha x_1, \cdots ,x_m, s_\alpha x_m$ 
such that $s_\alpha x_j\bullet 0>x_j\bullet 0$ for each $x_j$, then the morphism 
$$\big(\bigoplus_{s_\alpha x\bullet 0>x\bullet 0} h_{x}^\alpha \big)\circ (F_1 \otimes_{\hat{S}} \hat{S}_\alpha):\ \bigoplus_{x\in W_{l,\omega}}M(\omega)_{\hat{S}_\alpha}\rightarrow  \bigoplus_{x\in W_{l,\omega}}M(\omega)_{\hat{S}_\alpha}$$ 
can be represented by the matrix 
$$\diag\big(\begin{pmatrix} 1 & 1 \\ 0 & \alpha \end{pmatrix},\cdots, \begin{pmatrix} 1 & 1 \\ 0 & \alpha \end{pmatrix}\big).$$

On the other hand, by Lemma \ref{lem B.5}, the restriction of $\{\xi^x\}_{x\in W_{l,\omega}}$ on $\Fun(W_{l,\omega},\hat{S}_\alpha)$ gives an $\hat{S}_\alpha$-basis of 
\begin{equation}\label{equ C.7}
\GKM(\langle s_\alpha\rangle, W_{l,\omega})_{\hat{S}_\alpha}=\{f\in \Fun(W_{l,\omega},\hat{S}_\alpha)|\ f(x)-f(s_\alpha x)\in \alpha \hat{S}_\alpha,\ \forall x\in W_{l,\omega}\}. 
\end{equation}
Consider the functions $\delta_i, \delta'_i \in \Fun(W_{l,\omega},\hat{S}_\alpha)$ given by $\delta_i(x_j)=\delta_i(s_\alpha x_j)=\delta_{ij}$, $\delta'_i(x_j)=0$ and $\delta'_i(s_\alpha x_j)=\delta_{ij}\alpha$. 
Clearly $\{\delta_i,\delta'_i\}_{i=1}^m$ also forms an $\hat{S}_\alpha$-basis of (\ref{equ C.7}). 
Hence up to multiplication by an element in $GL_{|W_{l,\omega}|}(\hat{S}_\alpha)$ from the right, the matrix $\sH_1$ can be represented by 
$$\alpha^{-1}\cdot \diag\big(\begin{pmatrix} -1 & 0 \\ 1 & \alpha \end{pmatrix},\cdots, \begin{pmatrix} -1 & 0 \\ 1 & \alpha \end{pmatrix}\big).$$ 
Therefore, the composition 
$$\big(\bigoplus_{s_\alpha x\bullet 0>x\bullet 0} h_{x}^\alpha \big)\circ (F_1 \otimes_{\hat{S}} \hat{S}_\alpha) \circ \sH_1$$
can be represented by an invertible matrix. 
It shows that 
$$\Upsilon_{\hat{S}_\p,1}:= (F_1 \otimes_{\hat{S}} \hat{S}_\alpha) \circ \sH_1$$ 
has the desired property. 
\end{proof}

\begin{proof}[Proof of Proposition \ref{prop C.4}(2)] 
Without loss of generality, we may assume $z=1$. 
Then by the hypothesis that $s_\alpha\bullet \omega>\omega$, we have $s_\alpha x\bullet 0>x\bullet 0$ for any $x\in W_{l,\omega}$. 
By the canonical decompositions (\ref{equ 1-1.5}) and (\ref{equ 1-1.5.1}), the adjunction for the pair $(\sT,\sT')$ induces an adjunction $(\sT_{x},\sT'_{x})$ for any $x\in W_{l,\omega}$. 
Moreover the restriction of $(\sT_{x},\sT'_{x})$ induces an equivalence of categories 
$$\sT_{x}: \sO^{\{x\bullet 0, s_\alpha x\bullet 0\}}_{\hat{S}_\alpha} 
\xs \sO^{\{\omega, s_\alpha \bullet \omega\}}_{\hat{S}_\alpha}, \quad 
\sT'_{x}: \sO^{\{\omega, s_\alpha \bullet \omega\}}_{\hat{S}_\alpha}
\xs \sO^{\{x\bullet 0, s_\alpha x\bullet 0\}}_{\hat{S}_\alpha}.$$ 
Denote the counit by $\epsilon_{x}^{\alpha}:\sT_{x}\circ \sT'_{x}\xs \id$. 
Then by the property of adjunction, we have 
$$F_y \otimes_{\hat{S}} \hat{S}_\alpha= \bigoplus_{x\in W_{l,\omega}} (\epsilon^{\alpha}_{x,M(y\bullet \omega)_{\hat{S}_\alpha}})^{-1}, \quad y=1,s_\alpha.$$ 
Hence there is a commutative diagram 
$$\begin{tikzcd}0\arrow[r]& 
M(s_\alpha\bullet \omega)_{\hat{S}_\alpha}^{\oplus |W_{l,\omega}|} \arrow[r] \arrow[d,"F_{s_\alpha} \otimes_{\hat{S}} \hat{S}_\alpha","\simeq"'] 
& Q(\omega)_{\hat{S}_\alpha}^{\oplus |W_{l,\omega}|} \arrow[r] 
\arrow[d,"\bigoplus\limits_{x\in W_{l,\omega}} (\epsilon^{\alpha}_{x})^{-1}","\simeq"'] 
& M(\omega)_{\hat{S}_\alpha}^{\oplus |W_{l,\omega}|} \arrow[r] \arrow[d,"F_1 \otimes_{\hat{S}} \hat{S}_\alpha","\simeq"']& 0 \\ 
0\arrow[r]& 
\sT\sT'M(s_\alpha\bullet \omega)_{\hat{S}_\alpha} \arrow[r] 
& \sT\sT'Q(\omega)_{\hat{S}_\alpha} \arrow[r] 
& \sT\sT'M(\omega)_{\hat{S}_\alpha} \arrow[r] & 0.
\end{tikzcd}$$ 
Therefore, to find the desired isomorphism $\Upsilon_{\alpha,1}$, it is enough to show that there is a morphism $\theta$ fitting into the commutative diagram 
$$\begin{tikzcd}0\arrow[r]& 
M(s_\alpha\bullet \omega)_{\hat{S}_\alpha}^{\oplus |W_{l,\omega}|} \arrow[r] \arrow[d,"\sH_{s_\alpha}","\simeq"'] 
& Q(\omega)_{\hat{S}_\alpha}^{\oplus |W_{l,\omega}|} \arrow[r] 
\arrow[d,dashrightarrow,"\theta"] 
& M(\omega)_{\hat{S}_\alpha}^{\oplus |W_{l,\omega}|} \arrow[r] \arrow[d,"\sH_1","\simeq"']& 0 \\ 
0\arrow[r]& 
M(s_\alpha\bullet \omega)_{\hat{S}_\alpha}^{\oplus |W_{l,\omega}|} \arrow[r] 
& Q(\omega)_{\hat{S}_\alpha}^{\oplus |W_{l,\omega}|} \arrow[r] 
& M(\omega)_{\hat{S}_\alpha}^{\oplus |W_{l,\omega}|} \arrow[r] & 0, 
\end{tikzcd}$$ 
then we can set $\Upsilon_{\alpha,1}=\bigoplus\limits_{x\in W_{l,\omega}} (\epsilon^{\alpha}_{x})^{-1}\circ \theta$. 
We claim that 
$$\sH_{s_\alpha}-\sH_1 \equiv 0 \quad \text{mod }\alpha \hat{S}_\alpha,$$ 
then such $\theta$ exists by Lemma \ref{lem C.5}. 
Indeed, for the $(x,x')$-entry we have 
\begin{align*}
s_\alpha x(\bL^{-1}_\omega)\cdot \xi^{x'}&(s_\alpha x)- x(\bL^{-1}_\omega)\cdot \xi^{x'}(x)\\ 
&=\big(s_\alpha x(\bL^{-1}_\omega)-x(\bL^{-1}_\omega)\big)\cdot \xi^{x'}(s_\alpha x)+ 
x(\bL^{-1}_\omega)\cdot \big(\xi^{x'}(s_\alpha x)-\xi^{x'}(x)\big)\in \alpha \hat{S}_\alpha. \qedhere 
\end{align*}
\end{proof} 

\begin{proof}[Proof of Proposition \ref{prop 5.0} for $(\omega_1,\omega_2)=(0,\omega)$] 
Consider the base change functor $-\otimes_{\hat{S}}R: \sO_{\hat{S}}\rightarrow \sO_{R}$. 
By the equivalence (\ref{equ 2.10}), $\Upsilon_{\hat{S}}$ induces a natural isomorphism 
$$\Upsilon_{R}:\ \id^{\oplus |W_{l,\omega}|}\xs \sT\sT'$$ 
of functors on $\sP^{A,\leq\nu}_R$ for each $\nu$, where $A=U^\hb_q$. 
Since $\sP^{A,\leq\nu}_R$ consists of projective modules in $\sO^{\leq\nu}_{R}$ and the functors $\id^{\oplus |W_{l,\omega}|}$, $\sT\sT'$ are exact, it extends to a natural isomorphism $\Upsilon_{R}$ on $\sO^{\leq \nu}_{R}$ for each $\nu$, and hence on $\sO_{R}$. 
\end{proof} 

\subsection{An enhancement}\label{subsect B.4} 
Base on the construction of $\Upsilon$, we have an enhancement of Proposition \ref{prop 5.0} as follows. 
Recall that we have algebra homomorphisms 
$$H^\bullet_{\check{T}\times \Gm}(\Fl^{\omega,\circ})\rightarrow Z(\sO^\omega_{\hat{S}}), \quad \forall \omega\in \Xi_\sc.$$ 
They extend to the homomorphisms 
$$H^\bullet_{\check{T}\times \Gm}(\Fl^{\omega,\circ})\otimes_{\hat{S}'} R \rightarrow Z(\sO^\omega_{R}), \quad \forall \omega\in \Xi_\sc,$$ 
for any commutative Noetherian $\hat{S}$-algebra $R$. 
In particular, any module in $\sO^\omega_{R}$ is naturally a module of $H^\bullet_{\check{T}\times \Gm}(\Fl^{\omega,\circ})\otimes_{\hat{S}'} R$. 
It yields natural functors 
$$\varrho_{\omega}: \sO^\omega_{R}\rightarrow 
\sO^\omega_{H^\bullet_{\check{T}\times \Gm}(\Fl^{\omega,\circ})\otimes_{\hat{S}'} R}, 
\quad \forall \omega\in \Xi_\sc,$$ 
which are identities after composing with the forgetful functors to $\sO^\omega_{R}$. 

\begin{prop}\label{prop B.8} 
Let $R$ be a commutative Noetherian $\hat{S}$-algebra. 
For any $\omega_1,\omega_2\in \Xi_\sc$ such that $W_{l,\omega_1}\subseteq W_{l,\omega_2}$, there is a natural isomorphism 
$$\tilde{\Upsilon}_{\omega_2}^{\omega_1}:\ 
H^\bullet_{\check{T}\times \Gm}(\Fl^{\omega_1,\circ})
\otimes_{H^\bullet_{\check{T}\times \Gm}(\Fl^{\omega_2,\circ})} \id 
\xs T_{\omega_1}^{\omega_2}\circ \varrho_{\omega_1}\circ T^{\omega_1}_{\omega_2}$$ 
of functors from $\sO^{\omega_2}_{R}$ to $\sO^{\omega_2}_{H^\bullet_{\check{T}\times \Gm}(\Fl^{\omega_1,\circ})\otimes_{\hat{S}'} R}$. 
\end{prop} 
\begin{proof} 
For simplicity we only consider the case when $(\omega_1,\omega_2)=(0,\omega)$. 
%We abbreviate $\Fl^{\circ}=\Fl^{0,\circ}$. 
Recall by Proposition \ref{prop B.2}, we identify 
$$H^\bullet_{\check{T}\times \Gm}(\Fl^{\omega,\circ})={\GKM}_f(W_{l,\af},W^{\omega}_{l,\af})_{\hat{S}'}.$$ 
Then $H^\bullet_{\check{T}\times \Gm}(\Fl^{0,\circ})$ is free as a $H^\bullet_{\check{T}\times \Gm}(\Fl^{\omega,\circ})$-module, with a basis given by $\{\xi^x\}_{x\in W_{l,\omega}}$. 
Consider an isomorphism for any module $M$ in $\sO^{\omega_2}_{R}$, 
\begin{equation}\label{equ D.8}
\tilde{\Upsilon}_{R,M}:\ H^\bullet_{\check{T}\times \Gm}(\Fl^{0,\circ})
\otimes_{H^\bullet_{\check{T}\times \Gm}(\Fl^{\omega,\circ})} M=
\bigoplus_{x\in W_{l,\omega}} \xi^x\otimes M \xrightarrow[\sim]{\Upsilon_{R,M}} \sT\sT' M
\end{equation} 
where $\Upsilon_R$ is the natural isomorphism in Proposition \ref{prop C.3} base changed to $R$. 
It remains to show that (\ref{equ D.8}) is compatible with the action of $H^\bullet_{\check{T}\times \Gm}(\Fl^{0,\circ})$ on both sides. 
This compatibility for general $R$ follows from the compatibility for $\hat{S}$ by base change, which then follows from the compatibility for $\hat{\K}$ by the inclusion $\hat{S}\subset \hat{\K}$. 
From now on, we assume $R=\hat{\K}$ and $M=M(z\bullet \omega)_{\hat{\K}}$ for some $z\in W_{l,\af}$. 

Choose any element $\psi\in H^\bullet_{\check{T}\times \Gm}(\Fl^{0,\circ})$. 
And set 
$$\psi \cdot \xi^x=\sum_{x,x'} \psi_{x,x'}\cdot \xi^{x'}, \quad \psi_{x,x'}\in H^\bullet_{\check{T}\times \Gm}(\Fl^{\omega,\circ}).$$ 
Then the action of $\psi$ on $\bigoplus\limits_{x\in W_{l,\omega}} \xi^x\otimes M(z\bullet \omega)_{\hat{\K}}$ is given by the matrix $\big(\psi_{x',x}(z)\big)_{x,x'}$. 
On the ohther hand, we have a commutative diagram 
$$\begin{tikzcd} 
\bigoplus\limits_{x\in W_{l,\omega}} M(z\bullet \omega)_{\hat{\K}} \arrow[r,"F_z"] \arrow[d,"\diag{\big(\psi(zx)\big)_{x}}"'] 
& \sT\sT'M(z\bullet \omega)_{\hat{\K}} \arrow[d,"\psi"] \\ 
\bigoplus\limits_{x\in W_{l,\omega}} M(z\bullet \omega)_{\hat{\K}} \arrow[r,"F_z"]  
& \sT\sT'M(z\bullet \omega)_{\hat{\K}} .
\end{tikzcd}$$ 
Hence by construction of $\Upsilon_{\hat{\K}}$, to show the compatibility of (\ref{equ D.8}) with $\psi$ is the equivalent to show the equality 
$$\sH_z \cdot \big(\psi_{x',x}(z)\big)_{x,x'}= \diag{\big(\psi(zx)\big)_{x}}\cdot \sH_z.$$ 
We compute that (below we abbreviate $C=\diag{\big(zx(\bL_\omega^{-1})\big)_{x}}$) 
\begin{align*} 
\sH_z \cdot \big(\psi_{x',x}(z)\big)_{x,x'}&= C\cdot \big(\xi^{x'}(zx)\big)_{x,x'}\cdot \big(\psi_{x',x}(z)\big)_{x,x'} \\ 
&= C\cdot \big[ (\sum_{x''} \xi^{x''}\cdot \psi_{x',x''})(zx) \big]_{x,x'} \\ 
&= C\cdot \big[ (\psi\cdot \xi^{x'})(zx) \big]_{x,x'} \\ 
&= \diag{\big(\psi(zx)\big)_{x}}\cdot \sH_z. 
\end{align*} 
It completes the proof. 
\end{proof}

From the proposition above, we obtain a natural isomorphism 
$$\tilde{\Upsilon}_{\omega_2}^{\omega_1}:\ 
H^\bullet_{\check{T}\times \Gm}(\Fl^{\omega_1,\circ})
\otimes_{H^\bullet_{\check{T}\times \Gm}(\Fl^{\omega_2,\circ})} \id 
\xs T_{\omega_1}^{\omega_2}\circ  T^{\omega_1}_{\omega_2}$$ 
of functors from $\sO^{\omega_2}_{R}$ to $\sO^{\omega_2}_{R}$, recovering Proposition \ref{prop 5.0}. 

\begin{prop} 
The composition 
$$H^\bullet_{\check{T}\times \Gm}(\Fl^{\omega_1,\circ}) 
\otimes_{H^\bullet_{\check{T}\times \Gm}(\Fl^{\omega_2,\circ})} \id\xs T_{\omega_1}^{\omega_2}\circ T^{\omega_1}_{\omega_2} \xrightarrow{\text{counit}} \id$$ 
is induced by the pushforward of cohomology rings
%$$H^\bullet_{\check{T}\times \Gm}(\Fl^{\omega_1,\circ})\rightarrow H^\bullet_{\check{T}\times \Gm}(\Fl^{\omega_2,\circ})$$ 
via the natural projection $\Fl^{\omega_1,\circ} \rightarrow \Fl^{\omega_2,\circ}$. 
\end{prop} 
\begin{proof}
For simplicity we only consider the case when $(\omega_1,\omega_2)=(0,\omega)$. 
Let $\epsilon: \sT \sT'\rightarrow \id$ be the counit map, and let $\pi: \Fl^{0,\circ} \rightarrow \Fl^{\omega,\circ}$ be the natural projection. 
We have to show the equality 
\begin{equation}\label{equ D.9}
\epsilon_M \circ \tilde{\Upsilon}_{R,M}= \pi_*\otimes \id_M
\end{equation} 
for any commutative Noetherian $\hat{S}$-algebra $R$ and $M\in \sO_R^{\omega}$. 
As explained in the proof of Proposition \ref{prop B.8}, it only needs to consider the case when $R=\hat{\K}$ and $M=M(z\bullet \omega)_{\hat{\K}}$ for $z\in W_{l,\af}$. 
By the construction of $\Upsilon$ and the property of adjunction, we have 
\begin{align*}
\epsilon_{M(z\bullet \omega)_{\hat{\K}}} \circ \tilde{\Upsilon}_{\hat{\K},M(z\bullet \omega)_{\hat{\K}}}
&=(1\ 1 \cdots\ 1)\cdot \sH_z \\ 
&=\big(\frac{\sum\limits_{x'\in W_{l,\omega}}(-1)^{l(x')} \xi^x({zx'})}
{z(\bL_\omega)}\big)_{x\in W_{l,\omega}} 
\end{align*} 
as homomorphisms from $\bigoplus\limits_{x\in W_{l,\omega}} M(z\bullet \omega)_{\hat{\K}}$ to $M(z\bullet \omega)_{\hat{\K}}$. 
On the other hand, the $\check{T}\times \Gm$-equivariant Euler class of the normal bundle of $zP^0/P^0$ in $zP^\omega/P^0$ is given by $z(\bL_\omega)$. 
We obtain a commutative diagram 
$$\begin{tikzcd} 
H^\bullet_{\check{T}\times \Gm}(\Fl^{0,\circ}) \arrow[d,"\pi_*"] \arrow[r,hook] 
& \Fun(W_{l,\af}, \hat{\K}) \arrow[d,"\pi'_*"]\\ 
H^\bullet_{\check{T}\times \Gm}(\Fl^{\omega,\circ}) \arrow[r,hook] 
& \Fun(W^{\omega}_{l,\af}, \hat{\K}), 
\end{tikzcd}$$
where the horizontal maps are induced by restriction on $\check{T}$-fixed points, and $\pi'_*$ is defined as 
$$\pi'_*(\psi)(z)=
\frac{\sum\limits_{x\in W_\omega}(-1)^{l(x)} \psi({zx})}{z(\bL_\omega)},\quad \forall \psi\in \Fun(W_{l,\af}, \hat{\K}),\ \forall z\in W^{\omega}_{l,\af}.$$ 
The equality (\ref{equ D.9}) follows. 
\end{proof}

%\nocite{*}
\bibliographystyle{plain}
\bibliography{MyBibtex1}

\end{document}